\documentclass[a4paper, 11pt]{amsart}
\pdfoutput=1
\usepackage{custom-template, shortcuts,graphicx}
\usepackage{float}
\usepackage{mathtools}
\usepackage[lined,boxed]{algorithm2e}
\IncMargin{1em}

\makeatletter
\def\imod#1{\allowbreak\mkern5mu{\operator@font mod}\,\,#1}
\makeatother

\def\makeop#1{\expandafter\def\csname#1\endcsname
  {\mathop{\rm #1}\nolimits}\ignorespaces}
\makeop{Hom}   \makeop{End}   \makeop{Aut}   \makeop{Isom}  \makeop{Pic}
\makeop{Gal}   \makeop{ord}   \makeop{Char}  \makeop{Div}   \makeop{Lie}
\makeop{PGL}   \makeop{Corr}  \makeop{PSL}   \makeop{sgn}   \makeop{Spf}
\makeop{Spec}  \makeop{Map}    \makeop{Nm}    \makeop{Fr}    \makeop{disc}
\makeop{Proj}  \makeop{supp}  \makeop{ker}   \makeop{im}    \makeop{dom}
\makeop{coker} \makeop{Stab}  \makeop{SO}    \makeop{SL}    \makeop{SL} \makeop{O}
\makeop{Cl}    \makeop{cond}  \makeop{Br}    \makeop{inv}   \makeop{rank}
\makeop{id}    \makeop{Fil}   \makeop{Frac}  \makeop{GL}    \makeop{SU}
\makeop{Trd}   \makeop{Sp}    \makeop{Tr}    \makeop{Trd}   \makeop{diag}
\makeop{Res}   \makeop{ind}   \makeop{depth} \makeop{Tr}    \makeop{st}
\makeop{Ad}    \makeop{Int}   \makeop{tr}    \makeop{Sym}   \makeop{can}
\makeop{length}\makeop{SO}    \makeop{torsion} \makeop{GSp} \makeop{Ker}
\makeop{Adm}   \makeop{Frob}  \makeop{id}    \makeop{Tor}   \makeop{Ind}
\makeop{CoInd} \makeop{Inf}   
\makeop{Ann} \makeop{rec} \makeop{Image} \makeop{Rep} \makeop{Inn} \makeop{Out}
\makeop{Def} \makeop{B} \makeop{gr}
\makeop{Art}
\makeop{GO} \makeop{HT} 
\makeop{Tran} \makeop{MW}

\allowdisplaybreaks

\newcommand{\bP}{\mathbb P}
\newcommand{\bC}{\mathbb C}
\newcommand{\bZ}{\mathbb Z}
\newcommand{\bQ}{\mathbb Q}
\newcommand{\bF}{\mathbb F}
\newcommand{\cF}{\mathcal F}

\newcommand{\cE}{\mathcal E}

\usepackage{float}
\restylefloat{table}
\makeatletter
\@namedef{subjclassname@2020}{%
	\textup{2020} Mathematics Subject Classification}
\makeatother

\begin{document}

\title{Five-dimensional compatible systems and the Tate conjecture for elliptic surfaces}

\author{Lian Duan}
\author{Xiyuan Wang}
\author{Ariel Weiss}

\address{Lian Duan, Institute of Mathematical Sciences, ShanghaiTech University, China.\vspace*{-5pt}}
\email{duanlian@shanghaitech.edu.cn\vspace*{-3pt}}
\address{Xiyuan Wang, Department of Mathematics, The Ohio State University, USA.\vspace*{-5pt}}
\email{wang.15476@osu.edu\vspace*{-3pt}}
\address{Ariel Weiss, Department of Mathematics, The Ohio State University, USA.\vspace*{-15pt}}
\address{Department of Mathematics, Trinity College, CT, USA}
\email{ariel.weiss@trincoll.edu\vspace*{-3pt}}

\date{}

\begin{abstract}
    Let $(\rho_\lambda\:G_\Q\to \GL_5(\elb))_\lambda$ be a strictly compatible system of Galois representations such that no Hodge--Tate weight has multiplicity $5$. Under mild assumptions, we show that if $\rho_{\lambda_0}$ is irreducible for some prime $\lambda_0$, then $\rho_\lambda$ is irreducible for all but finitely many primes $\lambda$. More generally, if $(\rho_\lambda)_\lambda$ is essentially self-dual, we show that either $\rho_\lambda$ is irreducible for all but finitely many $\lambda$, or the compatible system $(\rho_\lambda)_\lambda$ decomposes as a direct sum of lower-dimensional compatible systems.

    We apply our results to study the Tate conjecture for elliptic surfaces. For example, if $X_0\: y^2 + (t+3)xy + y= x^3$, we prove the codimension-one $\l$-adic Tate conjecture for all but finitely many $\l$, for all but finitely many general, degree $3$, genus $2$ branched multiplicative covers of $X_0$ with exactly nine geometric singular fibres. 
    
    To prove this result, we classify the elliptic surfaces into four one-dimensional families and two isolated classes. For each of the four families, we prove, using perverse sheaf theory and a result of Cadoret--Tamagawa \cites{Cadoret-Tamagawa-open-img-I, Cadoret-Tamagawa-open-img-II}, that if the relevant $5$-dimensional Galois representation is irreducible for
one surface in a family, then it is irreducible for all but finitely many
surfaces in that family. We then verify this irreducibility for one representative of each family by making our irreducibility result explicit: for the compatible system arising from the transcendental part of $H^2_{\et}(X_{\Qb}, \Ql(1))$ for a representative $X$, we formulate an algorithm that takes as input the characteristic polynomials of Frobenius, and terminates if and only if the compatible system is irreducible.
\end{abstract}
\maketitle
\setcounter{tocdepth}{1}
\tableofcontents

\section{Introduction}

Let $E$ be a number field, let $G_\Q = \Ga \Q$ be the absolute Galois group of $\Q$, and let
\[(\rho_\lambda\:G_\Q\to \GL_n(\elb))_\lambda\]
be a compatible system of semisimple $\lambda$-adic Galois representations. By the Brauer--Nesbitt and the Chebotarev density theorems, each $\rho_\lambda$ is completely determined by the characteristic polynomials of $\rho_\lambda(\Frob_p)$ for all but finitely many primes $p$, and by the definition of compatibility, these characteristic polynomials are independent of $\lambda$. Therefore, it is natural to ask whether the representation-theoretic properties of $\rho_\lambda$ are independent of $\lambda$. For example:

\begin{question}\label{quest:irred}
    Suppose that for some prime $\lambda_0$, $\rho_{\lambda_0}$ is irreducible. Is $\rho_\lambda$  irreducible for all primes $\lambda$?
\end{question}

While it is believed that the answer to \Cref{quest:irred} is yes, it is wide open in general. Nevertheless, there has been substantial recent progress towards answering it in low dimensions. In particular, a recent result of Hui \cite{hui-BLMS} answers this question when $n \le 4$, under the assumption that the compatible system is Hodge--Tate regular.

Our first result answers \Cref{quest:irred} when $n = 5$, without the assumption of Hodge--Tate regularity.

\begin{theorem}\label{thm:irreducible}
    Let $(\rho_\lambda\:G_{\Q}\to \GL_5(\elb))_\lambda$ be a strictly compatible system of Galois representations. Suppose that for some prime $\lambda_0$, $\rho_{\lambda_0}$ is irreducible. Suppose moreover that one of the following conditions holds:
    \begin{enumerate}
        \item Each Hodge--Tate weight has multiplicity $1$.
        \item Each Hodge--Tate weight has multiplicity at most $4$ and each $\rho_\lambda$ is isomorphic to a representation valued in $\GO_5(\elb)$.
        \item Each Hodge--Tate weight has multiplicity at most $4$ and $\rho_{\lambda_0}$ is not isomorphic to a representation valued in $\GO_5(\elb)$.
    \end{enumerate}
    Then $\rho_\lambda$ is irreducible for all but finitely many primes $\lambda$.
\end{theorem}

More generally, suppose that $\rho_{\lambda_0}\simeq\rho_1\+\cdots\+\rho_k$ decomposes as a direct sum of irreducible subrepresentations. Then, analogously to \Cref{quest:irred}, one expects that there should be a corresponding decomposition of $(\rho_\lambda)_\lambda$ as a sum of compatible systems. Our second result proves that this is indeed the case when each $\rho_\lambda$ is essentially self-dual.

\begin{theorem}\label{thm:decomp}
    Let $(\rho_\lambda\:G_{\Q}\to \GL_5(\elb))_\lambda$ be a strictly compatible system of Galois representations. Suppose that each $\rho_\lambda$ is isomorphic to a representation valued in $\GO_5(\elb)$ and that one of the following conditions holds:
    \begin{enumerate}
        \item Each Hodge--Tate weight has multiplicity $1$.
        \item Each Hodge--Tate weight has multiplicity at most $2$, and either $(\rho_\lambda)_\lambda$ is pure or $\Tr(\rho_\lambda(c)) = \pm 1$ for some complex conjugation $c\in G_\Q$.
        \item Each Hodge--Tate weight has multiplicity at most $4$, $(\rho_\lambda)_\lambda$ is pure, and $\Tr(\rho_\lambda(c)) = \pm 1$ for some complex conjugation $c\in G_\Q$.
    \end{enumerate}
    Then either $\rho_\lambda$ is irreducible for all but finitely many primes $\lambda$, or for all primes $\lambda$, we have a decomposition
    \[\rho_\lambda = \rho_{\lambda, 1}\+ \rho_{\lambda, 2}\+\cdots\+\rho_{\lambda, k}\]
    where each $(\rho_{\lambda, i})_\lambda$ is a strictly compatible system of Galois representations that is irreducible for all but finitely many primes $\lambda$.
\end{theorem}

\Cref{quest:irred} was first considered by Serre\footnote{See, for example, \cite{serre-abelian}*{I-12}, where he poses the stronger question of whether there is a Lie algebra $\g$ over $E$ such that $\Lie(\im(\rho_\lambda)) = \g\tensor_EE_\lambda$ for all $\lambda$.} in his study of the images of Galois representations attached to elliptic curves. The case that $n=2$ was essentially settled by Ribet \cites{Ribet75, Ribet77} in his proof that the Galois representations attached to modular forms are irreducible, using the following argument: 
\begin{enumerate}
    \item If some $\rho_{\lambda_0}$ is reducible, then we can write $\rho_{\lambda_0}\simeq\chi_1\+\chi_2$, where $\chi_1$ and $\chi_2$ are one-dimensional representations.
    \item By class field theory, there are compatible systems of Galois representations $(\chi_{1, \lambda})_\lambda$ and $(\chi_{2, \lambda})_\lambda$, such that $\chi_1 = \chi_{1, \lambda_0}$ and $\chi_2 = \chi_{2, \lambda_0}$.
    \item By the Chebotarev density theorem and the Brauer--Nesbitt theorem (see \Cref{prop:brauer-nesbitt}), it follows that $\rho_{\lambda}\simeq \chi_{1, \lambda}\+\chi_{2, \lambda}$ for all primes $\lambda$.
\end{enumerate}
This strategy of using automorphy theorems (e.g.\ class field theory) to prove that every proper subrepresentation of $\rho_{\lambda_0}$ extends to a compatible system has been used in all previous results answering cases of \Cref{quest:irred} \cites{dieulefait-compatible,Dieulefait_Vila-2,Dieulefait_Vila, Duan-Wang-Tate-conj-2020, hui-BLMS, dai1}, as well as the closely related problem of proving that the Galois representations attached to cuspidal automorphic representations of $\GL_n$ are irreducible \cites{blasius-rogawski, taylorimaginary, Dieulefait2002maximalimages, dimitrov, Dieulefait-endoscopy,Ramakrishnan, BLGGT,patrikis-taylor, Xia, weissthesis,DZ,weiss2018image, Hui, bockle2024irreducibility,fengwhitmore,huilee,dai2}. We refer to \cite{hui-BLMS}*{Sec.~1.2} for a summary of these results.

Applying this strategy to prove \Cref{thm:irreducible,thm:decomp} is complicated by the fact that we do not assume that our Galois representations are Hodge--Tate regular. For example, in the setting of \Cref{thm:decomp}, consider the following two cases:
\begin{enumerate}
    \item that for some prime $\lambda_0$, $\rho_{\lambda_0}\simeq \sigma_2 \oplus\sigma_3$, where $\sigma_2$ is an irreducible two-dimensional representation with Hodge--Tate weights $\{-a, a\}$ for some non-zero $a\iZ$, and $\sigma_3$ is an irreducible three-dimensional representation with Hodge--Tate weights $\{0,0,0\}$.
    \item that for some prime $\lambda_0$,  $\rho_{\lambda_0}\simeq\sigma_2\+\sigma_2'\+\chi$, where $\chi$ is a quadratic character, $\sigma_2$ is an irreducible, two-dimensional self-dual representation with Hodge--Tate weights $\{-a,a\}$ and $\sigma'_2$ is an irreducible, two-dimensional self-dual \emph{even} Galois representation, with Hodge--Tate weights $\{0,0\}$.
\end{enumerate}

In case $(i)$, no known automorphy theorems apply directly to $\sigma_3$. However, by hypothesis, $\sigma_3$ becomes self-dual after twisting by a finite order character. Combining the exceptional isomorphism $\PGL_2 \cong \SO_3$ with a lifting result for crystalline representations valued in $\PGL_2$ to those valued in $\GL_2$ (\Cref{thm:lifting}), we deduce that $\sigma_3$ is the symmetric square of an irreducible two-dimensional representation $r'$ with Hodge--Tate weights $\{0,0\}$. Using our assumption that $\Tr(\rho_\lambda(c)) = \pm 1$, we show that $r'$ is odd. Let $\l_0$ be the rational prime dividing $\lambda_0$. If we also knew that the residual representation $\overline{r'}|_{\Q(\zeta_{\l_0})}$ were irreducible, we could invoke a theorem of Pilloni--Stroh \cite{Pilloni-Stroh} to conclude that $r'$ is modular, and hence that $\sigma_3$ is an Artin representation, and therefore lies in a compatible system. In \Cref{lem:3d-irregular}, we establish this residual irreducibility using a recent result of Hui \cite{Hui} (\Cref{thm:hui}), which shows that if case $(i)$ occurs infinitely often and $\l_0$ is sufficiently large, then $\overline{r'}|_{\Q(\zeta_{\l_0})}$ is indeed irreducible. This argument is carried out in \Cref{lem:3+2}.

In case $(ii)$, there are no automorphy theorems that apply to the representation $\sigma_2'$ at all: the even representation $\sigma_2'$ is expected to correspond to a Maass form, but proving this is completely open. 
Instead, we study the representation $\rho_{\lambda_0}$ via the corresponding representation $r_{\lambda_0}\:G_\Q\to \Gf(\overline{E}_{\lambda_0})$ obtained by lifting $\rho_{\lambda_0}$ through the exceptional isomorphism $\SO_5\cong\PGSp_4$. First, we classify the possible decompositions of a five-dimensional self-dual representation $\rho\:G_\Q\to \SO_5(\elb)$ in terms of the possible decompositions of its corresponding representation $r\:G_{\Q}\to \Gf(\elb)$
(\Cref{prop:rl-red,prop:rl-irred} and \Cref{table-cases}). In this case, the representation $r_{\lambda_0}$ is induced from a two-dimensional representation of $G_K$, where $K$ is the field cut out by $\chi$. Using a generalisation of a result of Calegari \cite{Calegari-even1}, we show that if $\lambda_0$ is large enough, then $\sigma_2$ must be odd (\Cref{prop:odd}). Since $\sigma_2$ is also self-dual, we have $\sigma_2\simeq\sigma_2\dual\simeq\sigma_2\tensor\det\sigma_2$, so it must be an induced representation. Combining these results with a classification of the possible algebraic monodromy groups of $\rho_\lambda$, we show that $\sigma_2'$ must be an Artin representation. As a result, we show that $\sigma_2'$ does indeed extend to a compatible system of Galois representations, even though we have no way to show that this compatible system is automorphic. This argument is carried out in \Cref{lem:5ii}.

To prove \Cref{thm:irreducible,thm:decomp}, we combine Ribet's strategy (\Cref{prop:brauer-nesbitt}), Hui's recent work \cite{Hui} (\Cref{thm:hui}), and automorphy theorems \cites{BLGGT, Pilloni-Stroh} to produce a series of results that guarantee that certain subrepresentations of $\rho_\lambda$ are contained in compatible systems (see especially \Cref{prop:odd,lem:2d-3d-compatible,lem:2d-irregular,lem:3d-irregular}). We then combine these results with ``independence of $\l$'' theorems \cites{larsen-pink, hui-mrl} (\Cref{thm:semisimple-rank}),  Galois lifting theorems \cite{patrikis-variations} (\Cref{thm:lifting}), a careful classification of the possible decompositions of an arbitrary representation $\rho\:G_{\Q}\to\GO_5(\elb)$ (\Cref{prop:rl-red,prop:rl-irred} and \Cref{table-cases}), and a delicate analysis of each of the possible decompositions of $\rho_\lambda$ (\Cref{sec:decomp-analysis}).

If we assume the stronger hypothesis that $(\rho_\lambda)_\lambda$ is Hodge--Tate regular, then our results hold for compatible systems of $G_K$, for any totally real field $K$, under a mild symmetry hypothesis.

\begin{theorem}
    Let $K$ be a totally real field and let $(\rho_\lambda\:G_{K}\to \GL_5(\elb))_\lambda$ be a strictly compatible system of Galois representations that is Hodge--Tate regular. Assume that for each $\lambda$, $\rho_\lambda$ satisfies the symmetry hypothesis of \cite{patrikis-sign}*{Prop.~5.5}.
    Suppose that for some prime $\lambda_0$, $\rho_{\lambda_0}$ is irreducible. Then $\rho_\lambda$ is irreducible for all but finitely many primes $\lambda$.

     Suppose moreover that each $\rho_\lambda$ is isomorphic to a representation valued in $\GO_5(\elb)$.
    Then either $\rho_\lambda$ is irreducible for all but finitely many primes $\lambda$, or for all primes $\lambda$, we have a decomposition
    \[\rho_\lambda = \rho_{\lambda, 1}\+ \rho_{\lambda, 2}\+\cdots\+\rho_{\lambda, k}\]
    where each $(\rho_{\lambda, i})_\lambda$ is a strictly compatible system of Galois representations that is irreducible for all but finitely many primes $\lambda$.
\end{theorem}

The proofs go through verbatim with this change. In the Hodge--Tate irregular case, the assumption that $K=\Q$ is required to rule out situations such as $\rho_\lambda$ having a subrepresentation $\sigma$ with Hodge--Tate weights $\{0,0\}$ at one embedding and $\{-a,a\}$ at another with $a\iZ\setminus\{0\}$: in this case, there are no known automorphy theorems that apply to $\sigma$. However, these pathologies do not occur if $\rho_\lambda$ is Hodge--Tate regular, and the automorphy theorems of \cite{BLGGT} hold over any totally real field. See \Cref{rem:totally real} for further details. Since this paper focuses mainly on the Hodge--Tate irregular case, for notational simplicity, we assume throughout that $K=\Q$.

We also note that our results hold with the weaker condition that $(\rho_\lambda)_\lambda$ is weakly compatible, but with a weaker conclusion that the system is irreducible for all primes $\lambda\mid \l$ for all $\l$ in a set of rational primes of Dirichlet density $1$. See \Cref{rem:psw} for further discussion.

\subsection{The Tate conjecture}

Let $X$ be a smooth projective variety over a number field $K$. Then for each integer $d\ge 0$ and for all primes $\l$, the twisted \'etale cohomology $H_{\mathrm{\acute{e}t}}^{2d}(X_{\overline{K}}, \bQ_\ell(d))$ admits a natural action of the absolute Galois group $G_K$. The invariant subspace $H_{\mathrm{\acute{e}t}}^{2d}(X_{\overline{K}}, \bQ_\ell(d))^{G_K}$ contains the subspace spanned by the image of the Chern map from the codimension $d$-cycles of $X$ defined over $K$. Tate conjectured that the span of the image of the Chern map is exactly $H_{\mathrm{\acute{e}t}}^{2d}(X_{\overline{K}}, \bQ_\ell(d))^{G_K}$.

\begin{conjecture}[The $\l$-adic Tate conjecture]
    The $\bQ_\ell$-linear subspace $H_{\mathrm{\acute{e}t}}^{2d}(X_{\overline{K}}, \bQ_\ell(d))^{G_K}$ is spanned by the classes of codimension-$d$ subvarieties of $X$. 
\end{conjecture}

Specialising to the case that $d = 1$, let $\NS(X)$ denote the N\'eron--Severi group of $X$, i.e.\ the finitely generated abelian group of $K$-divisors modulo algebraic equivalence. Let $\NS(X_{\overline K})$ denote the N\'eron--Severi group of $X_{\overline K}$. Then the first Chern map
\[
C^1\: \NS(X_{\overline{K}})\otimes \Q_\ell \to H^2_{\rm{\acute{e}t}}(X_{\overline{K}}, \Q_{\l}(1))
\]
induces a map
\[
C^1\: \NS(X)\otimes \Q_\ell \to H^2_{\rm{\acute{e}t}}(X_{\overline{K}}, \Q_{\l}(1))^{G_K}.
\]

\begin{conjecture}[The codimension-one $\l$-adic Tate conjecture]\label{conj:codim one}
The map
\begin{equation}\label{eqn-cycle-map}
    C^1\:\NS(X)\otimes\Q_{\l} \longrightarrow H^2_{\rm{\acute{e}t}}(X_{\overline{K}}, \Q_{\l}(1))^{G_K}
\end{equation}
is an isomorphism.
\end{conjecture}

The first case of the codimension-one Tate conjecture to be proven was by Tate himself, for abelian varieties over finite fields \cite{Tate-Tate-Conj-AV}. Subsequently, the same conjecture has been proven for abelian varieties over number fields \cite{Falting-FS}, for most K3 surfaces over number fields \cites{Nygaard-Ogus-Tate-Conj-K3, Maulik-Supersingular-K3, Moonen-Tate-Conj}, and, with the help of the Kuga--Satake construction, for all varieties with $h^{2, 0} = 1$, under a mild moduli condition \cite{Moonen-Tate-Conj} (see also \cite{Andre-Tate-Conj-K3}). Over finite fields, the Tate conjecture for elliptic surfaces is equivalent to the Birch--Swinnerton-Dyer conjecture for the generic fibre. In this direction,  Hamacher--Yang--Zhao recently proved the Tate conjecture for a class of genus one elliptic surfaces over finite fields \cite{Hamacher-Yang-Zhao-finiteness-Sha}. A generalisation of the connection between the Tate conjecture and BSD can be found in Qin's article \cite{Qin-comparison-Tate-conj}. 
We refer the reader to \cite{Totaro-Tate-conj-survey} and \cite{Li-Zhang-note-Tate-conj} for more results, as well as more information about the Tate conjecture and its variants.

\subsubsection{The Tate conjecture for elliptic surfaces}

Let $f\: X\to X'$ be a finite surjective morphism of smooth projective varieties over $\Q$. If the Tate conjecture is true for $X$, then the Tate conjecture is true for $X'$,\footnote{In fact, the composition $f_* \circ f^*$ acts as multiplication by $\deg(f)$ on $H^{2}_{\et}(X'_{\overline{\mathbb{Q}}}, \mathbb{Q}_\l(1))$ and the pullback map $f^*$ embeds the cohomology of $X'$ into that of $X$ as a Galois module.} however, the converse statement is highly non-trivial. 

Suppose that $X'$ is an elliptic surface, and that $X$ is a general, degree $3$, genus $2$ branched multiplicative cover of $X'$. We focus on the example of $X' = X_0$, where $X_0$ is the No.~63 elliptic surface of \cite{Shioda-Schutt-MW-Lattice}*{Table~8.3} induced by the Weierstrass equation
\begin{equation}
    X_0\: y^2 + (t+3)xy + y = x^3,
\end{equation}
however, our methods can be applied to covers of many other elliptic surfaces.    

The surface $X_0$ is parametrised by $\P^1_t$ over $\bQ$. Let $\SS$ be the set of genus $2$ elliptic surfaces $X$ such that
\begin{enumerate}[leftmargin=*]
    \item $X$ is equipped with a commutative diagram 
\begin{equation}\label{Eqn: Cartesian_diagram}
    \begin{tikzcd}
	X= \bP^1_s\times_{\bP_{t}^{1}} X_0 \arrow[r] \arrow[d] & X_0 \arrow[d] \\
	\bP^1_s  \arrow[r, "\varphi"] & \bP^1_t,
\end{tikzcd}
\end{equation}
with $\varphi\: \bP^1_s\to \bP^1_t$ a degree $3$ morphism defined over $\bQ$. 
\item All the singular fibres of $X$ are of multiplicative type.
\item If $e_s$ denotes the ramification index of $\varphi$ at $s$, then
\[
\sum_{\substack{s\in \mathbb P^1\\ X_{0,\varphi(s)}\text{ is non-smooth}}}(e_s-1)=3.
\]
One can see more details about this in \Cref{Sect: 63_curve}, in particular \Cref{Tab: ramification_type}.
\end{enumerate}

Our third result is the following theorem:

\begin{theorem}\label{thm:Tate-intro}
    For all but finitely many primes $\l$, the codimension-one $\l$-adic Tate conjecture is true for all $X\in \SS\setminus \SS'_\l$ for some finite subset $\SS'_\l\sub \SS$.
\end{theorem}

To prove \Cref{thm:Tate-intro}, we first observe that the Tate conjecture is closely related to the irreducibility of Galois representations.
For a smooth projective variety $X/\Q$, and for each prime $\l$, we define the transcendental part $\Tran_\l(X)$ of $H^2_\et(X_{\Qb}, \Ql(1))$ to be the semisimplification of the quotient $H^2_{\rm{\acute{e}t}}(X_{\overline{\Q}}, \Q_{\l}(1))/(\NS(X_{\Qb})\otimes\Q_{\l})$. Thus, up to semisimplification, we have an isomorphism of Galois representations
\[H^2_\et(X_{\overline \Q}, \Ql(1))^{ss}\cong \Tran_\l(X)^{ss}\+(\NS(X_{\overline \Q})\tensor\Ql),\]
and the natural map $\NS(X)\otimes\Q_{\l} \rightarrow (\NS(X_{\overline{\Q}})\otimes \Q_{\l})^{G_\Q}$ is an isomorphism (\cite{Fulton-int-theory}*{19.3.1}, \cite{Totaro-Tate-conj-survey}*{p.~580}). Hence, \Cref{conj:codim one} is equivalent to saying that $\Tran_\l(X)$ does not have a subrepresentation isomorphic to the trivial representation, or equivalently, that $(\Tran_\l(X)^{ss})^{G_\Q}$ is trivial.\footnote{Assuming the Tate conjecture, the Galois action is semisimple \cite{Moonen-remark-Tate-conj}, meaning $\mathrm{Tran}_\ell(X) \cong \mathrm{Tran}_\ell(X)^{ss}$ as Galois modules, and hence their $G_{\mathbb{Q}}$-invariant subspaces coincide.} Notice that as complex vector spaces, $H^2(X_{\overline{\bQ}}, \C)/(\NS(X_{\overline{\bQ}})\otimes \bC)$ contains $H^{2,0}\oplus H^{0,2}$. Hence by the classical comparison theorem, when the genus of the surface $X$ is nonzero, the transcendental part $\Tran_\ell(X)$ has dimension $\geq 2$. In this case, if $\Tran_\l(X)$ is \emph{irreducible}, then the $\l$-adic Tate conjecture is true for $X$. This observation was used by the first two authors to prove the Tate conjecture for certain families of genus three elliptic surfaces, whose transcendental Galois representations decompose as a direct sum of three $3$-dimensional representations, each with regular Hodge--Tate weights $\{-1,0,1\}$ \cite{Duan-Wang-Tate-conj-2020}. 

Now assume that $X$ is an elliptic surface contained in the set $\SS$. Our proof of \Cref{thm:Tate-intro} requires multiple steps. 

First, in \Cref{Sect: elliptic_surfaces,Sect: concrete_eg_5dim}, we construct a compatible system of Galois representations $(\operatorname{NTriv}_\l(X))_\l$ such that each $\operatorname{NTriv}_\l(X)$ contains $\Tran_\l(X)$ as a subrepresentation. In particular, if $\operatorname{NTriv}_\l(X)$ is irreducible, then it must be equal to $\Tran_\l(X)$. The compatible system $(\operatorname{NTriv}_\l(X))_\l$ is $5$-dimensional  with irregular Hodge--Tate weights $\{-1,-1,0,1,1\}$. In particular, by \Cref{thm:irreducible}, if we can prove that $\operatorname{NTriv}_\l(X)$ is irreducible for a single prime $\l$, then it is irreducible for all but finitely many primes $\l$. Consequently, if $\operatorname{NTriv}_\l(X)$ is irreducible for a single prime $\l$, then the Tate conjecture holds for all but finitely many $\ell$.

Second, we classify all the elliptic surfaces $X\in \SS$. In \Cref{Sect: concrete_eg_5dim}, we prove that over $\Qb$, the elliptic surfaces $X\in\SS$ are classified according to their ramification types, which we detail in \Cref{Tab: ramification_type}. Consequently, up to $\Q$-isomorphism, each surface $X\in\SS$ belongs to one of four one-dimensional families or to one of two isolated classes.

Third, in \Cref{Sect: spread_out_irr}, given a family of elliptic surfaces $\pi\:\X\to B$, we construct a lisse sheaf whose stalk at each point $b\in B(\Q)$ is the non-trivial part $\operatorname{NTriv}_\ell(\mathcal X_b)$ of the specialisation $\X_b$. After verifying irreducibility, this non-trivial part coincides with
$\Tran_\ell(\mathcal X_b)$. Our construction requires a careful study of the second cohomology group of an elliptic surface using perverse sheaf theory. Upon constructing this lisse sheaf, using the work of Cadoret and Tamagawa \cites{Cadoret-Tamagawa-open-img-I, Cadoret-Tamagawa-open-img-II}, we prove in \Cref{Prop: spread_out} that if the transcendental $\ell$-adic representation $\operatorname{NTriv}_\l(\mathcal{X}_{b_0})$ is irreducible for some specific specialisation $b_0 \in B(\mathbb{Q})$, then $\operatorname{NTriv}_\l(\mathcal{X}_b)$ is irreducible for all but a finite number of specialisations $b \in B(\mathbb{Q})$. 

Hence, to prove \Cref{thm:Tate-intro}, it suffices to demonstrate that for each of the four one-dimensional families $\pi\: \mathcal{X} \to B$ discussed in \Cref{Sect: concrete_eg_5dim}, there exists at least one specialisation $b \in B(\Q)$ for which the transcendental $\ell$-adic representation $\operatorname{NTriv}_\ell(\mathcal{X}_b)$ is irreducible. We prove this assertion by making \Cref{thm:decomp} explicit. Let $(\rho_\l)_\l$ be a five-dimensional $\Q$-rational weakly compatible system of Galois representations, with Hodge--Tate weights $\{-1,-1,0,1,1\}$. In \Cref{Sect: alg-for-irreducibility}, we develop an algorithm, which takes as input the characteristic polynomials of $\rho_\l(\Frob_p)$ as well as the set of ramified primes, and which terminates if and only if $\rho_\l$ is irreducible for some $\l$, and hence for all but finitely many $\l$.

Finally, in \Cref{Sect: verify_Tate_conj_concrete_family} and Appendix~\ref{Appendix: other_cases}, we implement our algorithm on well-chosen specialisations $\Tran_\l(\X_b)$ for each of the families $\pi\:\X\to B$. In practice, our algorithm terminates extremely quickly: in our examples, it is enough to determine the set of ramified primes and calculate the characteristic polynomials of $\rho_\l(\Frob_p)$ for $p < 20$, after which, we can execute our algorithm by hand!
Since, in each of the families, the algorithm terminates, the above steps combine to prove \Cref{thm:Tate-intro}.

\section{Compatible systems of Galois representations}\label{sec:compatible-systems}

\subsection{Preliminaries}

Let $K$ and $E$ be number fields and let
\[(\rho_\lambda\:G_{K}\to\GL_n(\elb))_\lambda\]
be a family of continuous semisimple $\lambda$-adic Galois representations, where $\lambda$ runs over the primes of $E$. We note that in this paper, whenever we talk about compatible systems of Galois representations, we always assume that the representations are semisimple. We fix once and for all embeddings $\overline E\hookrightarrow\elb$ for every prime $\lambda$.

\subsubsection{Compatible systems}

\begin{definition}
    We say that $(\rho_\lambda)_\lambda$ is a \emph{Serre compatible system of Galois representations} of $G_K$ with coefficients in $E$ if there exist:
    \begin{itemize}
    \item a finite set of primes $S$ of $K$;
    \item a degree $n$ monic polynomial $Q_v(X)\in E[X]$, for every prime $v\notin S$;
    \end{itemize}
    such that for each prime $\lambda$ of $E$, with residue characteristic $\l$, if $v\notin S$ and $v\nmid \l$, then $\rho_\lambda$ is unramified at $v$ and $\rho_\lambda(\Frob_v)$ has characteristic polynomial $Q_v(X)$.

\end{definition}

\begin{definition}\label{Defn: weakly_compatible_sys}
    We say that a Serre compatible system $(\rho_\lambda)_\lambda$ is \emph{weakly compatible} if for every embedding $\tau\:K\hookrightarrow\overline E$, there exists a size $n$ multiset of integers $H_\tau$, such that for each prime $\lambda$ of $E$, with residue characteristic $\l$ and for each prime $v\mid \l$:
    \begin{itemize}
        \item $\rho_\lambda|_{G_{K_v}}$ is de Rham;
        \item if $v\notin S$, then $\rho_\lambda|_{G_{K_v}}$ is crystalline;
        \item if $v$ corresponds to the embedding $\tau\:K\hookrightarrow\overline E$ via the identification $\Hom_{\Q}(K,\overline E) = \Hom_{\Q}(K,\elb) = \coprod_{v\mid \l}\Hom_{\Ql}(K_v, \elb)$, then $\rho_\lambda|_{G_{K_v}}$ has Hodge--Tate weights $H_\tau$.
    \end{itemize}
\end{definition}

\begin{definition}
    We say that a weakly compatible system $(\rho_\lambda)_\lambda$ is \emph{strictly compatible} if for every prime $v\nmid\l$,
    the Weil--Deligne representation $\mathrm{WD}(\rho_\lambda|_{G_{K_v}})^{F\text{-}ss}$ is independent of $\lambda$.
\end{definition}

\begin{definition}\label{def:ht-regular}
    We say that a weakly compatible system $(\rho_\lambda)_\lambda$ is \emph{Hodge--Tate regular} if for every embedding $\tau\:K\hookrightarrow\overline E$, the multiset $H_\tau$ consists of $n$ \emph{distinct} integers.
\end{definition}

\begin{definition}
    We say that a Serre compatible system $(\rho_\lambda)_\lambda$ is \emph{$E$-rational} if for every prime $\lambda$, $\rho_\lambda$ is isomorphic to a representation
    $G_K\to \GL_n(E_\lambda)$.
\end{definition}

 By \cite{hui-coefficients}*{Thm.~1.3}, there always exists a finite extension $E'/E$ such that $(\rho_\lambda)_\lambda$ is $E'$-rational. Thus, we may always assume, with no loss in generality, that our compatible systems are $E$-rational.

\subsubsection{Purity}

We now specialise to the case that $K=\Q$. In particular, there is just one embedding $\Q\hookrightarrow\overline E$, so we can refer to the Hodge--Tate weights of $(\rho_\lambda)_\lambda$ without ambiguity.

\begin{definition}
    We say that a Serre compatible system $(\rho_\lambda)_\lambda$ is \emph{pure} if there exists an integer $w$ such that for each prime $p\notin S$, each root $\alpha$ of $Q_p(X)$ in $\overline E$, and each embedding $\iota\:\overline E\hookrightarrow{\C}$, we have 
    \[|\iota(\alpha)|^2 = p^w.\]
    In this case, we say that $(\rho_\lambda)_\lambda$ is pure of weight $w$.
\end{definition}

\begin{lemma}\label{lem:pure-ht-weight}
    Assume that the Serre compatible system $(\rho_{\lambda})_\lambda$ is pure of weight $w$, and suppose that $\rho'$ is a $k$-dimensional subrepresentation of $\rho_\lambda$, for some $\lambda$. Then $\det\rho'$ has Hodge--Tate weight $\frac{kw}2$.
\end{lemma}

\begin{proof}
    We thank Chun Yin Hui for informing us that this lemma holds without the hypothesis of weak compatibility, and for providing this argument.

    Since $\det\rho'(\Frob_p)$ is an algebraic number for all but finitely many primes $p$, it follows from \cite{serre-abelian}*{Prop.~2 on p.~III-25} that $\det\rho'$ is almost locally algebraic, i.e.\ that there is an integer $N$ such that $(\det\rho')^N$ is locally algebraic. Hence, by class field theory, $(\det\rho')^N =  \chi\epsilon_\l^m$ for some integer $m$ and some finite order character $\chi$, where $\epsilon_\l$ denotes the $\l$-adic cyclotomic character.

    Fix an embedding $\iota\:\overline E\hookrightarrow{\C}$ and a prime $p\notin S$. Let $\alpha_1, \ldots, \alpha_k\in\overline E$ be the eigenvalues of $\rho'(\Frob_p)$. Then, since $\rho_\lambda$ is pure of weight $w$, we have $|\iota(\alpha_i)|^2 = p^w$ for each $i$. Hence,
    \begin{equation}\label{eq:ht-comparison-pure}
        p^{2m} = |\iota(\det\rho'(\Frob_p))|^{2N} = \prod_{i=1}^k|\iota(\alpha_i)|^{2N} =p^{Nkw}.
    \end{equation}

    It follows that $m = \frac{Nkw}2$. We see that $(\det\rho'\tensor\epsilon_\l^{-kw/2})^N = \chi$ is a finite order character, and hence $\det\rho'\tensor\epsilon_\l^{-kw/2}$ is too. Taking the convention that $\epsilon_\l$ has Hodge--Tate weight $1$, it follows that $\det\rho'$ has Hodge--Tate weight $\frac{kw}2$.
\end{proof}

\subsubsection{Lie irreducibility and primitivity}

\begin{definition}\label{def: Lie_irred}
    Let $G$ be a group and let $\rho\:G\to\GL(V)$ be a representation of $G$. 
    \begin{enumerate}
        \item We say that $\rho$ is \emph{imprimitive} if $\rho$ is an induced representation from a proper subgroup of $G$. Otherwise, we say $\rho$ is \emph{primitive}.
        \item We say that $\rho$ is \emph{Lie irreducible} if $\rho|_H$ is irreducible for every finite index subgroup $H$ of $G$.
    \end{enumerate}
\end{definition}
Note that by Frobenius reciprocity, Lie irreducible representations are necessarily primitive.

\begin{definition}
    We say that a compatible system $(\rho_\lambda)_\lambda$ is irreducible if $\rho_\lambda$ is irreducible for all primes $\lambda$. We say that $(\rho_\lambda)_\lambda$ is Lie irreducible if $\rho_\lambda$ is Lie irreducible for all primes $\lambda$.
\end{definition}

\subsection{Independence of $\lambda$ results}

Let $(\rho_\lambda)_\lambda$ be a compatible system and suppose that for a fixed prime $\lambda$, $\rho_\lambda$ decomposes as 
	\[\rho_\lambda = \rho_{1}\+ \rho_{2}\+\cdots\+\rho_{k}.\]
We will employ several strategies to prove that this decomposition is independent of $\lambda$.
Our first strategy depends on the following proposition, which is an easy consequence of the Chebotarev density theorem and the Brauer--Nesbitt theorem. (Recall that our convention is that each $\rho_\lambda$ is semisimple.)

\begin{proposition}\label{prop:brauer-nesbitt}
    Fix a prime $\lambda_0$ and suppose that:
    \begin{enumerate}
        \item There exist subrepresentations $\rho_1, \ldots, \rho_k$ such that
        \[\rho_{\lambda_0} = \rho_{1}\+ \rho_{2}\+\cdots\+\rho_{k}.\]
        \item For each $i$, there is a compatible system of Galois representations $(\rho_{i, \lambda})_\lambda$ such that $\rho_i\simeq\rho_{i, \lambda_0}$.
    \end{enumerate}
    Then 
    \[\rho_\lambda = \rho_{1, \lambda}\+ \rho_{2, \lambda}\+\cdots\+\rho_{k, \lambda}\]
    for all primes $\lambda$.
\end{proposition}

Our second strategy uses independence of $\lambda$ results for compatible systems in the style of Larsen--Pink \cite{larsen-pink}. For each prime $\lambda$, let $\G_\lambda$ be the \emph{algebraic monodromy group} of $\rho_\lambda$, i.e.\ the Zariski closure of $\rho_\lambda(G_\Q)$ in $\GL_{n,\elb}$. Let $\G_\lambda^\circ$ be its identity connected component, and let $(\G_\lambda^\circ)'$ be the derived subgroup of $\G_\lambda^\circ$.

\begin{definition}
    The \emph{semisimple rank} of $\rho_\lambda$ is the rank of the Lie group $(\G_\lambda^\circ)'$.
\end{definition}

We will make frequent use of the following result of Hui, which builds on a previous result of Serre and Larsen--Pink \cite{larsen-pink} that the rank of $\G_\lambda$ is independent of $\lambda$:

\begin{theorem}[\cite{hui-mrl}*{Thm.~3.19, Rem.~3.22}, \cite{hui-coefficients}*{Thm.~1.3}]\label{thm:semisimple-rank}
    Let $(\rho_\lambda\:G_\Q\to\GL_n(\elb))_\lambda$ be a Serre compatible system of Galois representations. Then the semisimple rank of $\rho_\lambda$ is independent of $\lambda$.
\end{theorem}

\subsection{Lifting Galois representations}\label{sec:lifting}

A key component of our proofs of \Cref{thm:irreducible,thm:decomp} is the following result, which will allow us to lift crystalline representations $G_\Q\to\SO_3$ and $G_\Q\to\SO_5$ to crystalline representations $G_\Q\to \GL_2$ and $G_\Q\to \Gf$. We are grateful to Stefan Patrikis for explaining how this result follows from the results of \cites{patrikis-sign,patrikis-kuga,patrikis-variations}.

\begin{theorem}\label{thm:lifting}
    Let $\pi\:\widetilde H\twoheadrightarrow H$ be a surjection of linear algebraic groups over $\Qlb$ with kernel equal to a central torus $T$ in $\widetilde H$. Let
    \[\rho\:G_\Q\to H(\Qlb)\]
    be a continuous Galois representation that is unramified almost everywhere and such that $\rho|_{G_{\Ql}}$ is crystalline.
    Then there exists a Galois representation
    \[\widetilde\rho\:G_\Q\to \widetilde H(\Qlb)\]
    that is unramified almost everywhere, with $\widetilde\rho|_{G_{\Ql}}$  crystalline and such that the diagram
        \[\begin{tikzcd}
                                                               & \widetilde H(\Qlb) \arrow[d, two heads,"\pi"] \\
G_{\Q} \arrow[r, "\rho"'] \arrow[ru, "\widetilde\rho", dotted] & H(\Qlb)                                
\end{tikzcd}\]
commutes.
\end{theorem}

\begin{proof}
    By \cite{patrikis-sign}*{Prop.~5.5}, there exists a Galois representation $\widetilde\rho_0\:G_\Q\to \widetilde H(\Qlb)$ lifting $\rho$ that is \emph{geometric}, i.e.\ $\widetilde\rho_0$ is almost everywhere unramified and $\widetilde\rho_0|_{G_{\Ql}}$ is de Rham.
    Moreover, by \cite{patrikis-variations}*{Cor.~3.2.13}, there is a local Galois representation $r\:G_{\Ql}\to \widetilde H(\Qlb)$ lifting $\rho|_{G_{\Ql}}$ that is crystalline.

    Consider the representations $r$ and $\widetilde\rho_0|_{G_{\Ql}}$. Since both representations lift $\rho|_{G_{\Ql}}$, there is a character $\chi\:G_{\Ql}\to T(\Qlb)$ such that $r \simeq \widetilde\rho_0|_{G_{\Ql}}\tensor\chi$. Since both $r$ and $\widetilde\rho_0|_{G_{\Ql}}$ are de Rham, so is $\chi$. Fixing an isomorphism $T(\Qlb)\simeq (\Qlb\t)^k$ for some integer $k$, we can view $\chi$ as a product of $k$ de Rham characters $\chi_i\:G_{\Ql}\to\Qlb\t$. Each $\chi_i$ can be written as $\epsilon_\l^{a_i}\chi_i'$, where $\epsilon_\l^{a_i}$ is some power of the cyclotomic character and $\chi_i'$ is de Rham with Hodge--Tate weight $0$. Let $\epsilon = \prod\epsilon_\l^{-a_i}\:G_\Q\to T(\Qlb)\simeq(\Qlb\t)^k$. Then replacing $r$ with $r\tensor\epsilon$, we may assume that each $\chi_i$ is de Rham with Hodge--Tate weight $0$. It follows that the restriction ${\chi_i}|_{I_{\Ql}}$ has finite image. Thus we can choose global characters $\widetilde\chi_i\:G_\Q\to \Qlb\t$ lifting $\chi_i$. Let $\widetilde\chi = \prod\widetilde\chi_i\:G_{\Q}\to T(\Qlb)\simeq(\Qlb\t)^k$. The result follows by setting $\widetilde\rho = \widetilde\rho_0\tensor\widetilde\chi$.
\end{proof}

\begin{remark}
    This result holds for representations of $G_K$, $K$ totally real, if $\rho$ satisfies the symmetry hypothesis of \cite{patrikis-sign}*{Prop.~5.5}.
\end{remark}

\section{Subrepresentations of compatible systems}\label{sec:compatible}

\subsection{Residual irreducibility}

In order to apply automorphy results to an irreducible subrepresentation $\sigma$ of $\rho_\lambda$, one typically needs to know that the residual representation $\overline\sigma|_{\Q(\zeta_\l)}$ is irreducible. The following result, due to Hui \cite{Hui}, shows that we can always assume this irreducibility once $\l$ is large enough.

\begin{theorem}\label{thm:hui}
Let $(\rho_\lambda\:G_\Q\to\GL_n(\elb) )_\lambda$ be a Serre compatible system of Galois representations. Suppose further that $(\rho_\lambda)_\lambda$ has the following two properties:
    \begin{enumerate}[label=$(\alph*)$, leftmargin=*]
        \item \emph{Bounded tame inertia weights:} there exist integers $N_1, N_2\ge 0$ such that for all but finitely many $\lambda$, the tame inertia weights of the local representation $(\orho_\lambda^{ss}\,\otimes\,\overline{\epsilon}_\l^{N_1})|_{\Ql}$ belong to $[0, N_2]$. Here $\overline{\epsilon}_\l$ is the mod $\l$ cyclotomic character.
        \item \emph{Potential semistability:} there is a finite extension $K/\Q$ such that for all but finitely many $\lambda$ and for every place $v$ of $K$ not above $\l$, the representation $\orho^{ss}_\lambda|_{K_v}$ is unramified.
    \end{enumerate}
    Then, for all but finitely many $\lambda$:
    \begin{enumerate}[leftmargin=*]
        \item If $\sigma$ is a two or three-dimensional irreducible subrepresentation of $\rho_\lambda$, then $\overline\sigma$ is irreducible.
        \item If $\sigma$ is a two or three-dimensional irreducible subrepresentation of $\rho_{\lambda}$ such that $\sigma|_{K}$ is irreducible for every finite abelian extension $K/\Q$, then $\overline\sigma|_{K}$ is irreducible for every finite abelian extension $K/\Q$.
        \item If $\sigma$ is a four-dimensional irreducible subrepresentation of $\rho_{\lambda}$, with algebraic monodromy group contained in $\SO_4$, such that $\sigma|_{K}$ is irreducible for every finite abelian extension $K/\Q$, then $\overline\sigma|_{K}$ is irreducible for every finite abelian extension $K/\Q$.
    \end{enumerate}
\end{theorem}

\begin{proof}
    By \cite{hui-coefficients}*{Thm.~1.3}, enlarging $E$ if necessary, we may assume that the compatible system is $E$-rational.
    
    Part $(i)$ is exactly \cite{Hui}*{Cor.~1.3}.
   
   Parts $(ii)$ and $(iii)$ follow from \cite{Hui}*{Thm.~3.12(v)}, using the fact that in each of the cases, the algebraic monodromy group of $\sigma$ is necessarily of type A (in the sense that the root system of its Lie algebra is a product of type A root systems).
\end{proof}

\begin{proposition}[\cite{Hui}*{Thm.~4.1}, \cite{hui-coefficients}*{Thm.~1.3}]\label{strict-compatibility}
     Let $(\rho_\lambda\:G_\Q\to\GL_n(E_\lambda))_\lambda$ be a strictly compatible system of Galois representations. Then $(\rho_\lambda)_\lambda$ satisfies the hypotheses of \Cref{thm:hui}.
\end{proposition}

\begin{proposition}\label{prop:etale-compatibility}
    Let $X$ be a smooth projective variety over $\Q$, and let $(\rho_\lambda\:G_\Q\to \GL_n(\elb))_\lambda$ be a compatible system of Galois representations that occurs as a subquotient of the $i$-th \'etale cohomology group of $X$. Then $(\rho_\lambda)_\lambda$ satisfies the hypotheses of \Cref{thm:hui}.
\end{proposition}

\begin{proof}
    Hypothesis $(a)$ of \Cref{thm:hui} holds for any weakly compatible system of Galois representations. Indeed, the proof of \cite{Hui}*{Thm.~4.1} only uses the facts that the Hodge--Tate weights of $\rho_\lambda$ are independent of $\lambda$ and that $\rho_{\lambda}|_{\Ql}$ is crystalline for $\lambda$ with sufficiently large residue characteristic. The result is then a simple consequence of Fontaine--Laffaille theory.
    
    For hypothesis $(b)$, it is sufficient to prove that the hypothesis holds when $(\rho_\lambda)_\lambda$ is the Galois representation arising from the $i$-th \'etale cohomology group of $X$. In this case, the result follows from \cite{de-Jong} (see for example \cite{Geisser-de-Jong-alternation-application}*{Thm.~4.3}), which shows that for each prime $p\in S$, there is a finite index subgroup $I'_p$ of the inertia group $I_p$, such that $\rho_\lambda|_{I'_p}^{ss}$ is trivial for all primes $\lambda$ not dividing $p$. We can then choose our number field $K$ such that for all $p$ in $S$ and for all $v\mid p$, the absolute inertia group $I_{K_v}$ of $K_v$ is contained in $I'_p$.
\end{proof}

\begin{remark}\label{rem: theorem conditions}
    While we've stated \Cref{thm:irreducible,thm:decomp} for strictly compatible systems, hypothesis $(b)$ of \Cref{thm:hui} is the only component of their proof that requires any hypotheses that are stronger than $(\rho_{\lambda})_\lambda$ being a weakly compatible system. In particular, by \Cref{prop:etale-compatibility}, the theorems hold for any weakly compatible system that comes from geometry, except with the weaker conclusion that the subrepresentations $\rho_{\lambda, i}$ are weakly compatible.
\end{remark}

\subsubsection{Residual irreducibility in general}

\Cref{thm:hui} requires that the subrepresentation $\sigma$ has algebraic monodromy group of type A. In the general case, we have the following weaker result, which is a generalisation of \cite[Cor.~3.10]{patrikis-snowden-wiles} and \cite{hui-coefficients}*{Thm.~1.4}. We are grateful to Stefan Patrikis for his assistance with the proof.

\begin{theorem}\label{psw-generalisation}
    Let $K$ be a number field and let $(\rho_\lambda)_\lambda$ be a Serre compatible system of $n$-dimensional Galois representations of $G_K$, defined over a number field $E$. Write
$$\rho_\lambda\otimes \overline{E}_\lambda = \bigoplus_{i=1}^r \rho_{\lambda, i}^{\oplus m_i}$$
    where the $\rho_{\lambda, i}$ are irreducible and pairwise non-isomorphic.
    Then, for each integer $d$ there exists a density $1$ set of rational primes $\Sigma$ with the following property: if $L/K$ is a field extension of degree at most $d$ and if $\lambda$ lies above a prime $\l$ of $\Sigma$, then 
    \begin{enumerate}
        \item if $\rho_{\lambda,i}|_{G_L}$ is irreducible, then  $\orho_{\lambda, i}|_{G_L}$ is irreducible.
        \item if $\rho_{\lambda,i}$ is Lie irreducible, then the representation $\orho_{\lambda, i}|_{G_{L(\zeta_\l)}}$ is irreducible.
    \end{enumerate}
\end{theorem}

We first record the following proposition, which is implicit in \cite{hui-coefficients}. This result is \cite{patrikis-snowden-wiles}*{Thm.~1.7} with the weaker hypothesis that $(\rho_\lambda)_\lambda$ is a Serre compatible system.

\begin{proposition}\label{prop:bh-modified}
    Let $K$ be a number field and let $(\rho_\lambda)_\lambda\:G_K\to \GL_n(E_\lambda)$ be an $E$-rational semisimple Serre compatible system. Let $\rho_\l = \bigoplus_{\lambda\mid\l}\rho_\lambda$, and let $\G_\l$ be its algebraic monodromy group. Then there exist a density $1$ set of rational primes $\Sigma$ and a positive integer $N$ such that for each $\l\in\Sigma$, there is a reductive group scheme $\mathcal{G}_\l/\Z_\l$, with generic fibre $\G_\l$, such that $\mathcal{G}_\l(\Zl)$ contains $\rho_\l(G_K)$ as an open subgroup of index at most $N$. 
\end{proposition}

\begin{proof}
    As explained in \cite{hui-coefficients}*{Sec.~7}, by \cite{hui-coefficients}*{Thm.~1.1(ii)}, there is an abelian $\Q$-rational compatible system $\phi_\l^{ab}$ whose algebraic monodromy group $\G_\l^{ab}$ is exactly $\G_\l^{tor}:=\G_\l^\circ/[\G_\l^{\circ},\G_\l^{\circ}]$. Being abelian, this compatible system is automatically Hodge--Tate, and therefore satisfies the hypotheses of \cite{patrikis-snowden-wiles}*{Thm.~1.6}. The result then follows as in the proof of \cite{patrikis-snowden-wiles}*{Thm.~1.7}, as explained in \cite{patrikis-snowden-wiles}*{Sec.~3.5} and \cite{hui-coefficients}*{Sec.~7}.
    
\end{proof}

\begin{proof}[Proof of \Cref{psw-generalisation}]
 We follow the strategy of the proofs of \cite[Corollaries~3.9, 3.10]{patrikis-snowden-wiles} but with \Cref{prop:bh-modified} in place of \cite{patrikis-snowden-wiles}*{Thm.~1.7}.
    
    By \cite{hui-coefficients}*{Thm.~1.3}, extending $E$ if necessary, we may assume that $(\rho_\lambda)_\lambda$ is $E$-rational and that each irreducible factor of each $\rho_\lambda$ is absolutely irreducible.   
    Let $(\rho_\l)_\l = (\bigoplus_{\lambda\mid\l}\rho_\lambda)_\l$. We consider $(\rho_\l)_\l$ as a $\Q$-rational $n[E:\Q]$-dimensional compatible system.

    Let $\G_\l$ denote the Zariski closure of the image of $\rho_\l$ in $\GL_{[E:\Q]n}$. By \Cref{prop:bh-modified}, there exists a positive integer $N$ and a density one set of rational primes $\Sigma$ such that for each $\l\in \Sigma$, there is a reductive group scheme $\mathcal{G}_\l/\Z_\l$ with generic fibre $\G_\l$ such that $\mathcal{G}_\l(\Zl)$ contains $\rho_\l(G_K)$ as an open subgroup of index at most $N$. 
    
    Let $\mathcal{G}_\l^0$ be the identity connected component of $\mathcal{G}_\l$, let $\pi_0(\mathcal{G}_\l)$ be its component group, and let $\pi_0\: \mathcal{G}_\l\to \pi_0(\mathcal{G}_\l)$ be the projection map. Fix an extension $L/K$ of degree at most $d$. Then $\mathcal{G}_\l(\Zl)$ contains $\rho_\l(G_L)$ with index at most $Nd$. 
    
    Let $\mathcal{H}_\l = \pi_0^{-1}(\pi_0(\rho_\l(G_L)))$, the preimage in $\mathcal{G}_\l$ of the subgroup of $\pi_0(\mathcal{G}_\l)$ that $\rho_\l(G_L)$ intersects.
 Then $\rho_\l(G_L)$ meets every geometric connected component of $\mathcal{H}_\l$, and $\rho_\l(G_L)$ has index at most $Nd$ in $\mathcal{H}_\l(\Z_\l)$.

We first prove (i).
 Fix $i$ and consider the
     representation $\rho_{\lambda,i}$ of $\G_\ell$, let
    $\mathcal V_i$ be a $\mathcal G_\ell$-stable lattice, and let 
        $\mathcal E_i=\End(\mathcal V_i)$.
         Then, if $\ell\in\Sigma$ is sufficiently large, we have
    \[
        (\mathcal E_i^{\rho_\ell(G_L)})_{\overline\F_\lambda}
        =
        (\mathcal E_i^{\mathcal H_\ell})_{\overline\F_\lambda}
        =
        (\mathcal E_{i,\overline\F_\lambda})^{\mathcal H_{\ell,\overline\F_\lambda}}
        =
        (\mathcal E_{i,\overline\F_\lambda})^{\orho_\ell(G_L)}.
    \]

    The first equality holds since $\rho_\l(G_L)$ and $\mathcal{H}_\l$ have the same Zariski closure. The second equality follows from \cite[Proposition 3.5]{patrikis-snowden-wiles} when $\l$ is sufficiently large.
    
    The third equality follows from \cite[Proposition 3.6]{patrikis-snowden-wiles}. To apply this proposition, we need to use the fact that the index of $\orho_\l(G_L)$ in $\mathcal{H}_\l(\F_\l)$ is bounded independently of $L$ and $\l$ and the fact that for $\l$ large enough, $\orho_\l(G_L)$ meets every geometric component of $\mathcal{H}_{\l,\overline{\F}_\lambda}$. This latter fact is true because, by definition, $\rho_\l(G_L)$ meets every geometric component of $\mathcal{H}_{\l,\Z_\l}$ and, when $\l$ is large enough, the reduction of $\rho_\l(G_L)$ mod $\l$ has kernel coprime to $\#\pi_0(\mathcal{G}_\l)$. This proves $(i)$.

We now prove $(ii)$. Shrinking $\Sigma$ if necessary, we may assume that part $(i)$ holds with $nd$ in place of $d$, where $n = \dim\rho_{\lambda}$. 

Suppose for contradiction that $\orho_{\lambda, i}|_{G_{L(\zeta_\l)}}$ is reducible. Then, by Clifford theory, there is an intermediate field $L\sub L'\sub L(\zeta_\l)$ with $[L':K]\le \dim\rho_{\lambda, i}\le n$ and $\orho_{\lambda, i}|_{G_{L'}}$ reducible. But $[L':K]\le nd$ and $\rho_{\lambda, i}|_{G_{L'}}$ is irreducible, so this contradicts $(i)$.
\end{proof}

\begin{remark}\label{rem:psw}
    Using \Cref{psw-generalisation}, one can prove versions of \Cref{thm:irreducible,thm:decomp} with the weaker hypothesis that $(\rho_\lambda)_\lambda$ is weakly compatible, but with the weaker conclusion that the system is irreducible for all primes in a set of Dirichlet density $1$.
\end{remark}

\subsection{Compatibility of subrepresentations of compatible systems}

Recall that $(\rho_\lambda)_\lambda$ is a strictly compatible system of representations of $G_\Q$. We deduce the following proposition, which is a generalisation of \cite{CG}*{Prop.~2.7} (cf.\ \cite{Calegari-even1}):

\begin{proposition}\label{prop:odd}
    Let $(\rho_\lambda)_\lambda$ be a  strictly compatible system. For all but finitely many $\lambda$, if $\sigma$ is an irreducible two-dimensional Hodge--Tate regular subrepresentation of $\rho_\lambda$, then $\sigma$ is odd.
\end{proposition}

\begin{proof}
    We first observe that if $\sigma|_K$ is reducible for some finite extension $K/\Q$, then we can assume that $K$ is a quadratic extension. Indeed, if $\sigma|_K$ is reducible, then since $\sigma$ is Hodge--Tate regular, $\sigma|_K$ is a sum of distinct characters, and it follows from Clifford theory that $K$ contains a quadratic subextension $K'$ such that $\sigma|_{K'}$ splits as a sum of two characters, and $\sigma$ is induced from either of these characters. 
    
    Moreover, $K$ must be imaginary. Indeed, if $K$ is real, then, by class field theory, every Hodge--Tate character of $G_K$ is of the form $\chi\epsilon_\l^i$, where $\chi$ is a finite order character and $\epsilon_\l$ is the cyclotomic character. Thus $\Ind_K^\Q\chi\epsilon_\l^i$ is not Hodge--Tate regular.

    Hence, if $\sigma|_K$ is reducible for some finite extension $K/\Q$, then $\sigma$ is induced from an imaginary quadratic extension, so $\Tr\sigma(c) = 0$. It follows that $\det\sigma(c) = -1$.
    
    Thus, we may assume that $\sigma$ is Lie irreducible. It follows that $\Sym^2(\sigma)$ is a Lie irreducible subrepresentation of $\rho_\lambda\tensor\rho_\lambda$.  
    
    Applying \Cref{thm:hui}$(ii)$ to the compatible system $(\rho_\lambda\tensor\rho_\lambda)_\lambda$,  we deduce that $\Sym^2(\overline\sigma)|_{\Q(\zeta_\l)}$ is irreducible for all sufficiently large $\lambda$. The result follows from \cite{CG}*{Prop.~2.5}.
\end{proof}

    \begin{proposition}\label{lem:2d-3d-compatible}
        Let $(\rho_\lambda)_\lambda$ be a strictly compatible system. For all but finitely many primes $\lambda$, if $\rho_\lambda$ contains an irreducible, Hodge--Tate regular two- or three-dimensional subrepresentation $\sigma$, and if $\sigma\simeq\sigma\dual\tensor\chi$ for some character $\chi$, then $\sigma$ is contained in an absolutely irreducible strictly compatible system.
    \end{proposition}
    
    \begin{proof}

     Since $\sigma$ is Hodge--Tate regular, if $\sigma|_K$ is reducible for some finite extension $K/\Q$, then by Clifford theory, $\sigma$ is induced from a one-dimensional representation, in which case the result follows from class field theory.
    	
    	If $\sigma|_K$ is irreducible for all finite extensions $K/\Q$ and $\lambda$ is large enough, then $\sigma$ satisfies all the conditions of \cite{BLGGT}*{Thm.~C}. Indeed, odd essential self-duality follows from \Cref{prop:odd} if $\sigma$ is two-dimensional, and is automatic if $\sigma$ has odd dimension. Residual irreducibility follows from \Cref{thm:hui}$(ii)$ and potential diagonalisability follows from \cite{gao-liu}. Hence, by \cite{BLGGT}*{Thm.~C}, $\sigma$ is contained in a strictly compatible system of $G_\Q$. The absolute irreducibility of this compatible system follows from the irreducibility of Galois representations associated to the corresponding automorphic representations \cites{Ribet77, blasius-rogawski} (see also \cite{hui-BLMS}*{Thm.~1.1}).
    \end{proof}

            \begin{proposition}\label{lem:2d-irregular}
        Let $(\rho_\lambda)_\lambda$ be a  strictly compatible system. For all but finitely many primes $\lambda$, if $\rho_\lambda$ contains an irreducible two-dimensional subrepresentation $\sigma$, such that:
        \begin{itemize}
            \item $\sigma$ is odd: for any complex conjugation $c$, $\Tr\sigma(c) = 0$
            \item $\sigma$ has Hodge--Tate weights $\{a,a\}$ for some $a\iZ$
        \end{itemize}
        then $\sigma$ is contained in an absolutely irreducible strictly compatible system.
    \end{proposition}

    \begin{proof}
        Replacing $\sigma$ with $\sigma\tensor\epsilon_\l^{-a}$, we may assume, with no loss in generality, that $a = 0$.
    
        We will show that if $\lambda$ is large enough, then either $\sigma$ is an Artin representation, or $\sigma$ is modular, in the sense that it is the Galois representation associated to a weight $1$ modular form (and hence is in any case Artin). In particular, by \cite{DeligneSerre}, $\sigma$ is contained in an absolutely irreducible strictly compatible system.

        If $\sigma$ is Artin, then $\sigma$ is automatically contained in a strictly compatible system.     
        So suppose that $\sigma$ is not an Artin representation. Then, by \cite{patrikis-variations}*{Prop.~3.4.1}, $\sigma$ is either Lie irreducible or an induced representation. But $\sigma$ cannot be induced: if it were an induced representation, it would have to be induced from a character with parallel Hodge--Tate weights $0$, i.e.\ an Artin character, thus making $\sigma$ Artin as well.
        
        Thus, $\sigma$ is Lie irreducible and by \Cref{thm:hui}, taking $\lambda$ large enough, it follows that $\overline\sigma|_{\Q(\zeta_\l)}$ is irreducible.  Thus, $\sigma$ satisfies the hypotheses of \cite{Pilloni-Stroh}*{Thm.~0.2}, so it is modular, and the result follows.        
    \end{proof}

        \begin{proposition}\label{lem:3d-irregular}
        Let $(\rho_\lambda)_\lambda$ be a strictly compatible system. For all but finitely many primes $\lambda$, if $\rho_\lambda$ contains an irreducible three-dimensional subrepresentation $\sigma$, such that:
        \begin{itemize}
            \item $\sigma\simeq\sigma\dual\tensor\chi$ for some character $\chi$;
            \item for any complex conjugation $c$, $\Tr\sigma(c) = \pm1$;
            \item $\sigma$ has Hodge--Tate weights $\{0,0,0\}$
        \end{itemize}
        then $\sigma$ is contained in an absolutely irreducible strictly compatible system.
    \end{proposition}

    \begin{proof}
        Note that by class field theory, both $\chi$ and $\det\sigma$ are contained in compatible systems. Let $\sigma' = \sigma\tensor\det\sigma\ii\tensor\chi$. Then $\sigma'$ is self-dual with trivial determinant, so takes values in $\SO_3(\elb)$, and $\sigma'(c)$ has eigenvalues $\{-1,-1,1\}$.

        Via the isomorphism $\SO_3\cong\PGL_2$, we can view $\sigma'$ as a representation $G_\Q\to\PGL_2(\elb)$, which, by \cite{patrikis-sign}*{Prop.\ 5.5}, lifts to a representation $r'\:G_\Q\to\GL_2(\elb)$ with the property that the composition
        \[G_\Q\xrightarrow{r'}\GL_2(\elb)\to\PGL_2(\elb)\to\SO_3(\elb)\]
        is isomorphic to $\sigma'$. Explicitly, we have $\sigma' \simeq \Sym^2(r')\tensor(\det r')\ii$. In particular, $r'(c)$ has eigenvalues $\{-1,1\}$, so $r'$ is odd. Moreover, if $\lambda$ is large enough so that $\rho_\lambda|_{G_{\Ql}}$ is crystalline, then by \Cref{thm:lifting}, we can take $r'$ to be crystalline as well, with Hodge--Tate weights $\{0,0\}$.

        A similar argument to that of \Cref{lem:2d-irregular} shows that $r'$ is modular. Indeed, if $r'$ is an Artin representation, then $r'$ is modular by \cite{Pilloni-Stroh}*{Thm.~0.3}. And if $r'$ is not Artin, then as in \Cref{lem:2d-irregular} it is Lie irreducible.  It follows that $\sigma'$ is Lie irreducible.  By \Cref{thm:hui},  taking $\lambda$ large enough, it follows that $\overline\sigma'|_{\Q(\zeta_\l)}$ is irreducible, and hence that $\overline r'|_{\Q(\zeta_\l)}$ is too. Thus, $r'$ satisfies the hypotheses of \cite{Pilloni-Stroh}*{Thm.~0.2}, so $r'$ is modular (and hence Artin).

        It follows from \cite{DeligneSerre} that $r'$ is contained in an absolutely irreducible compatible system $(r'_\lambda)_\lambda$. Composing with the above maps, we see that $\sigma'$, and hence $\sigma$ is contained in a compatible system. Moreover, this compatible system is absolutely irreducible, since it is reducible if and only if $r'_\lambda$ is an induced representation, but this property is independent of $\lambda$, and does not hold for the initial representation $r'$. 
    \end{proof}

    \begin{remark}\label{rem:totally real}
        \Cref{prop:odd,lem:2d-3d-compatible} hold when $(\rho_\lambda)_\lambda$ is a strictly compatible system of $G_K$ for a totally real field $K$, with exactly the same proof.

        However, while we expect them to be true, we do not know how to prove \Cref{lem:2d-irregular,lem:3d-irregular}, except for compatible systems of representations of $G_\Q$. Specifically, over a totally real field, we cannot handle the case that $\sigma$ is Hodge--Tate irregular at one place, but Hodge--Tate regular at another. In this case, the associated automorphic representation is conjecturally a partial weight $1$ Hilbert modular form. 

        In particular, if we assume that $(\rho_\lambda)_\lambda$ is Hodge--Tate regular, then \Cref{thm:irreducible,thm:decomp} both hold for representations of $G_K$ for a totally real field $K$. However, for irregular representations, we are forced to assume that $K=\Q$.
    \end{remark}

\section{Decompositions of Galois representations}

The goal of this section is to prove \Cref{thm:irreducible,thm:decomp}. Let 
\[(\rho_\lambda\:G_\Q\to \GL_5(\elb))_\lambda\]
be a strictly compatible system of Galois representations.

\subsection{Decompositions of five-dimensional Galois representations}

For a five-dimensional representation $\rho_\lambda\: G_\Q \to \GL_5(\elb)$, the following table lists the possibilities for its monodromy group $G=(\G_\lambda^\circ)'$. The entries in the first column are of the form $(G, V)$, where $G$ is a semisimple Lie group and $V$ is a $5$-dimensional representation of $G$. Here, $\iota$ denotes the usual $n$-dimensional injective representation of $\SL_n$, $\Sp_n$ or $\SO_n$ composed with a block diagonal embedding into $\SL_5$. The entries in the second column indicate the decomposition of $\rho_\lambda|_K$ into subrepresentations, where $\Gal(K/\Q)\cong \G_\lambda/\G_\lambda^\circ$. The third column indicates whether or not $G$ can be embedded into $\GO_5$. The fourth indicates whether or not $G$ can be self-dual up to twist.

\begin{table}[H]
\caption{Possible monodromy groups and semisimple ranks of $\rho_\lambda$}
\resizebox{\columnwidth}{!}{
\begin{tabular}{|l|l|l|l|l|l|}
\hline
&$(G, V)$                                & Decomposition of $\rho_\lambda|_K$ & $\GO_5$-valued &self-dual& Rank of $(\G_\lambda^\circ)'$ \\ \hline
1&$(\SL_5, \iota)$            & Irreducible & No  &No& $4$ \\ \hline
2&$(\SO_5, \iota)$            & Irreducible & Yes &Yes & $2$ \\ \hline
3&$(\SL_2, \Sym^4(\iota))$          & Irreducible & Yes&Yes & $1$ \\ \hline
4&$(\SL_4, \iota)$            & $4+1$       & No &No & $3$ \\ \hline
5&$(\Sp_4, \iota)$             & $4+1$       & No &Yes & $2$ \\ \hline
6&$(\SO_4, \iota)$            & $4+1$       & Yes&Yes & $2$ \\ \hline
7&$(\SL_2, \Sym^3(\iota))$          & $4+1$       & No&Yes  & $1$ \\ \hline
8&$(\SL_3\times\SL_2, \iota)$ & $3 + 2$     & No&No  & $3$ \\ \hline
9&$(\SL_2\times\SL_2, \Sym^2(\iota)\times \iota)$ & $3+2$& No  &Yes     & $2$  \\ \hline
10&$(\SL_2\times\SL_2, \iota)$ & $2+2+1$     & No&Yes & $2$ \\ \hline
11&$(\SL_2, \iota\times\iota)$    & $2+2+1$     & Yes&Yes& $1$ \\ \hline
12&$(\SL_2, \Sym^2(\iota)\times\iota)$    & $3+2$     & No&Yes& $1$ \\ \hline
13&$(\SL_3, \iota)$ & $3+1+1$& No &No & $2$\\ \hline
14&$(\SL_2, \Sym^2(\iota))$ & $3+1+1$    & Yes    &Yes    & $1$      \\ \hline
15&$(\SL_2, \iota)$             & $2+1+1+1$   & No &Yes& $1$ \\ \hline
16&$\{1\}$                    & $1+1+1+1+1$ & Yes&Yes & $0$ \\ \hline
\end{tabular}\label{table:semisimple}
}
\end{table}

\subsection{The non-self-dual case}

We first prove \Cref{thm:irreducible} assuming condition $(iii)$, that there exists a prime $\lambda_0$ such that $\rho_{\lambda_0}$ is both irreducible and does not factor through $\GO_5$:

\begin{proposition}\label{prop:non-self-dual}
    Suppose that for some prime $\lambda_0$, $\rho_{\lambda_0}$ is irreducible and not isomorphic to a representation that factors through $\GO_5(\overline E_{\lambda_0})$. Then $\rho_\lambda$ is irreducible for all primes $\lambda$.
\end{proposition}

\begin{lemma}\label{lem:lie-irred}
    Assume the hypotheses of \Cref{thm:irreducible}. In particular, suppose that $\rho_{\lambda_0}$ is irreducible for some prime $\lambda_0$. Then either
    \begin{enumerate}
        \item $\rho_{\lambda_0}$ is Lie irreducible.
        \item There is a degree $5$ extension $K/\Q$ and a compatible system of one-dimensional representations $(\sigma_\lambda\:G_K\to\elb\t)_\lambda$ such that for all $\lambda$, $\rho_\lambda$ is irreducible and isomorphic to $\Ind_K^\Q(\sigma_\lambda)$.
    \end{enumerate}
    In particular, if $\rho_{\lambda_1}$ is reducible for some other prime $\lambda_1\ne \lambda_0$, then $\rho_{\lambda_0}$ is Lie irreducible.
\end{lemma}

\begin{proof}
    By assumption, $\rho_{\lambda_0}$ has at least two distinct Hodge--Tate weights, so it cannot be a twist of an Artin representation. By \cite{patrikis-variations}*{Prop.~3.4.1}, there is a finite extension $K/\Q$, a Lie irreducible representation $\sigma$, and an Artin representation $\omega$ such that
	     \[\rho_{\lambda_0}\simeq \Ind_{K}^{\Q}(\sigma\tensor\omega).\]
	   Since $\rho_{\lambda_0}$ is five-dimensional, either $K = \Q$ or $[K:\Q] = 5$. 

    If $K=\Q$, then $\rho_{\lambda_0}\simeq\sigma\tensor\omega$. By assumption, $\rho_{\lambda_0}$ has at least two distinct Hodge--Tate weights, so it cannot be an Artin representation. Hence $\dim\omega = 1$ and $\rho_{\lambda_0}$ is Lie irreducible.

    If $[K:\Q] = 5$, then, without loss of generality, we can write $\rho_{\lambda_0}\simeq\Ind_K^\Q\sigma$ for some one-dimensional representation $\sigma$ of $G_K$. By class field theory, $\sigma$ lives in a compatible system $(\sigma_\lambda)_\lambda$. Hence, by the Chebotarev density theorem and the Brauer--Nesbitt theorem, $\rho_{\lambda} \simeq \Ind_K^\Q\sigma_{\lambda}$ for all $\lambda$. 
    
    Now, by Mackey's irreducibility criterion, $\Ind_{K}^\Q\sigma_\lambda$ is irreducible if and only if for every $\tau \in G_\Q\setminus G_K$, we have
    \[(\sigma_\lambda|_{G_K\cap G_{\tau(K)}})^\tau\not\simeq(\sigma_\lambda|_{G_K\cap G_{\tau(K)}}).\]
    By the Chebotarev density theorem, this condition can be checked on Frobenius elements, and it is therefore independent of $\lambda$. It follows that $\Ind_{K}^\Q\sigma_\lambda$ is irreducible for all $\lambda$. Hence, $\rho_{\lambda} \simeq \Ind_K^\Q\sigma_{\lambda}$ is irreducible for all $\lambda$.
\end{proof}

\begin{remark}\label{rem:lie-irred}
    Under our assumptions on the Hodge--Tate weights of $\rho_\l$, if $\rho_\l$ is irreducible but not Lie irreducible, then $\rho_\l$ cannot be self-dual. Indeed, suppose that $\rho_\l$ is self-dual and irreducible, but not Lie irreducible. Then, by \cite{patrikis-variations}*{Prop.~3.4.1}, since $\rho_\l$ cannot be Artin, we have $\rho_\l\simeq\Ind_K^\Q\chi$, where $K/\Q$ is a degree $5$ extension and $\chi$ is a character of $G_K$.

    Let $\widetilde K$ be the Galois closure of $K$. Restricting to $G_{\widetilde K}$, the semisimplification of $\rho_\l$ is a direct sum of the five conjugates of $\chi$. Since $\rho_\l$ is self-dual, this multiset of characters is stable under inversion. Since it has odd cardinality, one of these characters must be equal to its inverse. Thus $\chi^2$ is trivial on a finite-index subgroup of $G_K$, and so $\chi$ has finite image. Hence $\chi$, and therefore $\rho_\l$ is Artin, a contradiction.
\end{remark}

\begin{proof}[Proof of \Cref{prop:non-self-dual}]
    By \Cref{lem:lie-irred}, we may assume that $\rho_{\lambda_0}$ is Lie irreducible. Since $\rho_{\lambda_0}$ is not isomorphic to a representation valued in $\GO_5$, it follows from \Cref{table:semisimple} that $(\G_{\lambda_0}^\circ)'\cong \SL_5$ and that $\rho_{\lambda_0}$ has semisimple rank $4$. Hence, by \Cref{thm:semisimple-rank}, it follows that $\rho_\lambda$ has semisimple rank $4$ for all $\lambda$. Using \Cref{table:semisimple} again, we see that $(\G_{\lambda}^\circ)'\cong \SL_5$ for all $\lambda$, so $\rho_\lambda$ is irreducible for all $\lambda$.
\end{proof}

\subsection{The orthogonal case}

Next, we assume that each $\rho_\lambda$ takes values in $\GO_5(\elb)$. In particular, there is a compatible system of characters $(\chi_{\lambda})_\lambda$ such that $\rho_\lambda\simeq\rho_\lambda\dual\tensor\chi_\lambda$ for all $\lambda$. Let $\eta_\lambda = \det\rho_\lambda\tensor\chi_\lambda^{-2}$. Since
\[\det\rho_\lambda = \det\rho_\lambda\ii\tensor\chi_\lambda^5,\]
we see that $\eta_\lambda^2 = \chi_\lambda$. Replacing each $\rho_\lambda$ with $\rho_\lambda\tensor\eta_\lambda\ii$, we are free to assume, with no loss in generality, that for every $\lambda$:
\begin{itemize}
    \item If $\rho_\lambda$ is pure, then it is pure of weight $0$. In particular, if $\sigma$ is a subrepresentation of $\rho_\lambda$, then $\det(\sigma)$ has Hodge--Tate weight $0$;
    \item $\chi_\lambda = 1$---i.e.\ $\rho_\lambda$ is self-dual---and $\det\rho_\lambda = 1$. In particular, $\rho_\lambda$ is isomorphic to a representation valued in $\SO_5(\elb)$;
    \item $\rho_\lambda$ has Hodge--Tate weights $\{-b, -a, 0, a, b\}$ for some integers $b\ge a\ge 0$ such that either $b>a$ or $a>0$.
\end{itemize}

    There is an exceptional isomorphism
		\[\PGSp_4(\elb)\xrightarrow{\sim}\SO_5(\elb),\]
	which we can use to view each $\rho_\lambda$ as a representation
	\[G_\Q\to\PGSp_4(\elb).\]
	By \cite{patrikis-sign}*{Prop.~5.5}, this projective representation has a geometric lift 
	\[r_\lambda\:G_\Q\to\Gf(\elb),\]
        and by \Cref{thm:lifting}, we may assume that $r_\lambda$ is crystalline at $\l$ if $\rho_\lambda|_{\Ql}$ is. The choice of lift is well-defined up to twisting by a character. 

 \begin{remark}\label{rem:trace-mod-powers}
    We cannot assume that $(r_\lambda)_\lambda$ is a compatible system of Galois representations: see \cite{patrikis-variations}*{Ques.~1.1.9} for further discussion of this hard problem. In particular, we cannot assume that for a good prime $p$, $\Tr(r_\lambda(\Frob_p))$ is independent of $\lambda$. On the other hand, the image of $\Tr(r_\lambda(\Frob_p))$ in $\{0\}\cup\elb\t/\elb^{\times 4}$ is independent of $\lambda$.
 \end{remark}

     Let $\simil r_\lambda$ denote the composition of $r_\lambda$ with the similitude character $\simil\:\Gf(\elb)\to \elb\t$. Then there is an isomorphism
	\begin{equation}\label{eqn:wedge-r2-decomp}
			\wedge^2(r_\lambda)\tensor\simil r_\lambda\ii\simeq \rho_\lambda\+\chi_\triv,
	\end{equation}
	where $\chi_\triv$ denotes the trivial character. 
	In particular, if $\simil r_\lambda$ has Hodge--Tate weight $d$, and if $\{n,m,d-m,d-n\}$ are the Hodge--Tate weights of $r_\lambda$, with $n \le m$, then the Hodge--Tate weights of $\rho_\lambda$ are
	\[ \{-d + n+m,  n-m, 0, m-n, d - n - m\}.\]
	It follows that $n = -a+k$, $m=k$, and $d=2k+b-a$, for some $k\iZ$. Twisting by the $(a-k)$-th power of the cyclotomic character, we can choose $r_\lambda$ so that it has Hodge--Tate weights $\{0,a,b,a+b\}$.

    In the remainder of this section, we will use the decomposition of $r_\lambda$ to classify the possible decompositions of $\rho_\lambda$, culminating in \Cref{prop:rl-red} when $r_\lambda$ is reducible, and \Cref{prop:rl-irred} when $r_\lambda$ is irreducible. The possible decompositions are summarised in \Cref{table-cases}.
	
	\subsubsection{Reducible $r_\lambda$}
 
	Fix a prime $\lambda$ and suppose that $r_\lambda$ is reducible.
	
	\begin{lemma}\label{lem:no3d}
		The representation $r_\lambda$ does not contain an irreducible three-dimensional subrepresentation.
	\end{lemma}

	\begin{proof}
        In general, if $\sigma_1, \sigma_2$ are representations, we have
        \begin{equation}\label{eq:wedge-2}
            \wedge^2(\sigma_1\+\sigma_2)\simeq \bigoplus_{i = 0}^2\wedge^i(\sigma_1)\tensor\wedge^{2-i}(\sigma_2).
        \end{equation}
        Suppose that $r_\lambda = \chi\+\sigma$, with $\chi$ one-dimensional and $\sigma$ irreducible and three-dimensional. Then 
    \[\wedge^2(r_\lambda) \simeq \bigoplus_{i = 0}^2\wedge^{2-i}(\chi)\tensor\wedge^{i}(\sigma)\simeq (\chi\tensor\sigma) \+ \wedge^2(\sigma)\simeq (\chi\tensor\sigma) \+ (\sigma\dual\tensor\det\sigma)\]
         is a direct sum of two irreducible three-dimensional representations, contradicting \eqref{eqn:wedge-r2-decomp}, which states that $\wedge^2(r_\lambda)$ must contain a one-dimensional subrepresentation.
	\end{proof}
	
	\begin{lemma}\label{lem:rl2+1+1}
		Suppose that
		\[r_\lambda = \sigma \+ \chi_1 \+\chi_2,\]
		where $\sigma$ is a two-dimensional irreducible representation and $\chi_1, \chi_2$ are characters. Then
        \[\chi_1\chi_2 = \det\sigma = \simil r_\lambda.\]
	\end{lemma}

	\begin{proof}
	By \eqref{eqn:wedge-r2-decomp}, $\simil r_\lambda$ is a subrepresentation of 
		\[\wedge^2(r_\lambda) \simeq  \bigoplus_{i = 0}^2\wedge^i(\sigma)\tensor\wedge^{2-i}(\chi_1\+\chi_2)=( \sigma\tensor\chi_1) \+(\sigma\tensor\chi_2) \+ \det\sigma \+ \chi_1\chi_2.\]
		Since $\sigma$ is irreducible, it follows that either $\det\sigma = \simil r_\lambda$ or $\chi_1\chi_2 = \simil r_\lambda$. 
  
        On the other hand, we have $\det r_\lambda = (\det\sigma)\chi_1\chi_2$ and $\det r_\lambda = \simil r_\lambda^2$. It follows that $\chi_1\chi_2 = \det\sigma = \simil r_\lambda$.
	\end{proof}

	\begin{lemma}\label{lem:rl2+2}
		Suppose that $r_\lambda$ decomposes as
		\[r_\lambda\simeq \sigma_1\+\sigma_2\]
		where $\sigma_1, \sigma_2$ are irreducible two-dimensional representations. Then either:
            \begin{enumerate}
                \item Up to reordering, $\sigma_1$ has Hodge--Tate weights $\{0, a+b\}$, $\sigma_2$ has Hodge--Tate weights $\{a,b\}$, and $\det\sigma_1 = \det\sigma_2 = \simil r_\lambda$.
                \item $\sigma_2\simeq\sigma_1\tensor\chi$ for some  character $\chi$.
            \end{enumerate}
	\end{lemma}

\begin{proof}
By \eqref{eqn:wedge-r2-decomp} and \eqref{eq:wedge-2},  $\simil(r_\lambda)$ is a subrepresentation of 
		\[\wedge^2(r_\lambda) = \sigma_1\tensor\sigma_2 + \det\sigma_1\+\det\sigma_2.\]
Now, observe that $\sigma_2\simeq\sigma_1\tensor\chi$ if and only if $\chi$ is a subrepresentation of $\sigma_1\dual\tensor\sigma_2$. Since $\sigma_1\tensor\sigma_2\simeq (\sigma_1\dual\tensor\sigma_2)\tensor\det\sigma_1$, we see that $\sigma_2$ is a character twist of $\sigma_1$ if and only if $\sigma_1\tensor\sigma_2$ contains a one-dimensional subrepresentation.
  
  Hence, either $\simil(r_\lambda)$ is a subrepresentation of $\sigma_1\tensor\sigma_2$, in which case $(ii)$ holds, or we have either $\det\sigma _1= \simil r_\lambda$ or $\det\sigma_2 = \simil r_\lambda$. Since $\det\sigma_1\det\sigma_2= \det r_\lambda = \simil r_\lambda^2$, it follows that $\det\sigma_1 = \det\sigma_2 = \simil r_\lambda$, and the Hodge--Tate weights of $\sigma_1$ and $\sigma_2$ must have the same sum, which is case $(i)$.
	\end{proof}

\begin{remark}\label{rem:rl2+2}
    Suppose that $\rho_{\lambda}$ is not Hodge--Tate regular and that $\sigma_2\simeq\sigma_1\tensor\chi$ for some  character $\chi$. Then either:
    \begin{itemize}
        \item $a = b$, $\sigma_1$ has Hodge--Tate weights $\{0,a\}$ and $\sigma_2$ has Hodge--Tate weights $\{a,2a\}$.
        \item $a = 0$ and $\sigma_1, \sigma_2$ both have Hodge--Tate weights $\{0,b\}$.
    \end{itemize}
    The only other possibility, that $a = 0$, $\sigma_1$ has Hodge--Tate weights $\{0,0\}$, and $\sigma_2$ has Hodge--Tate weights $\{b,b\}$ cannot occur by \Cref{lem:pure-ht-weight}: note that by assumption, if $a = 0$, then $\rho_\lambda$, and hence $r_\lambda$ is pure.
\end{remark}

	\begin{proposition}\label{prop:rl-red}
		Suppose that $r_\lambda$ is reducible. Then either:
		\begin{enumerate}
			\item There are irreducible two-dimensional representations $\sigma_1, \sigma_2$, with Hodge--Tate weights $\{0, a+b\}$ and $\{a,b\}$, such that
   \[r_\lambda\simeq\sigma_1\+\sigma_2\]
   and
			\[\rho_\lambda \simeq (\sigma_1\tensor\sigma_2\dual) \+\chi_\triv.\]
			Moreover, $\sigma_1\tensor\sigma_2\dual$ is either irreducible, or a sum of two-dimensional irreducible representations.

			\item There is an irreducible two-dimensional representation $\sigma$ and a character $\chi$ such that
   \[r_\lambda\simeq \sigma\tensor\chi\+ \chi\+\simil r_\lambda\chi\ii\]
   and
			\[\rho_\lambda \simeq \sigma \+ \sigma\dual\+ \chi_\triv.\]

            \item There is a three-dimensional representation $\sigma$, with Hodge--Tate weights $\{-a,0,a\}$ or $\{-b,0,b\}$, a two-dimensional representation $\sigma_1$, and  characters $\chi, \chi_1, \chi_2$ such that
            \[r_\lambda\simeq \sigma_1\+\sigma_1\tensor\chi\]
            and
            \[\rho_\lambda \simeq \sigma \+ \chi_1\+\chi_2.\]
            In this case, $\sigma $is isomorphic to a quadratic twist of $\Sym^2(\sigma_1)\tensor\det\sigma_1\ii$, and is either irreducible, or a sum of an irreducible two-dimensional representation and a finite order character.
            \item Both $\rho_\lambda$ and $r_\lambda$ are direct sums of one-dimensional representations.
		\end{enumerate}
	\end{proposition}

	\begin{proof}
            First suppose that $r_\lambda$ contains a character. Then either it is a sum of four characters and we are in case $(iv)$, or, by \Cref{lem:rl2+1+1},
		\[r_\lambda\simeq\sigma\+ \chi\+\simil r_\lambda\chi\ii\] 
		for some character $\chi$ and some irreducible two-dimensional representation $\sigma$, such that $\det\sigma = \simil r_\lambda$. Hence, by \eqref{eq:wedge-2}, we have
		\[\wedge^2(r_\lambda) \tensor\simil r_\lambda\ii\simeq (\sigma\tensor\chi\ii)\+(\sigma\tensor\chi\simil r_\lambda\ii)\+\chi_\triv\+\chi_\triv.\]
		It follows from \eqref{eqn:wedge-r2-decomp} that
		\[\rho_\lambda\simeq (\sigma\tensor\chi\ii)\+(\sigma\tensor\chi\simil r_\lambda\ii)\+\chi_\triv\simeq  (\sigma\tensor\chi\ii)\+(\sigma\dual\tensor\chi)\+\chi_\triv.\]
		Replacing $\sigma$ with $\sigma\tensor\chi\ii$ shows that we are in case $(ii)$. 

    Now suppose that $r_\lambda$ does not contain a character. Then $r_\lambda\simeq \sigma_1\+\sigma_2$ where $\sigma_1, \sigma_2$ are irreducible two-dimensional representations. 
    
    First suppose that $\sigma_2\simeq\sigma_1\tensor\chi$ for some character $\chi$.  If $\simil r_{\lambda}$ occurs in $\sigma_{1}\otimes\sigma_{2}$,
then we can choose $\chi$ such that $\simil r_{\lambda }\simeq \det \sigma _{1}\otimes \chi $. In this case,
  \begin{align*}
      \rho_\lambda\+\chi_\triv\simeq\wedge^2(r_\lambda)\tensor\simil r_\lambda\ii&\simeq (\sigma_1\tensor\sigma_1\tensor\det\sigma_1\ii)\+\chi\ii\+\chi\\
      &\simeq (\Sym^2(\sigma_1)\tensor\det\sigma_1\ii) \+\chi\ii\+\chi_\triv\+\chi.
  \end{align*}

    Taking $\sigma = \Sym^2(\sigma_1)\tensor\det\sigma_1\ii$, we see that $\rho_\lambda$ decomposes as in case $(iii)$. The claim about the Hodge--Tate weights of $\sigma$ follows from the fact that $\sigma$ must be self-dual. By \Cref{rem:rl2+2}, $\sigma_1$ is Hodge--Tate regular. Hence, $\sigma_1$ is either Lie irreducible, in which case $\Sym^2(\sigma_1)$ is irreducible, or it is induced from a unique imaginary quadratic extension, in which case $\sigma$ is a direct sum of a two-dimensional irreducible subrepresentation and the quadratic character corresponding to that extension. Thus we are in case $(iii)$. 

    If $\simil r_{\lambda}$ does not occur in $\sigma_{1}\otimes\sigma_{2}$, then
we must have $\simil r_{\lambda}\simeq \det\sigma_{1}\simeq\det\sigma_{2}$, so
$\chi^{2}=\chi _{{\mathrm{triv}}}$. In this case,
\begin{equation*}
\rho _{\lambda}\simeq 
(\mathop{\mathrm{Sym}}\nolimits ^{2}(\sigma _{1})\otimes \det
\sigma _{1}^{-1}\otimes \chi) \oplus \chi _{{\mathrm{triv}}}
\oplus \chi,
\end{equation*}
and the same argument gives $(iii)$.
		
		Finally suppose that $\sigma_1$ is not a character twist of $\sigma_2$. By \Cref{lem:rl2+2}, 
		\[(\rho_\lambda\+\chi_\triv)\tensor\simil r_\lambda \simeq\wedge^2(r_\lambda)= (\sigma_1\tensor\sigma_2 )\+\det\sigma_1\+\det\sigma_2,\]
		and $\det\sigma_1 = \det\sigma_2 = \simil r_\lambda$. It follows that
		\[\rho_\lambda \simeq(\sigma_1\tensor\sigma_2\tensor(\det\sigma_2)\ii )\+\chi_\triv\simeq (\sigma_1\tensor\sigma_2\dual) \+ \chi_\triv,\]
    and $\sigma_1\tensor\sigma_2\dual$ does not contain a character, since $\sigma_1$ is not a character twist of $\sigma_2$. Hence, we are in case $(i)$.
	\end{proof}
	
	\subsubsection{Irreducible $r_\lambda$}
	
	Now suppose that $r_\lambda$ is irreducible. We will make frequent use of the following proposition, which is stated in \cite{GanTakedaSp4} when $G$ is a Weil group and $K = \C$, but works for any group $G$ and any algebraically closed characteristic $0$ field $K$. Recall that a representation is \emph{primitive} if it is not an induced representation.
	
	 \begin{proposition}[\cite{GanTakedaSp4}*{Prop.~5.1}]\label{prop:GT}
		Let $G$ be a group, let $K$ be an algebraically closed characteristic $0$ field and let $r\colon G\to \Gf(K)$ be an absolutely irreducible representation. Write $\std(r)$ for the corresponding $\SO_5$-valued representation, via the maps $\Gf\to\PGSp_4\xrightarrow{\sim}\SO_5$. Then either:
		\begin{enumerate}
			\item $r$ is primitive and $\std(r)$ is irreducible.
			\item There is an index $2$ subgroup $H$, a primitive representation $\sigma$ of $H$ and a character $\chi$ of $H$ with $\chi^2\ne 1$,  such that
			\[r = \Ind_H^G\sigma\text{ and } \sigma^\tau\simeq\sigma\tensor\chi,\]
			where $\tau$ is a representative for $G/H$, and such that
			\[\simil r|_H = \chi\tensor\det\sigma\ne \det\sigma.\]
			In this case, $\std(r)$ is a direct sum of an irreducible three-dimensional representation and an irreducible two-dimensional representation.
			\item There is an index $2$ subgroup $H$ of $G$ and an irreducible two-dimensional representation $\sigma$ of $H$ such that
			\[r =\Ind_H^G\sigma\text{ and }\simil r|_H = \det\sigma.\]
			In this case, $\std(r)$ contains the character $G\to G/H\cong\{\pm 1\}$.
		\end{enumerate}
	\end{proposition}

        Note that in our applications, $r = r_\lambda$ and $\std(r) = \rho_\lambda$.
	
	\begin{remark}\label{rem:asai}
        As in \eqref{eqn:wedge-r2-decomp}, we have an isomorphism
        \[\wedge^2(r)\tensor\simil r \ii \simeq \std(r)\+\chi_\triv.\]
        In cases $(ii)$ and $(iii)$, it follows that
        \begin{align*}
            \wedge^2(r)|_H&\simeq \wedge^2(\sigma\+\sigma^\tau)\\
            &\simeq\sigma\tensor\sigma^\tau\+\det\sigma\+\det\sigma^\tau.
        \end{align*}
        We therefore have
        \[\wedge^2(r) = \mathrm{As}_H^G(\sigma) + \Ind_{H}^G\det\sigma,\]
        where $\mathrm{As}_H^G(\sigma)$ is the Asai lift, otherwise known as the tensor induction of $\sigma$: it is a lift of the representation $\sigma\tensor\sigma^\tau$ of $H$ to a representation of $G$.

        In case $(ii)$, we have $\sigma^\tau\simeq\sigma\tensor\chi$, and therefore
        \[\sigma\tensor\sigma^\tau\simeq\sigma\tensor\sigma\tensor\chi\simeq (\chi\tensor\Sym^2\sigma) \+ (\chi\tensor\det\sigma).\]
        Since $\chi\tensor\det\sigma\simeq\simil r|_H$ and $\det\sigma\not\simeq\det\sigma^\tau$, it follows that $\Ind_H^G\det\sigma$ is irreducible, while $\mathrm{As}_H^G(\sigma)$ decomposes as an irreducible $3$-dimensional representation plus a character. It follows that $\std(r)\tensor\simil r$ decomposes as a direct sum of $\Ind_{H}^G\det\sigma$ and this $3$-dimensional representation.

         In case $(iii)$, we still have 
         \[\wedge^2(r) = \mathrm{As}_H^G(\sigma) + \Ind_{H}^G\det\sigma,\]
         but now, since $\det\sigma = \simil r|_H$, $\Ind_{H}^G\det\sigma$ is reducible: it decomposes as
         \[\Ind_{H}^G\det\sigma\simeq\simil r \+ (\simil r\tensor\chi),\]
         where $\chi$ is the quadratic character $G\to G/H\to\{\pm1\}$. It follows that
         \[\std(r)= \chi \+ (\mathrm{As}_H^G(\sigma)\tensor\simil r\ii).\]
	\end{remark}
	
	\begin{lemma}\label{lem:subrep-non-trivial}
	    Suppose that $r_\lambda$ is absolutely irreducible. Then every one-dimensional subrepresentation of $\rho_\lambda$ is non-trivial and has Hodge--Tate weight $0$.
	\end{lemma}

	\begin{proof}
            If $\rho_\lambda$ contains a one-dimensional subrepresentation, then we must be in case $(iii)$ of \Cref{prop:GT}. Hence, by \Cref{rem:asai}, $\rho_\lambda$ contains the quadratic character $\chi_{K/\Q}$ corresponding to the extension $K/\Q$. This character is non-trivial and has Hodge--Tate weight $0$. 

            Suppose that $\chi\not\simeq\chi_{K/\Q}$ is another subrepresentation of $\rho_\lambda$. By \Cref{rem:asai}, we have
            \[\rho_\lambda\simeq (\As_K^\Q(\sigma)\tensor\simil r_\lambda\ii) \+\chi_{K/\Q},\]
            and restricting to $K$, we have
            \[\rho_\lambda|_K\simeq (\sigma\tensor\sigma^\tau\tensor\simil r_\lambda\ii) \+ \chi_{\triv}\simeq (\sigma\dual\tensor\sigma^\tau) \+ \chi_{\triv},\]
            where $\tau$ is the non-trivial element of $\Gal(K/\Q)$. It follows that $\sigma\dual\tensor\sigma^\tau$ contains $\chi|_K$. Since $\sigma$ is irreducible, by Schur's Lemma, it follows that $\sigma\simeq\sigma^\tau\tensor\chi|_K$.

            Since we are in case $(iii)$ of \Cref{prop:GT}, $\det\sigma = \det\sigma^\tau$, so $\chi|_K^2=\chi_\triv$. Moreover, since $\Ind_K^\Q\sigma$ is irreducible, by Clifford theory, $\sigma\not\simeq\sigma^\tau$, so $\chi|_K$ is non-trivial. It follows that $\chi$ is non-trivial and has finite image, so it has Hodge--Tate weight $0$.
            \end{proof}

\begin{proposition}\label{prop:rl-irred}
    Suppose that $r_\lambda$ is absolutely irreducible. Then either:
    \begin{enumerate}
        \item $\rho_\lambda$ is irreducible and $r_\lambda$ is primitive.
        \item There is a quadratic extension $K/\Q$ and a two-dimensional representation $\sigma$ of $G_K$, with $\det\sigma\not\simeq\det\sigma^\tau$, such that $r_\lambda\simeq\Ind_K^\Q\sigma$. Moreover, there is a character $\chi$ of $G_K$ with $\chi^2\ne\chi_\triv$ such that $\sigma\simeq\sigma^\tau\tensor\chi$. In this case, $\rho_\lambda$ is a direct sum of a self-dual irreducible three-dimensional representation and a self-dual irreducible two-dimensional representation.
        \item  There is a quadratic extension $K/\Q$ and a two-dimensional representation $\sigma$ of $G_K$, with $\det\sigma\simeq\det\sigma^\tau$, such that $r_\lambda\simeq\Ind_K^\Q\sigma$. In this case, $\rho_\lambda$ is a direct sum of a four-dimensional representation and the quadratic character $\chi_{K/\Q}$. Moreover, this four-dimensional representation is either:
        \begin{enumerate}
            \item[$(a)$] Irreducible, with Hodge--Tate weights $\{-b,-a,a,b\}$.
            \item[$(b)$] A sum of an irreducible three-dimensional representation with Hodge--Tate weights $\{-b,0,b\}$ and a non-trivial finite order character. In this case, $a = 0$.
            \item[$(c)$] A sum of two irreducible self-dual two-dimensional representations with Hodge--Tate weights $\{-b,b\}$ and $\{-a,a\}$.
            \item[$(d)$] A sum of an irreducible two-dimensional representation with Hodge--Tate weights $\{-b,b\}$, and two non-trivial finite order characters. In this case, $a = 0$.
        \end{enumerate}
    \end{enumerate}
\end{proposition}

\begin{proof}
    Cases $(i)$ and $(ii)$ correspond to cases $(i)$ and $(ii)$ of \Cref{prop:GT}. In case $(iii)$ of \Cref{prop:GT}, $\rho_\lambda$ is a direct sum of a quadratic character and a four-dimensional representation. If this representation is irreducible, we are in case $(iii)(a)$. If it contains a one-dimensional subrepresentation, by \Cref{lem:subrep-non-trivial}, this subrepresentation is non-trivial and has finite order, so Hodge--Tate weight $0$, which gives cases $(iii)(b)$ and $(iii)(d)$. 

    Finally, suppose that the four-dimensional representation splits as a sum $\sigma_1\+\sigma_2$ of irreducible two-dimensional representations. Since $\rho_\lambda$ is self-dual, either $\sigma_1$ and $\sigma_2$ are both self-dual, in which case they have Hodge--Tate weights $\{-a,a\}$ and $\{-b,b\}$, or $\sigma_1\simeq\sigma_2\dual$. But then $\rho_\lambda \simeq\sigma_1\+\sigma_1\dual\+\chi$ where $\chi$ is a non-trivial quadratic character, so $\det\rho_\lambda=\chi$ is non-trivial, contradicting our assumption that $\rho_\lambda$ is valued in $\SO_5$.
\end{proof}

	\subsubsection{Possible decompositions of $\rho_\lambda$ in the orthogonal case}

	We conclude by combining \Cref{prop:rl-red,prop:rl-irred} into \Cref{table-cases}, which lists the possible ways in which $\rho_\lambda$ and $r_\lambda$ can decompose in the case that $\rho_\lambda$ factors through $\SO_5(\elb)$. We omit the Hodge--Tate regular instances of \Cref{prop:rl-red}$(ii)$, since \Cref{prop:brauer-nesbitt,lem:2d-3d-compatible} already show that these
decompositions extend to compatible systems

\begin{table}[H]
\caption{Possible decompositions of $\rho_\lambda$ and $r_\lambda$ in the orthogonal case}
\resizebox{\columnwidth}{!}{%
\begin{tabular}{|l|l|l|l|}
\hline
Case & Decomposition of $\rho_\lambda$   & Hodge--Tate weights                  & Decomposition of $r_\lambda$                                 \\ \hline
1    & Irreducible                       & $\{-b,-a,0,a,b\}$                    & $\ref{prop:rl-irred}(i)$                                                    \\ \hline
2(i)   & $4 + 1$, one-dimensional trivial  & $\{-b,-a,a,b\}, \{0\}$               & $\ref{prop:rl-red}(i)$                                                          \\ \hline
2(ii) &
  $4 + 1$, one-dimensional non-trivial &
  $\{-b,-a,a,b\}, \{0\}$ &
  $\ref{prop:rl-irred}(iii)(a)$ \\ \hline
3a   & $3 + 2$                           & $\{-a,0,a\}, \{-b,b\}$               & $\ref{prop:rl-irred}(ii)$ \\ \hline
3b   & $3+2$                             & $\{-b,0,b\}, \{-a,a\}$               & $\ref{prop:rl-irred}(ii)$\\ \hline
4a  & $3 + 1 + 1$                       & $\{-a,0,a\}, \{b\}, \{-b\}$          & $\ref{prop:rl-red}(iii)$                                       \\ \hline
4b(i)  & $3+1+1$                           & $\{-b,0,b\}, \{a\},\{-a\}$           & $\ref{prop:rl-red}(iii)$                                    \\ \hline
4b(ii) &
  $3+1+1$ &
  $\{-b,0,b\}, \{0\},\{0\}$ &
  $\ref{prop:rl-irred}(iii)(b)$ \\ \hline
5(i)   & $2 + 2+1$, one-dimensional trivial &  $\{-b,b\}, \{-a,a\}, \{0\}$          & $\ref{prop:rl-red}(i)$ or $a = b$ and $\ref{prop:rl-red}(ii)$                                                        \\ \hline
5(ii) &
  $2 + 2+1$, one-dimensional non-trivial &
  $\{-b,b\}, \{-a,a\}, \{0\}$ &
  $\ref{prop:rl-irred}(iii)(c)$ \\ \hline
6    & $2 + 2+1$, one-dimensional trivial                        & $\{-a,-a\}, \{a,a\}, \{0\}$ & $\ref{prop:rl-red}(i)$ or $\ref{prop:rl-red}(ii)$                                                          \\ \hline
7a  & $2+1+1+1$                         & $\{-a,a\}, \{b\},\{0\},\{-b\}$       & $\ref{prop:rl-red}(iii)$                     \\ \hline
7b(i)  & $2+1+1+1$                         & $\{-b,b\},\{0\},\{0\},\{0\}$        & $\ref{prop:rl-red}(iii)$                    \\ \hline
7b(ii) &
  $2+1+1+1$, all non-trivial &
  $\{-b,b\},\{0\},\{0\},\{0\}$ &
  $\ref{prop:rl-irred}(iii)(d)$\\ \hline
8    & $1+1+1+1+1$                       & $\{-b\},\{-a\},\{0\},\{a\},\{b\}$    &$\ref{prop:rl-red}(iv)$                                           \\ \hline
\end{tabular}\label{table-cases}
}
\end{table}

\subsection{The proof of \Cref{thm:decomp}}\label{sec:decomp-analysis}

    In this section, we prove \Cref{thm:decomp}. In particular, we show that if $(\rho_\lambda)_\lambda$ falls into any of the cases of \Cref{table-cases} infinitely often, then it falls into that case for all primes $\lambda$, and that the subrepresentations are compatible for different $\lambda$.
    
    \begin{lemma}\label{lem:all-regular-subs}
        Let $(\rho_\lambda)_\lambda$ be as in \Cref{thm:decomp}. Suppose that  for infinitely many primes $\lambda$, $\rho_{\lambda}$ is in case $4a$ $($resp.~$4b(i), 4b(ii), 7a, 7b(i), 7b(ii),8)$ of \Cref{table-cases}. Then $\rho_{\lambda}$ is in that case of \Cref{table-cases} for all primes $\lambda$ and \Cref{thm:decomp} holds for $(\rho_\lambda)_\lambda$.
    \end{lemma}

    \begin{proof}
        By \Cref{prop:brauer-nesbitt}, it is sufficient to show that for one of the infinitely many primes $\lambda$, every subrepresentation of $\rho_\lambda$ is contained in a strictly compatible system. This follows from class field theory for all the one-dimensional subrepresentations. Moreover, in these cases, the two- and three-dimensional subrepresentations are Hodge--Tate regular, so the result follows from \Cref{lem:2d-3d-compatible}.  Note that in cases $4a$ and $7a$, $\rho_\lambda$ contains a subrepresentation whose determinant has non-zero Hodge--Tate weight. Hence, $\rho_\lambda$ is not pure, so in these cases, $a \ne 0$ by assumption, and the representation is Hodge--Tate regular.
    \end{proof}

    \begin{lemma}\label{lem:3+2}
        Let $(\rho_\lambda)_\lambda$ be as in \Cref{thm:decomp}. Suppose that for infinitely many primes $\lambda$, $\rho_{\lambda}$ is in case $3a$ $($resp.~$3b)$ of \Cref{table-cases}. Then  $\rho_{\lambda}$ is in case $3a$ $($resp.~$3b)$ of \Cref{table-cases} for all primes $\lambda$ and \Cref{thm:decomp} holds for $(\rho_\lambda)_\lambda$.
    \end{lemma}

    \begin{proof}
    By assumption, we can choose $\lambda_0$ with arbitrarily large residue characteristic, such that $\rho_{\lambda_0}$ decomposes as a direct sum of an irreducible two-dimensional and an irreducible three-dimensional representation. Since $\rho_{\lambda_0}$ is self-dual, both representations are self-dual.
    
If $a >0$, then both these subrepresentations of $\rho_{\lambda_0}$ are Hodge--Tate regular, so the result follows from \Cref{lem:2d-3d-compatible} and \Cref{prop:brauer-nesbitt}. Hence, we can assume that $a = 0$ and that $b>a$. 

        By \Cref{rem:asai}, the two-dimensional subrepresentation of $\rho_{\lambda_0}$ is an induced representation. Hence, by class field theory, it lives in an absolutely irreducible strictly compatible system. If we are in case $3b$, the three-dimensional subrepresentation of $\rho_{\lambda_0}$ is regular, so it is contained in an absolutely irreducible compatible system by \Cref{lem:2d-3d-compatible}, in which case, the result follows from \Cref{prop:brauer-nesbitt}.

        Hence, we can assume that we are in case $3a$. Write $\rho_{\lambda_0}\simeq \sigma_2 \+ \sigma_3$, where $\sigma_2$ is two-dimensional with Hodge--Tate weights $\{-b,b\}$ and $\sigma_3$ is three-dimensional with Hodge--Tate weights $\{0,0,0\}$.
        
        By assumption, $r_{\lambda_0}$ is irreducible and is of the form $\Ind_{K}^{\Q}(\sigma)$ for some quadratic extension $K/\Q$ and some two-dimensional primitive representation $\sigma$ of $G_K$.
        
        By \Cref{rem:asai}, $\sigma_2$ is induced from a character of $G_K$. Since $\sigma_2$ is Hodge--Tate regular, it follows that $K$ is imaginary quadratic. Hence, if $c$ is a complex conjugation in $G_\Q$, by the definition of an induced representation, $\Tr \sigma_2(c) = 0$, so $\sigma_2$ is odd.
        
        Since $a = 0$, by assumption, the eigenvalues of $\rho_{\lambda_0}(c)$ are $\{-1,-1,1,1,1\}$ (since $\Tr(\rho_{\lambda_0}(c)) = \pm 1$ and $\det\rho_{\lambda_0} = \chi_\triv$). Since $\sigma_2$ is odd, the eigenvalues of $\sigma_2(c)$ are $\{-1,1\}$, so the eigenvalues of $\sigma_3(c)$ must be $\{-1,1,1\}$. Taking $\lambda_0$ large enough, $\sigma_3$ is contained in a strictly compatible system by \Cref{lem:3d-irregular}, and the result follows from \Cref{prop:brauer-nesbitt}.
    \end{proof}

    \begin{lemma}\label{lem:2+2+1-trivial}
        Let $(\rho_\lambda)_\lambda$ be as in \Cref{thm:decomp}. Suppose that  for infinitely many primes $\lambda$, $\rho_{\lambda}$ is in case $5(i)$ of \Cref{table-cases}. Then $\rho_{\lambda}$ is in case $5(i)$ of \Cref{table-cases} for all $\lambda$ and \Cref{thm:decomp} holds for $(\rho_\lambda)_\lambda$.
    \end{lemma}

    \begin{proof}
        By assumption, we can choose $\lambda_0$ with arbitrarily large residue characteristic such that $\rho_{\lambda_0}\simeq \sigma_1 \+\sigma_2\+\chi_\triv$, where $\sigma_1$ has Hodge--Tate weights $\{-b, b\}$ and $\sigma_2$ has Hodge--Tate weights $\{-a,a\}$. Moreover, either $b>a$ or $a>0$, so in either case, $\sigma_1$ is Hodge--Tate regular.

        By \Cref{prop:odd}, taking $\lambda_0$ large enough, we can therefore assume that $\sigma_1$ is odd. Since $\det\rho_{\lambda_0} = \chi_\triv$, it follows that $\det\sigma_2\simeq\det\sigma_1\ii$, so $\sigma_2$ is also odd.

        Since $\rho_{\lambda_0}$ is self-dual, either both $\sigma_1$ and $\sigma_2$ are self-dual, or $\sigma_1\dual\simeq\sigma_2$.

        In the first case, for each $i$, we have
        \[\sigma_i\simeq\sigma_i\dual\tensor\det\sigma_i\simeq\sigma_i\tensor\det\sigma_i\]
        and since $\det\sigma_i$ is non-trivial, it follows that $\sigma_i$ is an induced representation. Hence, it is contained in a strictly compatible system by class field theory, and the result follows from \Cref{prop:brauer-nesbitt}.

        In the second case, we have $a = b >0$, and the result follows from  \Cref{lem:2d-3d-compatible} and \Cref{prop:brauer-nesbitt}.
    \end{proof}

    \begin{lemma}\label{lem:6}
        Let $(\rho_\lambda)_\lambda$ be as in \Cref{thm:decomp}. Suppose that  for infinitely many primes $\lambda$, $\rho_{\lambda}$ is in case $6$ of \Cref{table-cases}. Then $\rho_{\lambda}$ is in case $6$ of \Cref{table-cases} for all $\lambda$ and \Cref{thm:decomp} holds for $(\rho_\lambda)_\lambda$.
    \end{lemma}

    \begin{proof}
        By assumption, we can choose $\lambda_0$ with arbitrarily large residue characteristic such that $\rho_{\lambda_0}\simeq\sigma\+\sigma\dual\+\chi_\triv$, where $\sigma$ is irreducible and two-dimensional with Hodge--Tate weights $\{a,a\}$. In particular, $\rho_{\lambda_0}$ is not pure, so, by assumption, $\rho_{\lambda_0}(c)$ has eigenvalues $\{1,1,1,-1,-1\}$, where $c$ is a choice of complex conjugation. Since $\sigma(c)$ and $\sigma\dual(c)$ have the same eigenvalues, it follows that $\sigma$ is odd. The result follows from \Cref{lem:2d-irregular} and \Cref{prop:brauer-nesbitt}.
    \end{proof}

    \begin{lemma}\label{lem:4+1-trivial}
        Let $(\rho_\lambda)_\lambda$ be as in \Cref{thm:irreducible}. Suppose that for infinitely many primes $\lambda$, $\rho_{\lambda}$ is in case $2(i)$ of \Cref{table-cases}. Then $\rho_{\lambda}$ is reducible for all primes.

        If, moreover, $(\rho_\lambda)_\lambda$ satisfies the hypotheses of \Cref{thm:decomp}, then $\rho_{\lambda}$ is in case $2(i)$ for all primes $\lambda$ and \Cref{thm:decomp} holds for $(\rho_\lambda)_\lambda$.
    \end{lemma}

    \begin{proof}
        By assumption and by \Cref{lem:rl2+2}, we can choose $\lambda_0$ with arbitrarily large residue characteristic, such that
        \[r_{\lambda_0} \simeq \sigma_1\+\sigma_2,\]
    where $\sigma_1, \sigma_2$ are irreducible, two-dimensional representations, with Hodge--Tate weights $\{0, a+b\}$ and $\{a, b\}$,  $\det\sigma_1 = \det\sigma_2$, and
    \[\rho_{\lambda_0} \simeq \sigma_1\tensor\sigma_2\dual\+\chi_\triv.\]  
    Let $\sigma = \sigma_1\tensor\sigma_2\dual$. Since $\det\sigma_1 = \det\sigma_2$, $\sigma$ is valued in $\SO_4$. 
    
    To prove the first conclusion, by \Cref{prop:brauer-nesbitt}, it is sufficient to show that both $\sigma_1$ and $\sigma_2$ are contained in compatible systems, whence $\sigma_1\tensor\sigma_2\dual$ is too.

    If $\sigma_1$ is not Lie irreducible, then since it is Hodge--Tate regular, it is an induced representation, so it is contained in a compatible system by class field theory, and it is odd. If $\sigma_1$ is Lie irreducible, then $\Sym^2(\sigma_1)\tensor(\det\sigma_1)\ii$ is Lie irreducible. Since $\det\sigma_1=\det\sigma_2$, we have
    \[\wedge^2(\sigma_1\tensor\sigma_2\dual)\simeq \br{\Sym^2(\sigma_1)\tensor(\det\sigma_1)\ii}\+\br{\Sym^2(\sigma_2)\tensor(\det\sigma_2)\ii}.\]
    Thus, applying \Cref{thm:hui}$(ii)$ to the compatible system $(\wedge^2\rho_\lambda)_\lambda$, we may assume that $\Sym^2(\overline{\sigma}_1)|_{G_{\Q(\zeta_\l)}}$ is irreducible. Hence, by \cite{CG}*{Prop.~2.5}, $\sigma_1$ is odd. Taking $\lambda_0$ large enough, it follows from \cite{BLGGT}*{Thm.~C} that $\sigma_1$ is contained in a strictly compatible system of Galois representations.

    If $a\ne b$, the exact same arguments apply to $\sigma_2$. If $a=b$, then $\sigma_2$ is not Hodge--Tate regular. Since $\det\sigma_1 = \det\sigma_2$, $\sigma_2$ is still odd. Thus applying the same arguments as in \Cref{lem:2d-irregular}, $\sigma_2$ is contained in a strictly compatible system of Galois representations.

    We see that $\sigma$ is contained in a compatible system of four-dimensional Galois representations $(\sigma_\lambda)_\lambda$, and therefore, by \Cref{prop:brauer-nesbitt}, $\rho_{\lambda}$ is reducible for all primes.
    
    Now, assume that $(\rho_\lambda)_\lambda$ satisfies the hypotheses of \Cref{thm:decomp}. It remains to show that $\sigma_\lambda$ is irreducible for all but finitely many primes. If $a\ne b$, then this follows from \cite{hui-BLMS}*{Thm.~1.1}. In general, if $\sigma_\lambda$ is reducible for infinitely many primes, then $\rho_\lambda$ must be in case $5(i), 6, 7a, 7b(i)$ or $8$ for infinitely many primes, which gives a contradiction to the previous lemmas.
    \end{proof}

    \begin{remark}
        \Cref{lem:4+1-trivial} follows a very similar strategy to \cite{liu-yu}*{Thm.~1.0.1}, however, our assumptions are slightly different: we do not need to assume anything about the eigenvalues of $r_\lambda(c)$, with $c$ a complex conjugation, and we also allow the case that $r_\lambda$ has Hodge--Tate weights $\{0,0,b,b\}$.
    \end{remark}

    \begin{lemma}\label{lem:5ii}
        Let $(\rho_\lambda)_\lambda$ be as in \Cref{thm:decomp}. Suppose that for infinitely many primes $\lambda$, $\rho_{\lambda}$ is in case $5(ii)$ of \Cref{table-cases}.  Then $\rho_{\lambda}$ is in case $5(ii)$ of \Cref{table-cases} for all primes $\lambda$ and \Cref{thm:decomp} holds for $(\rho_\lambda)_\lambda$.
    \end{lemma}

    \begin{proof}
        By assumption, for some arbitrarily large prime $\lambda_0$, we can write $\rho_{\lambda_0}\simeq \sigma_1\+\sigma_2\+\chi_{K/\Q}$, where $\sigma_1, \sigma_2$ are irreducible two-dimensional representations with Hodge--Tate weights $\{-b,b\}$ and $\{-a,a\}$, and $\chi_{K/\Q}$ is the quadratic character corresponding to an extension $K/\Q$. Moreover, $r_{\lambda_0}\simeq \Ind_{K}^{\Q}\sigma$ for some irreducible representation $\sigma$ of $G_K$ with $\det\sigma\simeq\det\sigma^\tau$, where $\tau$ generates $\Gal(K/\Q)$.

        Taking $\lambda_0$ sufficiently large, by \Cref{lem:2d-3d-compatible}, we can assume that $\sigma_1$ is contained in a strictly compatible system. 
        If $a>0$ or if $\sigma_2$ is odd, then by \Cref{lem:2d-3d-compatible} or \Cref{lem:2d-irregular}, $\sigma_2$ is also contained in a compatible system, and the result follows from \Cref{prop:brauer-nesbitt}.

        So assume that $\sigma_2$ is even and that $a = 0$. Since $\rho_{\lambda_0}$ is self-dual, it follows that $\sigma_2$ is too. Hence,
        \[\sigma_2\simeq\sigma_2\dual\simeq\sigma_2\tensor\det\sigma_2\ii.\]
        If $\det\sigma_2\not\simeq\chi_\triv$, then $\sigma_2$ is an induced representation, so by class field theory, it is contained in an absolutely irreducible compatible system, and the result follows from \Cref{prop:brauer-nesbitt}.

        It remains to address the case that $\det\sigma_2\simeq\chi_{\triv}$. Since $\det\rho_{\lambda_0}\simeq\chi_\triv$, it follows that $\det\sigma_1 \simeq\chi_{K/\Q}$. Moreover, since  $\rho_{\lambda_0}$ is self-dual, it follows that $\sigma_1$ is too. Hence,
        \[\sigma_1\simeq\sigma_1\dual\simeq\sigma_1\tensor\chi_{K/\Q}\ii.\] 
        It follows that $\sigma_1$ is induced from a representation of $G_K$, and therefore the connected component of its algebraic monodromy group is abelian. Now, by \Cref{table:semisimple}, since $\rho_{\lambda}$ is $\SO_5$-valued, its algebraic monodromy group cannot be $\SL_2$ embedded via its standard two-dimensional representation. It follows that the derived subgroup of its connected monodromy group must be trivial. Thus $\sigma_2$ must be an Artin representation, whence it is contained in an absolutely irreducible strictly compatible system. The result follows from \Cref{prop:brauer-nesbitt}.
    \end{proof}

        \begin{lemma}\label{lem:2ii}
            Let $(\rho_\lambda)_\lambda$ be as in \Cref{thm:decomp}. Suppose that for infinitely many primes $\lambda$, $\rho_{\lambda}$ is in case $2(ii)$  of \Cref{table-cases}. Then $\rho_{\lambda}$ is in case $2(ii)$  of \Cref{table-cases} for all primes $\lambda$ and \Cref{thm:decomp} holds for $(\rho_\lambda)_\lambda$.
        \end{lemma}

        \begin{proof}
        By assumption, for some arbitrarily large prime $\lambda_0$, we can write $\rho_{\lambda_0} \simeq \sigma_4 \+\chi_{K/\Q}$, where $\sigma_4$ is irreducible and four-dimensional, $\chi_{K/\Q}$ is the quadratic character corresponding to an extension $K/\Q$, and $r_{\lambda_0}$ is induced from a representation of $G_K$. In this case, we have $(\G_{\lambda_0}^\circ)' = \SO_4$. It follows from \Cref{table:semisimple} and \Cref{thm:semisimple-rank} that for all primes $\lambda$, $\rho_\lambda$ is in one of cases $1, 2(i)$, or $2(ii)$ of \Cref{table-cases}.

        Since $r_{\lambda_0}$ is induced from $G_K$, for all but finitely many primes $p$ that are inert in $K$, we have $\Tr r_{\lambda_0}(\Frob_p) = 0$. Although the $r_\lambda$ do not necessarily form a compatible system, nevertheless, by \Cref{rem:trace-mod-powers}, we still have $\Tr r_\lambda(\Frob_p) =0$ for all $\lambda$ and for all but finitely many primes $p$ that are inert in $K$. Thus, for all primes $\lambda$, we have $r_\lambda\simeq r_\lambda\tensor\chi_{K/\Q}$. It follows that for all $\lambda$, the representation $r_\lambda$ cannot be Lie irreducible. Thus, $\rho_\lambda$ cannot be irreducible. Moreover, if $r_\lambda$ decomposes as a sum of two-dimensional representations $\sigma_1, \sigma_2$, then we must either have $\sigma_1\simeq\sigma_2\tensor\chi_{K/\Q}$ or $\sigma_i\simeq\sigma_i\tensor\chi_{K/\Q}$. In either case, $(\G_{\lambda}^\circ)'$ will not have the same semisimple rank as $\SO_4$. Thus $r_{\lambda}$ must be irreducible, but not Lie irreducible. It follows that $\rho_\lambda$ is in case $2(ii)$ of \Cref{table-cases} for all primes $\lambda$. 

        Since $r_\lambda\simeq r_\lambda\tensor\chi_{K/\Q}$, it follows that for all primes $\lambda$, we can write
        \[\rho_\lambda\simeq\sigma_{4, \lambda}\+\chi_{K/\Q}\]
        where $\sigma_{4,\lambda}$ is irreducible and four-dimensional. Since $\chi_{K/\Q}$ and $\rho_\lambda$ have Frobenius characteristic polynomials that are independent of $\lambda$, so does $\sigma_{4,\lambda}$. Hence, $(\sigma_{4,\lambda})_\lambda$ is a weakly compatible system. Moreover, since $\chi_{K/\Q}$ is contained in a strictly compatible system, and since $(\rho_\lambda)_\lambda$ is strictly compatible, it follows that $(\sigma_{4,\lambda})_\lambda$ must be too. The result follows from \Cref{prop:brauer-nesbitt}.
        \end{proof}

\begin{proof}[Proof of \Cref{thm:decomp}]
    \Cref{thm:decomp} now follows from the previous lemmas and \Cref{table-cases}.
\end{proof}

\subsection{The proof of \Cref{thm:irreducible}}

In this section we prove \Cref{thm:irreducible}. Case $(iii)$ was proven in \Cref{prop:non-self-dual}.

\subsubsection{The proof of \Cref{thm:irreducible}(ii)}

\begin{lemma}\label{lem:sym4}
    Suppose that for some prime $\lambda_0$, $(\G_{\lambda_0}^\circ)'\simeq\Sym^4(\SL_2)$. Then $\rho_\lambda$ is irreducible for all primes $\lambda$.
\end{lemma}

\begin{proof}
    We apply \cite{larsen-pink}*{Prop.~6.12}. The formal character of $\rho_{\lambda_0}$ is $\br{\begin{smallmatrix}\alpha^4 &&&\\&\alpha^2&&\\&&1&\\&&&\alpha^{-2}&\\&&&&\alpha^{-4}\end{smallmatrix}}$, and one can check that no other representation in \Cref{table:semisimple} has this formal character (cf.~\cite{Hui}*{Table 1}). Hence, $(\G_{\lambda}^\circ)'\simeq\Sym^4(\SL_2)$ for all $\lambda$ and the result follows.
\end{proof}

We deduce \Cref{thm:irreducible} in case $(ii)$, that each $\rho_\lambda$ is isomorphic to a representation valued in $\GO_5(\elb)$.

\begin{proof}[Proof of \Cref{thm:irreducible}$(ii)$]
    By assumption, $\rho_{\lambda_0}$ is irreducible for some prime $\lambda_0$. By \Cref{lem:lie-irred}, we may assume that $\rho_{\lambda_0}$ is Lie irreducible. Hence, by \Cref{table:semisimple}, we have $(\G_{\lambda_0}^\circ)'\simeq \SO_5$ or $\Sym^4(\SL_2)$. The latter case is handled by \Cref{lem:sym4}. 
        
    If $(\G_{\lambda_0}^\circ)'\simeq \SO_5$, then by \Cref{table:semisimple} and \Cref{thm:semisimple-rank}, for any prime $\lambda$, $\rho_\lambda$ can only be cases $1, 2(i), 2(ii)$ of \Cref{table-cases}. Case $2(ii)$ can be ruled out as in the proof of \Cref{lem:2ii}: if some $\rho_{\lambda}$ is in this case, then $r_\lambda$ is induced, so $\Tr r_\lambda(\Frob_p) = 0$ for a positive proportion of primes. But this property is independent of $\lambda$, meaning that $r_\lambda$ is never Lie irreducible (e.g.\ by \cite{patrikis-variations}*{Cor.~3.4.11}), and hence by \Cref{prop:rl-red,prop:rl-irred}, that $\rho_\lambda$ is never irreducible. Hence, the result follows from \Cref{lem:4+1-trivial}.
\end{proof}

\subsubsection{The proof of \Cref{thm:irreducible}$(i)$}
It remains to prove \Cref{thm:irreducible} in case $(i)$, where the compatible system $(\rho_\lambda)_\lambda$ is Hodge--Tate regular and self-dual up to twist, but the individual representations $\rho_\lambda$ are not assumed to be valued in $\GO_5(\elb)$.
As before, there is a compatible system of characters $(\chi_{\lambda})_\lambda$ such that $\rho_\lambda\simeq\rho_\lambda\dual\tensor\chi_\lambda$ for all $\lambda$ and replacing each $\rho_\lambda$ with $\rho_\lambda\tensor\det\rho_\lambda\ii\chi_\lambda^2$, we are free to assume, with no loss in generality, that each $\rho_\lambda$ is self-dual.

By assumption, $\rho_{\lambda_0}$ is irreducible and self-dual for some prime $\lambda_0$. It follows that $(\G_{\lambda_0}^\circ)'$ is either $\SO_5$ or $\Sym^4(\SL_2)$. The latter case was handled in \Cref{lem:sym4}. By \Cref{thm:semisimple-rank}, $(\G_{\lambda}^\circ)'$ has rank $2$ for all primes $\lambda$. Hence, it remains to show that $(\G_{\lambda}^\circ)'$ cannot be in cases $5$, $9$, and $10$ of \Cref{table:semisimple}, where $\rho_{\lambda}$ can be self-dual but not $\GO_5$-valued.

\begin{lemma}\label{lem:nonsd}
    Suppose that $\rho_\lambda$ is Hodge--Tate regular and irreducible for a positive density of primes $\lambda$. Then $\rho_\lambda$ is irreducible for all but finitely many primes.
\end{lemma}

\begin{proof}
   Since $\rho_\lambda$ is irreducible for a positive density of primes, by \Cref{psw-generalisation} (see also \cite{BLGGT}*{Prop.~5.3.2}), so is $\orho_{\lambda}|_{G_{\Q(\zeta_\l)}}$. Hence, by \cite{BLGGT}*{Thm.~C}, there is a finite, totally real extension $K/\Q$ such that the compatible system $(\rho_{\lambda}|_K)$ is automorphic. Thus, by \cite{Hui}*{Thm.~1.4}, $(\rho_{\lambda}|_K)$, and hence $\rho_\lambda$, is irreducible for all but finitely many primes.
\end{proof}

    \begin{proposition}\label{prop:symplectic}
        Let $(\rho_\lambda)_\lambda\:G_\Q\to\GL_5(\elb)$ be a strictly compatible system. Suppose that there is a set of primes $\lambda$ of positive Dirichlet density for which $\rho_\lambda$ contains a four-dimensional subrepresentation $\sigma$ such that
        \begin{itemize}
            \item $\sigma$ is isomorphic to a representation valued in $\Gf(\elb)$
            \item $\sigma$ is Hodge--Tate regular
            \item $\sigma$ is Lie irreducible.
        \end{itemize}
        Then $\rho_\lambda$ is reducible for all primes $\lambda$.
    \end{proposition}

    \begin{proof}
        We carry out the strategy sketched in \cite{Calegari-even1}*{Remark on p.~11}. Let $\std(\sigma)$ denote the $\SO_5(\elb)$-valued representation obtained via the exceptional isomorphism $\PGSp_4\cong \SO_5$. Then 
        \[\wedge^2(\sigma) \simeq(\std(\sigma)\tensor\simil(\sigma)) \+\simil(\sigma).\]
        Since $\sigma$ is Lie irreducible, so is $\std(\sigma)$. Moreover, since $\std(\sigma)\tensor\simil(\sigma)$ is a subrepresentation of $\wedge^2(\rho_\lambda)$, and since $(\wedge^2(\rho_\lambda))_\lambda$ is a compatible system, by \Cref{psw-generalisation}, we may assume that the residual representation $\std(\overline\sigma)|_{\Q(\zeta_\l)}$ is irreducible. Similarly, applying \Cref{psw-generalisation} to $(\rho_\lambda)_\lambda$, we may assume that $\overline{\sigma}|_{G_{\Q(\zeta_\l)}}$ is irreducible.

        Since $\std(\sigma)$ is irreducible and takes values in $\SO_5(\elb)$, it is automatically odd essentially self-dual, and it is Hodge--Tate regular since $\sigma$ is. Hence, by \cite{BLGGT}*{Thm.~C}, there is a finite totally real extension $F/\Q$ such that $\std(\sigma)|_F$ is automorphic. 
        
        Let $c$ be a complex conjugation. Since $F$ is totally real and $\std(\sigma)|_F$ is irreducible and automorphic, by \cite{Calegari-even1}*{Thm.~3.2}, we have $\Tr\std(\sigma)(c) = \pm1$. Since $\std\sigma$ is five-dimensional with trivial determinant, this trace must be $1$.

        Now, $\det\sigma = (\simil\sigma)^2$, so $\det\sigma(c) = 1$. Thus, the eigenvalues of $\sigma(c)$ are either all $1$, all $-1$, or $\{1,1,-1,-1\}$. It follows that 
        \[\Tr\br{\wedge^2\sigma(c)}\in\{6,-2\}.\]
        On the other hand, we have
        \[\Tr\br{\wedge^2\sigma(c)} = \simil\sigma(c)\Tr(\std(\sigma)(c)) +\simil\sigma(c) = 2\simil\sigma(c).\]
        It follows that $\simil\sigma(c) = -1$, i.e.\ that $\sigma$ is odd.

        We deduce from \cite{BLGGT}*{Thm.~C} that $\sigma$ is contained in a strictly compatible system $(\sigma_\lambda)_\lambda$. Its one-dimensional complement in $\rho_\lambda$ is also contained in a compatible system by class field theory. Hence, by \Cref{prop:brauer-nesbitt} $\rho_\lambda$ is reducible for every prime $\lambda$.
    \end{proof}

\begin{proof}[Proof of \Cref{thm:irreducible}$(i)$]
    By \Cref{lem:nonsd}, it is sufficient to prove that $\rho_\lambda$ is irreducible for a positive density of primes. Moreover, by \Cref{lem:lie-irred,lem:sym4}, we may assume that $\rho_{\lambda_0}$ is Lie irreducible and has algebraic monodromy group $\SO_5$.
    
    So suppose for contradiction that $\rho_\lambda$ is reducible for $100\%$ of primes. Then, by the proofs of \Cref{thm:irreducible}$(ii)$-$(iii)$, we may assume that for $100\%$ of primes $\lambda$, $\rho_\lambda$ is reducible, self-dual, and not isomorphic to a representation valued in $\GO_5(\elb)$. Thus, by \Cref{thm:semisimple-rank}, $\rho_\lambda$ must be in one of cases $5$, $9$, or $10$ of \Cref{table:semisimple}. Thus, either:
    \begin{itemize}
        \item For a positive density of primes, $(\G_{\lambda}^\circ)'$ is in case $5$ of \Cref{table:semisimple}. So
        \[\rho_\lambda \simeq \sigma_4\+\chi\]
        where $\chi$ is a character and $\sigma_4$ is a Lie irreducible $4$-dimensional Galois representation that takes values in $\Gf$.
        \item For a positive density of primes, $(\G_{\lambda}^\circ)'$ is in case $9$ of \Cref{table:semisimple}. So
        \[\rho_\lambda \simeq \sigma_3\+\sigma_2\]
        where $\sigma_3$ and $\sigma_2$ are irreducible, self-dual $3$- and $2$-dimensional representations.
        \item For a positive density of primes, $(\G_{\lambda}^\circ)'$ is in case $10$ of \Cref{table:semisimple}. So
        \[\rho_\lambda \simeq \sigma_2\+\sigma_2'\+\chi\]
        where $\chi$ is a quadratic character and $\sigma_2,\sigma_2'$ are distinct, irreducible, $2$-dimensional representations.
    \end{itemize}

    It follows from \Cref{lem:2d-3d-compatible,prop:symplectic} that $\rho_\lambda$ is reducible for all primes, contradicting our assumption that some $\rho_{\lambda_0}$ is irreducible.
\end{proof}

\section{Elliptic surfaces}
The goal of the next three sections is to classify the elliptic surfaces from \Cref{thm:Tate-intro} into four one-dimensional families and two isolated classes, and to construct a compatible system of Galois representations associated to each of these families. 

More precisely, using the theory of elliptic surfaces, we will show that 
\begin{enumerate}
    \item there is a lisse sheaf associated to the non-trivial part (\Cref{Def: transcendental_lisse}) of a family of elliptic surfaces (\Cref{Prop: lisse_sheaf}); and
    \item for certain elliptic surfaces coming from a degree $3$ branched cover of a special elliptic surface, the non-trivial part of their second \'etale cohomology satisfies certain desired properties (\Cref{Prop: dim_5_eg}). 
\end{enumerate}
In order to make sense of our terminology, we recall the classical geometric and arithmetic theory of elliptic surfaces in \Cref{Sect: preliminary_ell_sur}. Then, in \Cref{Sect: construction_of_system}, we focus on the decomposition of the second \'etale cohomology of a general elliptic surface over $\bP^1$. In particular, we describe its so-called \emph{non-trivial part} using perverse sheaf theory (see \Cref{Prop: decomposition_H^2}). When a family of elliptic surfaces is parametrised by an open subscheme of $\bP^1$, we show in \Cref{Prop: lisse_sheaf} that their non-trivial parts form a local system. Finally in \Cref{Sect: concrete_eg_5dim}, we focus on the branched covers of a specific elliptic surface, classify them into four one-dimensional families and two isolated $\Q$-isomorphism classes, and prove properties about the Galois representations associated to their non-trivial parts (see \Cref{Prop: dim_5_eg}).

\subsection{Preliminaries on elliptic surfaces}\label{Sect: preliminary_ell_sur}
We establish some necessary terminology and list some well-known results about the geometry and arithmetic of elliptic surfaces. More details can be found in \cite{Silverman2}*{III, IV} and \cite{Shioda-Schutt-MW-Lattice}*{Chap.5}. 

In this paper, unless otherwise specified, a variety is defined as a separated and integral scheme of finite type over a field. Specifically, a curve (respectively, a surface) refers to a one-dimensional (respectively, two-dimensional) smooth variety defined over the base field.

Let $K$ be a finitely generated field over $\bQ$ and $X$ be a projective smooth surface over $K$.

\begin{definition}\label{def:elliptic-surface}
    An \emph{elliptic fibration} of $X$ is a surjective $K$-morphism
\[\pi\: X\to C,	\]
where $C$ is a curve over $K$, such that:
\begin{enumerate}
        \item all the fibres of $\pi$ are connected; 
	\item almost all fibres are smooth of genus $1$; 
	\item no fibre contains a $(-1)$-curve;
	\item $\pi$ has a zero section $\iota \: C\to X$ such that $\pi\circ \iota={\rm id}_C$;
	\item $\pi$ has at least one singular fibre. In particular, $\pi$ does not induce an isomorphism between $X_{\overline{K}}$ and $E\times_{\overline{K}} C_{\overline{K}}$ for any elliptic curve $E$ over $\overline{K}$. 
\end{enumerate}
\end{definition}

\begin{definition}
    The \emph{Euler characteristic} of $X$ is the alternating sum of the dimensions of the cohomology of its structure sheaf, i.e.\  
\[
	\chi(X)= \sum_{i=0}^2 (-1)^i h^i(X, \mathcal{O}_X)
	=p_g(X)-q(X)+1,
\]
where $p_g(X)=h^2(X, \mathcal{O}_X)=h^0(X, \omega_{X/K})$ is the geometric genus of $X$, and $q(X)=h^1(X, \mathcal{O}_X)$ is the \emph{irregularity} of $X$. The \emph{topological Euler number} (or \emph{Euler–Poincar\'e characteristic}), is defined by 
\[
e(X)=\sum_{i=0}^4 (-1)^i h^i(X, \bC).
\]
\end{definition}

These two Euler characteristics are related by \emph{Noether's formula}:
\begin{equation*}
    12\chi(X)=e(X). 
\end{equation*}
Using the fibration $\pi\: X\to C$, one can compute $e(X)$ more concretely. For each geometric point $t\in C_{\overline{K}}$, with residue field $k_t$, let $X_{t}=X\times \Spec k_t$ denote the scheme theoretic preimage, with $m_t$ distinct irreducible components. Then the fibrewise Euler characteristic is defined by
\[
e(X_t)=\begin{cases}
	0 & \text{if }X_t \text{ is smooth}\\
	m_t & \text{if }X_t \text{ is multiplicative, i.e.\  of type }I_n\\
	m_t+1 & \text{if }X_t \text{ is additive, i.e.\  not of type }I_n.
\end{cases}	
\]
By \cite{Shioda-Schutt-MW-Lattice}*{Thm.~5.47}, we have
\begin{equation}\label{Eqn: Euler_fibre_formula}
	e(X)=\sum_{t\in C_{\overline{K}}} e(X_t).
\end{equation}
Since there are only finitely many bad fibres, this sum is finite.

In addition to the bad fibres, another source of algebraic cycles of $X_{\overline{K}}$ is the set of sections $\{f\: C_{\overline{K}}\to X_{\overline{K}}\mid \pi\circ f={\rm id}_C\}$. In fact, using the theory of Kodaira--N\'eron models, or more explicitly by \cite{Shioda-Schutt-MW-Lattice}*{Thm.~6.5}, we have a short exact sequence of $\bZ[G_K]$-modules
\[
0\to \Triv(X_{\overline{K}})\to \NS(X_{\overline{K}})\to \MW(X)\to 0. 
\]
Here:
\begin{itemize}[leftmargin=*]
    \item $\Triv(X_{\overline{K}})$ is the \emph{trivial lattice}\footnote{This terminology follows \cite{Shioda-Schutt-MW-Lattice}*{Sect.~6.1}}, generated by the zero section and all fibre components.
    \item The \emph{N\'eron--Severi group} $\NS(X_{\overline{K}})$ is the group of algebraic equivalence classes of 1-cycles, or equivalently, the image of the first Chern map $C^1({\rm Pic}(X_{\overline{K}}))={\rm Pic}(X_{\overline{K}})/{\rm Pic}^0(X_{\overline{K}})$.
    \item The \emph{Mordell--Weil group} $\MW(X)$ is the Mordell--Weil group of the \emph{generic fibre} $E$ of $X$ over the function field $\overline{K}(C)$ of the base curve $C$.
\end{itemize}

Moreover, after tensoring with $\bQ$, the above short exact sequence splits (\cite{Shioda-Schutt-MW-Lattice}*{Lems.~6.16, 6.17} or \cite{Silverman2}*{III, Prop.~8.3}). 

In this paper, we will only consider the rational N\'eron--Severi group and the rational trivial lattice for $X_{\overline{K}}$. Hence, to ease notation, we let $\NS^\circ(X)=\NS(X_{\overline{K}}) \otimes \bQ$, $\Triv(X)=\Triv(X_{\overline{K}})\otimes \bQ$, and $\MW^\circ(X)\coloneqq \MW(X)\otimes \bQ$.
Therefore, as $G_K$-modules, we have 
\begin{equation}\label{Eqn: decomposition_NS=Triv+E}
    \NS^\circ(X)\cong \Triv(X)\oplus \MW^{\circ}(X).
\end{equation}
In particular, if $\rho(X)=\dim \NS^\circ(X)$ and $r(E)=\rk \MW(X)$, then by the Shioda-Tate formula (\cite[Cor. 6.7]{Shioda-Schutt-MW-Lattice} for instance)
\begin{equation}\label{Eqn: Shioda_Tate_eqn}
    \rho(X)=r(E)+2+\sum_{\text{bad } t}(m_t-1).
\end{equation}
We call $\Tran_\ell(X) := (H^2_{\rm{\acute{e}t}}(X_{\overline{ K}}, \bQ_\l(1))/(\NS^{\circ}(X)\otimes \bQ_\l))^{\mathrm{ss}}$ the ($\ell$-adic) \emph{transcendental part} of $X$, where $(-)^\mathrm{ss}$ is the semisimplification. Using the fact that for elliptic surfaces, algebraic and numerical equivalence coincide \cite{Shioda-Schutt-MW-Lattice}*{Thm.~6.5}, up to semisimplification, we have an isomorphism of Galois representations
\begin{equation}\label{Eqn: decomposition_H^2_NS+Tran}
    H^2_{\rm{\acute{e}t}}(X_{\overline{ K}}, \bQ_\l(1))^{ss}\cong (\NS^\circ(X)\otimes \bQ_\l)\oplus \Tran_\ell(X)^{ss}. 
\end{equation}

\begin{lemma}
    When the base curve $C$ is $\bP^1$, for each geometric point $t\in \bP^1_{\overline{K}}$, denote 
	\[
	\delta_t=
	\begin{cases}
		0 & \text{if }X_t \text{ is smooth}\\
		1 & \text{if }X_t \text{ is multiplicative}\\
		2 & \text{if }X_t \text{ is additive}.\\
	\end{cases}
	\]
	Then 
	\begin{equation}\label{Eqn: non_triv_rank_P^1}
		\rank \Tran_\ell(X)+r(E)=\sum_{t} \delta_t-4 .
	\end{equation}
\end{lemma}
    	
\begin{proof}
Since $h^2(X_{\overline{K}}, \bQ)=\rank \Tran_\ell(X)+\rho(X)$, applying \eqref{Eqn: Euler_fibre_formula} and \eqref{Eqn: Shioda_Tate_eqn} we have 
	\[
	e(X)-2+h^1(X,\bC)+h^3(X,\bC)=h^2(X_{\overline{K}}, \bQ)= \rank \Tran_\ell(X)+ r(E)+2+ \sum_{\text{bad } t}(m_t-1).
	\]
Since $C=\bP^1$, the desired equality follows from the well-known comparison result $h^3(X, \bC)=h^1(X,\bC)=h^1(C,\bC)=0$ \cite{Shioda-Schutt-ES}*{Thm.~6.12}. 
\end{proof}

For further information about the transcendental part, for instance its Hodge decomposition, we will need a way to read the geometric genus of $X$ explicitly from its equation. Recall that as an elliptic curve over $K(C)$, we can express the generic fibre $E$ by a Weierstrass equation with the infinite point induced by the zero section. So we have 
\begin{equation}\label{Eqn: Weierstrass_model}
    E\: y^2+a_1xy+a_3y=x^3+a_2x^2+a_4x+a_6, \quad a_i\in K(C). 
\end{equation}
Since $C=\bP^1$, if $t$ is a parameter of $\P^1$, then we can assume that every $a_i$ is an element of $K[t]$. Up to a change of variables we can assume that \eqref{Eqn: Weierstrass_model} is \emph{globally minimal} so that it is minimal (in the sense  of \cite{Silverman1}*{VII.1}) when localised at every place (including $t=\infty$) of $\bP^1$. We define 
\begin{equation}\label{Eqn: Defn_a}
    a:= \min\{n\in \mathbb{N}\mid \deg(a_i)\leq ni,\ \forall i\}. 
\end{equation}

\begin{proposition}[\cite{Shioda-Schutt-MW-Lattice}*{\S5.13}]\label{Prop: genus_d_eqn}
    With respect to the globally minimal model of $E$ and the integer $a$ defined above, we have $e(X)=12a$, $\chi(X)=a$, and $p_g(X)=a-1$. 
\end{proposition}

\section{Galois representations attached to families of elliptic surfaces}\label{Sect: elliptic_surfaces}

We keep the notations and assumptions in the previous section except that in this section, we assume that the base curve $C=\bP^1$. In this situation, in \eqref{Eqn: decomposition_H^2}, we use the language of perverse sheaves to give an alternative description of the decomposition \eqref{Eqn: decomposition_H^2_NS+Tran} of the second \'etale cohomology of such elliptic surfaces.
This language will enable us to generalise these results to local systems in \Cref{Sect: local_system_family_ell_sur}.

\subsection{Construction of the compatible system of Galois representations}\label{Sect: construction_of_system}
Suppose that $K$ is a finitely generated field over $\Q$ with a fixed embedding $K\hookrightarrow \C$. Let $\pi\: X\rightarrow \P^{1}$ be an elliptic surface over $K$, as defined in \Cref{def:elliptic-surface}. After base changing to $\overline{K}$, we have $\pi_{\overline{K}}\: X_{\overline{K}}\rightarrow \P^{1}_{\overline{K}}$. 

Let $U$ be the maximal open subvariety of $\P_{\overline{K}}^{1}$ such that for each $u\in U$, $X_{\overline{K}, u} = \pi_{\overline{K}}^{-1}(u)$ is smooth. We denote this open embedding by $j\: U\hookrightarrow \P^{1}_{\overline{K}}$.  The complement closed subvariety $Z = \P^{1}_{\overline{K}}\setminus U$ is a disjoint union of finitely many points. We have the following pullback diagram:
\[
\begin{tikzcd}
    X_{\overline{K}, U}\ar[r]\ar[d, "\pi_{\overline{K}, U}"'] & X_{\overline{K}} \ar[d, "\pi_{\overline{K}}"] \\
    U\ar[r, hook, "j", "\text{open}"'] & \P^{1}_{\overline{K}}.
\end{tikzcd}
\]
According to our definitions, $\pi_{\overline{K}, U}$ is proper and smooth.
\begin{lemma}\label{lemma:R1}
Let $F^{i}\: R^{i}\pi_{\overline{K}, \ast}\overline{\Q}_{\ell}\rightarrow j_{\ast}R^{i} \pi_{\overline{K}, U, \ast}\overline{\Q}_{\ell}$ be the natural adjunction map. Then
\begin{enumerate}
    \item $F^{0}$ is an isomorphism and $R^{0}\pi_{\overline{K}, \ast}\overline{\Q}_{\ell}\cong \overline{\Q}_{\ell}$.
    \item $F^{1}$ is an isomorphism.
\end{enumerate}

\end{lemma}
\begin{proof}
    For $(i)$, it is enough to note that $\pi$ has connected fibres.

    For $(ii)$, a similar result for a complex elliptic surface has been proven in \cite{Cox-Zucker-elliptic-surface}*{Lem.~1.2}. It suffices to prove that for $t\in Z$, the composition of maps
    \[
    H_{\mathrm{\acute{e}t}}^{1}(X_{\overline{K}, t}, \overline{\Q}_{\ell})\xrightarrow{\sim} (R^{1}\pi_{\overline{K}, \ast}\overline{\Q}_{\ell})_{t}\xrightarrow{F_{t}^{1}} (j_{\ast}R^{1}\pi_{\overline{K}, U, \ast}\overline{\Q}_{\ell})_{t} \xrightarrow{\sim} H^{1}_{\mathrm{\acute{e}t}}(X_{\overline{K}, \overline\eta}, \overline{\Q}_{\ell})^{I_t}
    \]
    is an isomorphism. Here, $\overline{\eta} $ is a geometric generic point of
$\mathbb P^{1}_{\overline{K}}$ and $I_{t}$ is the inertia group
at $t$.
    
    By the local invariant cycle theorem proved by Deligne (see \cite{Deligne-Weil-ii}*{Th\'eor\`eme~3.6.1}), $F_{t}^{1}$ is a surjective map. Using the classification of bad fibres of an elliptic surface and the local monodromy (see \cite{Kodaira-compact-analytic-surfaces-ii}), the source and target of $F_t^1$ are either both $0$ or both $\overline{\Q}_{\ell}$. Hence, $F_{t}^{1}$ is an isomorphism. 
\end{proof}

The following is a standard application of perverse sheaves. We refer the reader to \cite{BBD} (see also \cite{Sun-Zheng-intersection}) for further background.
\begin{proposition}\label{Prop: decomposition_H^2}
There is a $G_{K}$-equivariant decomposition 
\begin{equation}\label{Eqn: decomposition_H^2}
    H^{2}_{\mathrm{\acute{e}t}}(X_{\overline{K}}, \overline{\Q}_{\ell})\cong H^{2}_{\mathrm{\acute{e}t}}(\P^{1}_{\overline{K}}, \overline{\Q}_{\ell}) \oplus H^{1}_{\mathrm{\acute{e}t}}(\P^{1}_{\overline{K}}, R^{1}\pi_{\overline{K},\ast}\overline{\Q}_{\ell}) \oplus H^{0}_{\mathrm{\acute{e}t}}(\P^{1}_{\overline{K}}, R^{2}\pi_{\overline{K},\ast}\overline{\Q}_{\ell}).
\end{equation}
Moreover, 
\[
H^{0}_{\mathrm{\acute{e}t}}(\P^{1}_{\overline{K}}, R^{2}\pi_{\overline{K}, \ast}\overline{\Q}_{\ell}) \cong \overline{\Q}_{\ell}(-1) \oplus \bigoplus_{t\in Z} \overline{\Q}_{\ell}(-1)^{m_t - 1},
\]
where $m_{t}$ is the number of irreducible components of the bad fibre $X_{\overline{K}, t}$.
\end{proposition}
\begin{proof}
By \cite{Sun-Zheng-intersection}*{Thm.~1.8} (a refinement of the decomposition theorem in \cite{BBD}), there is a $G_{K}$-equivariant decomposition
\[
R\pi_{\overline{K}, \ast}\overline{\Q}_{\ell}[2] 
\cong \prescript{p}{}{\mathcal{H}}^{-1} (R\pi_{\overline{K}, \ast}\overline{\Q}_{\ell}[2])[1] \oplus \prescript{p}{}{\mathcal{H}}^{0} (R\pi_{\overline{K}, \ast}\overline{\Q}_{\ell}[2]) \oplus \prescript{p}{}{\mathcal{H}}^{1} (R\pi_{\overline{K}, \ast}\overline{\Q}_{\ell}[2])[-1].
\]
Note that the stratification of $\pi_{\overline{K}}$ is given by $\P^1_{\overline{K}} = U\sqcup Z$. By restricting this sum to $U$ and applying Deligne’s theorem \cite{Deligne-smooth-decomposition}, we have a decomposition, 
\[
\prescript{p}{}{\mathcal{H}}^{m} (R\pi_{\overline{K}, \ast}\overline{\Q}_{\ell}[2]) \cong 
j_{\ast}R^{m+1}\pi_{\overline{K}, U, \ast}\overline{\Q}_{\ell}[1] \oplus \mathcal{G}^{m},
\] 
where for each  $m\in \{-1,0,1\}$, $\mathcal{G}^{m}$ is a sheaf supported on $Z$. Moreover, by \cite{Sun-Zheng-intersection}*{Lem.~2.2.8}, this decomposition is Galois equivariant. By \Cref{lemma:R1} and the hard Lefschetz theorem, $j_{\ast}R^0 \pi_{\overline{K}, U, \ast}\overline{\Q}_{\ell}\cong \overline{\Q}_{\ell}$, $j_{\ast}R^2 \pi_{\overline{K}, U, \ast}\overline{\Q}_{\ell}\cong \overline{\Q}_{\ell}(-1)$, and $j_{\ast}R^1 \pi_{\overline{K}, U, \ast}\overline{\Q}_{\ell}\cong R^{1}\pi_{\overline{K}, \ast}\overline{\Q}_{\ell}$. So we have a Galois equivariant decomposition
\[
R\pi_{\overline{K}, \ast}\overline{\Q}_{\ell}[2] \cong (\overline{\Q}_{\ell}[2]\oplus \mathcal{G}^{-1}[1]) \oplus (R^{1}\pi_{\overline{K}, \ast}\overline{\Q}_{\ell}[1] \oplus \mathcal{G}^{0})\oplus (\overline{\Q}_{\ell}(-1) \oplus \mathcal{G}^{1}[-1]).
\]
Taking the cohomology sheaves and pulling back to $Z$, we obtain that $\mathcal{G}^{-1}\cong \mathcal{G}^{1}\cong 0$ and $\mathcal{G}^{0} \cong \bigoplus_{t\in Z} i_{t, \ast} \overline{\Q}_{\ell}(-1)^{m_t-1}$, where $m_t$ is the number of irreducible components of the bad fibre $X_{\overline{K}, t}$. 

In summary, we have 
\[
R\pi_{\overline{K}, \ast}\overline{\Q}_{\ell}[2] \cong \overline{\Q}_{\ell}[2] \oplus R^1 \pi_{\overline{K}, \ast}\overline{\Q}_{\ell}[1] \oplus \bigoplus_{t\in Z} i_{t, \ast} \overline{\Q}_{\ell}(-1)^{m_t-1} \oplus \overline{\Q}_{\ell}(-1).
\]
Taking $R^0\Gamma(\P^{1}_{\overline{K}}, -)$ on both sides proves the proposition.
\end{proof}

Comparing \Cref{Prop: decomposition_H^2} with the terminology recalled in the previous section, we see that up to semisimplification, 
\[
(H^{2}_{\mathrm{\acute{e}t}}(\P^{1}_{\overline{K}}, \overline{\Q}_{\ell}) \oplus H^{0}_{\mathrm{\acute{e}t}}(\P^{1}_{\overline{K}}, R^{2}\pi_{\overline{K}, \ast}\overline{\Q}_{\ell}) )(1)\cong\Triv(X)\otimes \overline{\Q}_{\ell}
\]
and 
\[
 H^{1}_{\mathrm{\acute{e}t}}(\P^{1}_{\overline{K}}, R^{1}\pi_{\overline{K}, \ast}\overline{\Q}_{\ell})(1)\cong(\Tran_\ell(X) \otimes \overline{\Q}_{\ell})\oplus (\MW(X)\otimes \overline{\Q}_\ell).
\]
\begin{definition}\label{Defn: non-trivial_part}
Let $\pi\: X\rightarrow \P^1$ be an elliptic surface over $K$. We call $H^{1}_{\mathrm{\acute{e}t}}(\P^{1}_{\overline{K}}, R^{1}\pi_{\overline{K}, \ast}\overline{\Q}_{\ell})(1)$ the \emph{non-trivial part} of (the second \'etale cohomology of) $X$. Define
\[
\rho_{\ell}(X)\: G_{K}\rightarrow \GL( (H^{1}_{\mathrm{\acute{e}t}}(\P_{\overline{K}}^{1}, R^1\pi_{\overline{K}, \ast}\overline{\Q}_{\ell})(1) )^{\mathrm{ss}
    })
\]
to be the semisimplification of the induced Galois representation.
\end{definition}

\begin{corollary}\label{Cor: compatible_system_at_closed_points}
    If $K$ is a number field, then $(\rho_{\ell}(X))_{\ell}$ is a weakly compatible system of $\Q$-rational Galois representations.
\end{corollary}
\begin{proof}
One can check that as the rational prime $\ell$ varies, each of $H^{2}_{\mathrm{\acute{e}t}}(X_{\overline{K}}, \overline{\Q}_{\ell})$, $ H^{2}_{\mathrm{\acute{e}t}}(\P^{1}_{\overline{K}}, \overline{\Q}_{\ell})$, and $ H^{0}_{\mathrm{\acute{e}t}}(\P^{1}_{\overline{K}}, R^{2}\pi_{\overline{K}, \ast}\overline{\Q}_{\ell})$ induces a weakly compatible system of $\Q$-rational Galois representations. The compatibility of $(\rho_\l(X))_{\ell}$ is then a direct consequence of \Cref{Prop: decomposition_H^2}.
\end{proof}

\subsection{A local system from a family of elliptic surfaces}\label{Sect: local_system_family_ell_sur}
In this section, we construct a lisse $\ell$-adic sheaf associated to the non-trivial parts of a family of elliptic surfaces.

In the classical setting, given a connected topological space $B$, a locally constant sheaf of finite-dimensional spaces over $B$, or equivalently, a local system, induces a representation of the fundamental group $\pi_1(B,s)$, with $s$ an arbitrary point of $B$. Moreover, every such representation arises in this way. In other words, the correspondence that assigns a local system $\mathcal{L}$ to the representation $\rho_s'\: \pi_1(B,s)\to \GL(\mathcal{L}_s)$ is an equivalence between the two corresponding categories. 

Now we switch to the \'etale topology. Let $B$ be a normal connected scheme of dimension one over a finitely generated field $k$ of characteristic zero. Let $\eta$ denote the generic point of $B$, and let $s$ be a special point of $B$ (we can choose $s$ to be rational over $k$ or not). Then fix geometric points $\bar{\eta}\to \eta$ and $\bar{s}\to s$\footnote{Here we can assume that the geometric points correspond to the separable closure of the residue field of the related points. } and write 
\[G:=\Gal(\kappa(\bar{\eta})/\kappa(\eta))=\Gal(\overline{\kappa({\eta})}/\kappa(\eta))=\pi_{1}^{\mathrm{\acute{e}t}}(\eta, \bar{\eta}),\]
and 
\[G_s:=\Gal(\kappa(\bar{s})/\kappa(s))=\Gal(\overline{\kappa({s})}/\kappa(s))=D_s/I_s,\]
where $D_s$ and $I_s$ are the decomposition and inertia groups of $\bar{s}/s$. 

Let $\cF$ be a lisse $\overline{\Q}_{\ell}$-sheaf of $B$, and let
\[
M:=\cF_{\bar{\eta}}, \quad M_s:=\cF_{\bar{s}}
\]
 be the corresponding stalks. Then there is a natural \emph{cospecialisation} map $\phi_s\: M_s\to M^{I_s}$ induced by the map $\cF(\mathcal{O}_{B,\bar{s}})\to \cF(\kappa(\bar{\eta}))$. 
 
 Now consider the data $\mathcal{M}(B):=\{M, \{M_s, \phi_s: s\in B\}\}$. One can verify the following properties (\cite{MilEC}*{II.~3.16 and V.~1.3}): 
\begin{enumerate}\label{List: locally_constant_data}
    \item \label{enum:1} $M$ (resp.\ $M_s$) has finite dimensional $\overline{\Q}_{\ell}
    $-stalk and is a continuous $G$ (resp.\ $G_s$) module (for each $s$).
    \item \label{enum:2} For every $s$,  $\phi_s$ is compatible with respect to the actions of $G_s$ and $G$.
    \item \label{enum:3} There exists an open subscheme $U\subset B$ such that for every point $s\in U$, $I_s$ acts trivially on $M$ and $\phi_s\: M_s\to M$ is an isomorphism. 
\end{enumerate}
Conversely any data $\mathcal{M}(B)$ satisfying the above properties induces a locally constant sheaf over $U$.  

Now we move to a more concrete construction. Let $\mathcal{C}:=\mathbb{P}^1_B$ be the projective space over $B$, with structure morphism $g\:\mathcal{C}\to B$. Let $\mathcal{X}$ be a projective scheme of finite type over $B$, with smooth structure morphism $f\: \mathcal{X}\to B$ which factors through $g$, i.e.\  we have the following commutative diagram 
\[
\begin{tikzcd}
    \mathcal{X} \arrow[dr, "h"] \arrow[dd, "f"']& \\
    & \mathcal{C} \arrow[dl, "g"]\\
    B &
\end{tikzcd}
\]
so that $f=g\circ h$. Moreover, we assume that $h$ is proper with a section, and has genus one fibres of dimension one. 

In down-to-earth terms, this is saying that both $\mathcal{X}$ and $\mathcal{C}$ can be considered as families of varieties parametrised by $s\in B$, and that for each given $s\in B$, $\mathcal{X}_s$ is an elliptic surface over the residue field $\kappa(s)$, and is parametrised by $\mathcal{C}_s=\mathbb{P}^1_{\kappa(s)}$ via $h_s$.

It is well known that $R^2f_{ \ast} \overline{\Q}_{\ell}$ is a lisse sheaf and thus induces a $\pi^{\rm{\acute{e}t}}_{1}(B, \bar{s})$-action via its fibre $H_{\rm{\acute{e}t}}^2(\mathcal{X}_{\bar{s}}, \overline{\Q}_{\ell})$. Similarly, for the generic point $\eta$, $H_{\rm{\acute{e}t}}^2(\mathcal{X}_{\bar{\eta}}, \overline{\Q}_{\ell})$ is a representation of $G=\pi^{\rm{\acute{e}t}}_{1}(\eta, \bar{\eta})=\Gal(\kappa(\bar \eta)/\kappa(\eta))$. Furthermore, the above construction shows that we can consider $\mathcal{X}_\eta$ as an elliptic surface over the function field $\kappa(\eta)$ of $B$, and we have a corresponding commutative diagram

\[
\begin{tikzcd}
    \mathcal{X}_{\eta} \arrow[dr, "h_{\eta}"] \arrow[dd, "f_{\eta}"']& \\
    & \mathcal{C}_{\eta} \arrow[dl, "g_{\eta}"]\\
    {\rm Spec}(\kappa(\eta)) &
\end{tikzcd}.
\]

Using the direct sum \eqref{Eqn: decomposition_H^2}, and letting $N\coloneqq H^1_{\rm{\acute{e}t}}(\mathcal{C}_{\bar{\eta}}, R^1 h_{\overline{\eta}, \ast}\overline{\bQ}_\ell)$, we see that $N$ is an orthogonal complement of $L\coloneqq H^2_{\rm{\acute{e}t}}(\mathcal{C}_{\bar{\eta}}, \overline{\bQ}_\ell)\oplus H^0_{\rm{\acute{e}t}}(\mathcal{C}_{\bar{\eta}}, R^2 h_{\overline{\eta}, \ast} \overline{\bQ}_\ell)$. 
So now $N$, as well as $L$, are subrepresentations of $G=\pi^{\rm{\acute{e}t}}_{1}(\eta, \bar{\eta})$. It is worth remarking that geometrically, $L$ can be described as the subspace generated by the $\kappa(\bar{\eta})$-irreducible components of the singular fibres of $h_{\eta}$ and the trivial section of the fibration $h_\eta\: \mathcal{X}_\eta\to \mathcal{C}_\eta$. 

Similarly, for each special point $s\in B$, we have the corresponding representations $N_s$ and $L_s$ of $G_s=\Gal(\kappa(\bar{s})/\kappa(s))$. Moreover, in an open subscheme $U$ of $B$, for each $s\in U$, the specialisation map of cocycles/Weil divisors induces an isomorphism of lattices $\phi_s\: L_s\to L$\footnote{To see this, recall that the irreducible components over the generic fibre come from the blowup of the singular point of the generic Weierstrass model. So in an open subscheme $U\subset B$, the Kodaira type of each fibre does not change.}.  Now, replacing $B$ with $U$ if necessary, we have a data $\mathcal{M}:=\{L, \{L_s, \phi_s: s\in B\}\}$. 

\begin{proposition}\label{Prop: lisse_sheaf}
    Replacing $B$ with an open subscheme $U$ if necessary, the data $\mathcal{M}:=\{L, \{L_s, \phi_s: s\in B\}\}$ satisfies properties \ref{enum:1}, \ref{enum:2}, and \ref{enum:3} above. In particular, it induces a lisse $\overline{\Q}_{\ell}$-subsheaf $\mathcal{L}$ of $R^2f_{ \ast}\overline{\Q}_{\ell}$.
\end{proposition}
\begin{proof}
    Property $(i)$ is obvious due to our construction. As for $(ii)$ and $(iii)$, the maps $\phi_s$ are built above. To show that they are compatible with the Galois actions, it suffices to show that $L$ descends to a representation of $\pi^{\rm{\acute{e}t}}_{1}(U, \bar{\eta})$, or equivalently, to show that for all but finitely many $s\in B$, $I_s$ acts trivially on $L$. But this is true since every divisor $\mathcal{Y}_{\eta}$ generating $L$ is defined over a certain finite extension over $\kappa(\eta)$. Letting $\mathcal{Y}_{\eta}$ run over all generators of $L$, we can find a finite extension $F/\kappa(\eta)$ such that every $\mathcal{Y}$ (and hence $L(1)$) admits a trivial action of $\Gal(\kappa(\bar{\eta})/F)$. Since the characteristic of $k$ is zero, $F/\kappa(\eta)$ has only finite ramification. So taking an open subscheme $U$ to avoid those ramified places if necessary, one can descend $L$ as a $\pi^{\rm{\acute{e}t}}_{1}(U, \bar{\eta})$-representation. Finally, replacing $B$ with $U$, the data $\mathcal{M}$ induces a locally constant sheaf $\mathcal{L}$, as required.  
\end{proof}

\begin{definition}\label{Def: transcendental_lisse}
We define the quotient lisse $\overline{\Q}_{\ell}$-sheaf
\[
\operatorname{NTriv}_{\ell}(\mathcal{X}):= (R^2f_{\ast}\overline{\Q}_{\ell} / \mathcal{L})(1).
\]
\end{definition}
Suppose that $b\in B$ is a closed point. Let $\overline{b}$ be a geometric point over $b$. Then
\[
\operatorname{NTriv}_{\ell}(\mathcal{X})_{\overline{b}} \cong (R^2f_{\ast}\overline{\Q}_{\ell} / \mathcal{L})_{\overline{b}}(1) \cong ((R^2f_{\ast}\overline{\Q}_{\ell})_{\overline{b}} / \mathcal{L}_{\overline{b}})(1) \cong H^{1}_{\mathrm{\acute{e}t}}(\P^{1}_{\overline{k}},
R^{1}h_{\overline{b},\ast}\overline{\Q}_{\ell})(1).
\]
By \Cref{Cor: compatible_system_at_closed_points}, $(\operatorname{NTriv}_{\ell}(\mathcal{X})^{\mathrm{ss}}_{\overline{b}})_{\ell}$ forms a weakly compatible system.

\section{Covers of a fixed elliptic surface}\label{Sect: concrete_eg_5dim}

In this section, we specialise the results of the last two sections to the case of covers of a fixed elliptic surface over the field of rational numbers.

Let $X_0$  be an elliptic surface over base curve $\bP^1$ and assume that $X_0$ is \emph{rational}, that is, as a surface over $\Q$, $X_0$ is birationally equivalent to the projective plane $\bP^2$. Then the N\'eron--Severi group $\NS^\circ(X_{0, \overline{\Q}})$ has dimension $10$ \cite{Shioda-Schutt-MW-Lattice}*{Prop.~7.1}. Assume that
\begin{equation}\label{Assumption: only_mul_fibre}
    \text{all singular fibres of $X_0$ are of multiplicative type}. 
\end{equation}
Using the Kodaira symbol (i.e.\ $I_n$ with the subscript $n$ indicating the number of irreducible components), and using $d_n$ to denote the number of $I_n$ fibres of $X_0$, we have $\sum_n d_n(n-1)+2\le10$ (note that this is a finite sum). 

Now we consider a $\Q$-morphism $\varphi\: \bP^1\to \bP^1$, and let $X$ be the pull-back of $X_0$ via this morphism, i.e.\  $X$ fits into the following commutative diagram
\[
\begin{tikzcd}
	X\coloneqq \bP^1_s\times_{\bP^1_t} X_0 \arrow[r] \arrow[d] & X_0 \arrow[d] \\
	\bP^1_s  \arrow[r, "\varphi"] & \bP^1_t.
\end{tikzcd}
\]

\begin{lemma}\label{Lem: dim_genus_formula}
    Let $X_0$, $X$ and $\varphi$ be as above. Assume that the degree of $\varphi$ is $d$ and let $a$ be the constant defined in \eqref{Eqn: Defn_a} for a globally minimal model of $X_0$. Then $p_g(X)=da-1$. Moreover, the compatible system $(\rho_{\ell}(X))_{\ell}$ $($see \Cref{Defn: non-trivial_part} and \Cref{Cor: compatible_system_at_closed_points}$)$ associated to $X$ has dimension 
    \[
    d\sum_n d_n-\sum_{s\in \varphi^{-1}(\Sigma_0)}(\varepsilon_s-1)-4,
    \]
    where $\Sigma_0\subset \bP^1_t$ is the set of points over which $X_0$ has singular fibres. Here, for each geometric point $s\in \bP^1$, $\varepsilon_s$ is the ramification index with respect to $\varphi$. 
\end{lemma}
\begin{proof}
    If we start with a globally minimal model of $X_0$, then due to Assumption \eqref{Assumption: only_mul_fibre} we know its pullback model is again globally minimal \cite{Shioda-Schutt-MW-Lattice}*{p.104-105}. Then \Cref{Prop: genus_d_eqn} tells us that $p_g(X)=da-1$. This proves the first statement. 

    For the second statement, observe that if the fibre $X_{0,t_0}$ is of type $I_n$, and if $s_0$ is a place lying above $t_0$ (via $\varphi$) with ramification index $\varepsilon_{s_0}$, then the pullback fibre $X_{s_0}$ is of type $I_{n\varepsilon_{s_0}}$. By \eqref{Eqn: non_triv_rank_P^1}, the dimension of $\rho_{\ell}(X)$ depends only on the number and types of the singular fibres. Since the pullback model contains only multiplicative-type singular fibres, we have
    $$
    \dim \rho_{\ell}(X)=\rank \Tran_\ell(X)+r(E)=\#\{\text{singular fibres}\}-4.
    $$
    Furthermore, a singular fibre of type $I_n$ on $X_0$ at a point $t_0$ lifts to exactly $d-\sum_{\varphi(s)=t_0} (\varepsilon_s-1)$ singular fibres on $X$, where $s$ runs over the preimages of $t_0$ under $\varphi$. Summing over all bad fibres, the total number of singular fibres on $X$ is exactly $d\sum_n d_n-\sum_{s\in \varphi^{-1}(\Sigma_0)}(\varepsilon_s-1)$. The dimension formula immediately follows.
\end{proof}

\subsection{The No.\ 63 elliptic surface}\label{Sect: 63_curve}

Recall from the introduction that $X_0$ is the No.\ $63$ elliptic surface of \cite{Shioda-Schutt-MW-Lattice}*{Table~8.3} induced by the Weierstrass equation
\begin{equation}\label{Eqn: No.63_ell_curve}
    X_0\: y^2 + (t+3)xy + y = x^3.
\end{equation}

This surface is rational, and it has $I_1$ fibre at $t=0$ and the roots of $t^2+9t+27=0$, and $I_9$ fibre at $t=\infty$. Moreover, it is extremal in the sense that its generic fibre has Mordell--Weil rank zero.

By \Cref{Lem: dim_genus_formula}, for a degree $3$ cover, the non-trivial
part (\Cref{Defn: non-trivial_part}) has dimension $5$ precisely when $\sum_{s\in\varphi^{-1}(\Sigma_0)}(\varepsilon_s-1)=3$. 
In \Cref{Tab: ramification_type}, we list the possible ramification types
over the singular fibres of $X_0$ satisfying this condition.

Each row of the table displays a possible ramification type. By writing $\varepsilon_1+\cdots +\varepsilon_r$ in the cell, we mean that the corresponding bad fibre splits into $r$ fibres in $X$, with ramification indices given by $\varepsilon_i$. For instance $(1+2)+(1+2)$ in the third cell of row three means that in case (3), for each root of $t^2+9t+27=0$, the corresponding bad fibre splits into two fibres in $X$, with ramification indices $1$ and $2$ respectively. 

\begin{proposition}\label{Prop: dim_5_eg}
    Let $X_0$ be as in \eqref{Eqn: No.63_ell_curve}. For each degree $3$ morphism $\varphi\: X\rightarrow X_0$ with any of the ramification types of \Cref{Tab: ramification_type}, the compatible system $(\rho_\ell (X))_\ell$ is $5$-dimensional and is self-dual with Hodge--Tate weights $\{-1,-1,0,1,1\}$.  
\end{proposition}
\begin{proof}
    According to the paragraph above this proposition, only the self-duality needs a proof. For this, notice that the genus of $X$ is two. Thus the transcendental part has rank at least four. If it has rank five, then there is nothing to explain. Otherwise, the non-trivial part will be a direct sum of the transcendental part and a finite character and the proof is reduced to showing that the character is quadratic. But this is not hard since this character comes from the algebraic part, i.e.\ it comes from a Galois representation with $\bZ$-coefficients. Finally, by a classical comparison theorem proved by Faltings \cite{Faltings-p-adic}, the Hodge--Tate weights of $X$ can be read from the Hodge numbers. 
\end{proof}

\begin{table}[h!]
\centering
    \begin{tabular}{|c|c|c|c|} \hline 
  No. &  fibres above $t=0$ &  fibres above $t=\infty$ &  fibres above roots of  $t^2+9t+27=0$  \\ \hline 
    (1) &  3 & 1+2 & (1+1+1)+(1+1+1)  \\ \hline 
    (2) &  1+2 & 3 & (1+1+1)+(1+1+1)  \\ \hline 
    (3) & 1+1+1 & 1+2 & (1+2)+(1+2)  \\ \hline 
    (4) & 1+2 & 1+1+1 & (1+2)+(1+2)  \\ \hline 
\end{tabular}
\caption{Possible ramification types of genus $2$, degree $3$ branched covers of $X_0$}
\label{Tab: ramification_type}
\end{table}

In the following, for each case, we write down all corresponding elliptic surfaces up to $\bQ$-isomorphism. 
\subsubsection{Case $(1)$}\label{case(1)_family}

In this case, we can write the Weierstrass equation of $X_0$ as 
\[
y^2+(1+3v)xy+v^3y=x^3
\]
with $t=1/v$. To determine the map $s\mapsto \varphi(s)=v$, up to an action of $\PGL_2(\bQ)$ we can assume:
\begin{enumerate}
    \item the unique ramified fibre lying above $v=0$ (i.e.\  $t=\infty$) is at $s=0$;
    \item the unique unramified fibre lying above $v=0$ is at $s=1$;
    \item the unique fibre lying above $v=\infty$ (i.e.\  $t=0$) is at $s=\infty$. 
\end{enumerate}
Hence, we have 
\[
v=\varphi(s)=bs^2(s-1),
\]
with $b\in \bQ\t$. Thus, we have a family of elliptic surfaces (parametrised by $b$)
\[
\mathcal{X}_b\: y^2+(1+3bs^2(s-1))xy+b^3s^6(s-1)^3y=x^3.
\]

\subsubsection{Case $(2)$}\label{case(2)_family}
In this case, up to a transformation by ${\rm PGL}_2(\bQ)$, we need a map $\varphi(s)=t$ such that 
\begin{enumerate}
    \item the unique fibre lying above $t=\infty$ is at $s=\infty$; 
    \item the unique ramified fibre lying above $t=0$ is at $s=0$;
    \item the unique unramified fibre lying above $t=0$ is at $s=1$.
\end{enumerate}
As in case $(1)$, we have that 
\[
t=\varphi(s)=bs^2(s-1)
\]
with $b$ a parameter. Thus, the family of elliptic surfaces is 
\[
\mathcal{X}_b: y^2+(bs^2(s-1)+3)xy+y=x^3. 
\]

\subsubsection{Case $(3)$}\label{case(3)_family}

In this case, up to a transformation by $\PGL_2(\bQ)$, we can require the map $\varphi(s)=t$ to be such that the only ramified lifting of $t=\infty$ is at $s=\infty$. 
Thus, we can assume that
\[
t=\varphi(s)=\frac{a_3s^3+a_2s^2+a_1s+a_0}{s-b_0}
\]
with $a_3\neq 0$. By translation, we can further assume that $a_2 = 0$.
Hence, we have 
\[
t=\varphi(s)=\frac{a_3s^3+a_1s+a_0}{s-b_0}.
\]
Scaling $s$ so that the only unramified lifting of $t=\infty$ is either $s=0$ or $s=1$, we have 
\[
t=\varphi(s)=\frac{a_3s^3+a_1s+a_0}{s-1}\quad \text{or} \quad \frac{a_3s^3+a_1s+a_0}{s}.
\]
Finally, we require that there are at most two liftings above each of the two roots $\frac{-9\pm3\sqrt{-3}}{2}$ of $t^2+9t+27=0$. Since $\varphi(s)\in\bQ(s)$, we only need to check this property for one of the two roots. Hence, each of the two possible $\varphi(s)$'s should have multiple roots. 

We deduce that there are two subcases of this case:
\begin{enumerate}
\item $t=\varphi(s)=\frac{a_3s^3+a_1s+a_0}{s}$ and the equation $\frac{a_3s^3+a_1s+a_0}{s}=-\frac{9+3\sqrt{-3}}{2}$ has multiple roots. 
\item $t=\varphi(s)=\frac{a_3s^3+a_1s+a_0}{s-1}$, and the equation $\frac{a_3s^3+a_1s+a_0}{s-1}=-\frac{9+3\sqrt{-3}}{2}$ has multiple roots.
\end{enumerate}
In subcase $(i)$, we must have $a_1=-6$ and $a_3=-\frac{4}{a_0^2}$. Thus, if we let $b=a_0$,  then 
\[
t=\varphi(s)=\frac{-\frac{4}{b^2}s^3-6s+b}{s}
\]
and the resulting Weierstrass equation is 
\[
\mathcal{X}_b\: y^2+\br{\frac{-\frac{4}{b^2}s^3-6s+b}{s}+3}xy+ y=x^3,
\]
which can be rewritten as 
\[
\mathcal{X}_b\: y^2+\br{-\frac{4}{b^2}s^3-3s+b}xy+s^3y=x^3.
\]

Although $b\in\Q^\times$ appears as a parameter, it can be scaled away. Indeed, substituting 
\[
s= bu,\qquad x= b^2X,\qquad y= b^3Y
\]
gives rise to a $\Q$-isomorphism between $\mathcal{X}_b$ and the representative
\[
\mathcal{X}_1\: Y^2+(-4u^3-3u+1)XY+u^3Y=X^3.
\]
Thus the covers $\mathcal X_b\to X_0$ are all $\Q$-isomorphic, and subcase~$(i)$ consists of a single $\Q$-isomorphism class.

In subcase $(ii)$, one can check that the coefficients $a_0$, $a_1$, and $a_3$ satisfy the equations of a certain curve, which is unfortunately not as explicit as in subcase $(i)$.

\subsubsection{Case $(4)$}\label{case(4)_family} 
For this case, we rewrite the Weierstrass equation of $X_0$ as 
\[
y^2+(1+3v)xy+v^3y=x^3
\]
with $t=1/v$. By an analogous argument to case $(3)$, we can reduce to the following two cases:

\begin{enumerate}
\item  $v=\varphi(s)=\frac{a_3s^3+a_1s+a_0}{s}$ and $\frac{a_3s^3+a_1s+a_0}{s}=-\frac{3+\sqrt{-3}}{18}$ has multiple roots.
\item $v=\varphi(s)=\frac{a_3s^3+a_1s+a_0}{s-1}$ and $\frac{a_3s^3+a_1s+a_0}{s-1}=-\frac{3+\sqrt{-3}}{18}$ has multiple roots. 
\end{enumerate}

In subcase $(i)$, we must have $a_1=-\frac{2}{9}$ and $a_0^2a_3=-\frac{2^2}{3^9}$. Writing $b=a_0$, we have
\[
v=\varphi(s)=\frac{-\frac{2^2}{3^9b^2}s^3-\frac{2}{9}s+b}{s},
\]
and the corresponding elliptic surface will be 
\[
\X_b\: y^2+(59049b^3 + 6561b^2s - 12s^3)xy+(19683b^3 - 4374b^2s - 4s^3)^3y=x^3.
\]
Again, $b\in\Q^\times$ can be scaled away. Substituting
\[
s= bu,\qquad x=b^6X,\qquad y= b^9Y
\]
transforms $\mathcal X_b$ into
\[
Y^2+(59049+6561u-12u^3)XY +(19683-4374u-4u^3)^3Y=X^3.
\]
Thus subcase~$(i)$ also consists of a single $\Q$-isomorphism class.

As with case $3(ii)$, for subcase $(ii)$, one can verify that the coefficients satisfy the defining equations of a certain curve, which is not as explicit as in subcase $4(i)$. 

Before ending this section, we write the next lemma for later use. 

\begin{lemma}\label{lem:rho-l-odd}
    Let $X$ be a degree $3$ cover of $X_0$ as in \Cref{Prop: dim_5_eg}. Let $c$ be a complex conjugation. Then for all primes $\l$, $\Tr\rho_\l(X)(c) = 1$ or $-1$, and after twisting by a finite order character, we may assume that $\Tr\rho_\l(X)(c) = 1$.
\end{lemma}

\begin{proof}
    Define $M\coloneqq H^2(X, \bQ)(1)/\Triv(X_{\overline{\Q}})$, which induces the five-dimensional representation $\rho_\ell(X)$ after base changing to $\bQ_\ell$. Notice that $M$ is a rational Hodge structure. And by our construction, $M\otimes \C$ contains the subspaces $H^{1,-1}$ and $H^{-1,1}$ of $H^2(X, \C)(1)$, which are both of dimension two because the genus of $X$ is $2$. Hence $\dim ((M\otimes \C)\cap H^{0,0})=1$. Then the result of this lemma follows by the fact that the involution $c$ interchanges $H^{1,-1}$ with $H^{-1,1}$. Finally, since $\rho_\l(X)$ is automatically $\operatorname{O}(5)$-valued, we can twist by a finite order character to ensure that  $\Tr\rho_\l(X)(c) = 1$.
\end{proof}

\section{An algorithm to verify irreducibility}\label{Sect: alg-for-irreducibility}

The aim of this section is to develop an algorithm for verifying the irreducibility of compatible systems of specific 5-dimensional Galois representations, with the primary examples being those constructed in \Cref{Sect: concrete_eg_5dim}. More precisely, we assume that 
\[(\rho_\l\:G_\Q\to\GL_5(\Qlb))_\l\]
is a $\Q$-rational, weakly compatible system of Galois representations, such that:
\begin{enumerate}
    \item For all primes $\l$, $\rho_\l$ is self-dual with trivial determinant. Moreover, the image of $\rho_\ell$ respects a nondegenerate symmetric bilinear form. In particular, $\rho_\l$ is isomorphic to a representation valued in $\SO_5(\Qlb)$;
    \item $(\rho_\l)_\l$ is pure of weight $0$;
    \item For all primes $\l$, $\rho_\l|_{G_{\Ql}}$ has Hodge--Tate weights $\{-1,-1,0,1,1\}$;
    \item For all primes $\l$ and for any complex conjugation $c$, $\Tr\rho_\l(c) = 1$;
    \item $(\rho_\l)_\l$ satisfies the hypotheses of \Cref{thm:hui}. In particular, by \Cref{rem: theorem conditions}, both \Cref{thm:decomp} and all the results of \Cref{sec:compatible} apply to $(\rho_\l)_\l$. 
\end{enumerate}

Under these assumptions, for a given prime $\l$, $\rho_{\l}$ can only be in cases $1, 2(i), 2(ii), 3a, 5(i)$, or $5(ii)$ of \Cref{table-cases}: we have $a=b=1$, and the remaining cases are ruled out by the assumption that $\rho_\l$ is pure of weight $0$.

\subsection{Criteria for irreducibility}

In this section, for each of the possible cases of \Cref{table-cases}, we find checkable criteria to rule out the possibility that $\rho_\l$ falls into this case for infinitely many $\l$. In the next subsection, we will use these criteria to formulate an algorithm to verify the irreducibility of $(\rho_\l)_\l$.

\subsubsection{Case $3a$ of \Cref{table-cases}}

Fix a prime $\l$, and suppose that $\rho_\l$ is in case $3a$ of \Cref{table-cases}. Then, by \Cref{prop:rl-irred} and \Cref{rem:asai}, there is a finite extension $K/\Q$ and a two-dimensional representation $\sigma$ of $G_K$ such that $\rho_\l$ contains the irreducible, two-dimensional representation $\rho':=(\Ind_K^\Q\det\sigma)\tensor(\simil r_\l)\ii$. Moreover, $\rho'$ has Hodge--Tate weights $\{-1,1\}$. It follows that $K/\Q$ is imaginary quadratic and that $\rho'$ is odd.

Since $\rho_\l$ is self-dual, so is $\rho'$, so 
\[\rho' \simeq(\rho')\dual\simeq\rho'\tensor(\det\rho')\ii.\]

It follows that $\det\rho'$ is a quadratic character, which is necessarily non-trivial, since $\rho'$ is odd.
Hence, if $p$ is a prime for which $\det\rho'(\Frob_p) = -1$, then we have $\Tr\rho'(\Frob_p) = 0$.

We deduce the following proposition:

\begin{proposition}\label{prop:3a}
     Suppose that for some prime $\l$, $\rho_\l$ is in case $3a$ of \Cref{table-cases}. Let $p$ be a prime for which $\det\rho'(\Frob_p) = -1$. Then $-1$ is an eigenvalue of $\rho_\l(\Frob_p)$, with multiplicity at least $2$.
\end{proposition}

\begin{proof}
    Let $a_p,b_p$ be the eigenvalues of $\rho'(\Frob_p)$. Then $a_p b_p = \det\rho'(\Frob_p) = -1$ and $a_p+b_p = \Tr\rho'(\Frob_p) = 0$. It follows that $\{a_p, b_p\} = \{-1, 1\}$, so $-1$ is an eigenvalue of $\rho_\l(\Frob_p)$. Now, since $\rho_\l$ is self-dual and valued in $\SO_5(\Qlb)$, the eigenvalues of $\rho_\l(\Frob_p)$ are of the form  $\{\alpha_p, \beta_p,1, \alpha_p\ii, \beta_p\ii\}$, and it immediately follows that the multiplicity of $-1$ must be at least $2$.
\end{proof}

Recall that since $(\rho_\l)_\l$ is a compatible system, there is a finite set of primes $S$ and monic degree $5$ polynomials $Q_p(X)$ for all $p\notin S$, such that for all $p\notin S\cup\{\l\}$, $\rho_\l$ is unramified at $p$ and the characteristic polynomial of $\rho_\l(\Frob_p)$ is $Q_p(X)$ .

\begin{corollary}\label{cor:3a}
    Suppose that for some prime $\l$, $\rho_\l$ is in case $3a$ of \Cref{table-cases}. Then there exists an imaginary quadratic extension $K/\Q$ that is unramified outside $S$ such that $(X+1)^2 \mid Q_p(X)$ for all primes $p\notin S$ that are inert in $K$.
\end{corollary}

\begin{proof}
    By assumption, $\rho_\l$ contains the two-dimensional induced representation $\rho'$. Moreover, since $\rho_\l$ is unramified outside $S\cup \{\l\}$ so is $\det\rho'$. Now, if $\l\notin S$, then $\rho_\l$ is crystalline at $\l$, and therefore $\det\rho'$ is crystalline at $\l$ with Hodge--Tate weight $0$, in which case, it follows that $\det\rho'$ is unramified at $\l$. Hence, $\det\rho'$ is unramified outside $S$.

    Let $K$ be the imaginary quadratic extension cut out by the kernel of $\det\rho'$. Then $K$ is unramified outside $S$, and for $p\notin S$, $\det\rho'(\Frob_p) = -1$ if and only if $p$ is inert in $K$. The result follows from \Cref{prop:3a}.
\end{proof}

\subsubsection{Cases $2(ii)$ and $5(ii)$ of \Cref{table-cases}}

Fix a prime $\l$ and suppose that $\rho_\l$ is in one of cases $2(ii)$, $5(ii)$ of \Cref{table-cases}. Then $\rho_\l$ contains a  unique one-dimensional non-trivial representation $\chi$. Since $\rho_\l$ is self-dual, $\chi$ is a quadratic character. We deduce the following proposition:

\begin{proposition}\label{prop:2ii5ii}
    Suppose that for some prime $\l$, $\rho_\l$ is in case $2(ii)$ or $5(ii)$ of \Cref{table-cases}. Then there exists a quadratic extension $K/\Q$ that is unramified outside $S$ such that $(X+1)^2 \mid Q_p(X)$ for all primes $p\notin S$ that are inert in $K$.
\end{proposition}

\begin{proof}
    As in the proof of \Cref{cor:3a}, $\chi$ is a quadratic character that is unramified outside $S$. Let $K$ be the fixed field of its kernel. Then $\chi(\Frob_p) = -1$ whenever $p$ is inert in $K$. As in the proof of \Cref{prop:3a}, if $-1$ is an eigenvalue of $\rho_\l(\Frob_p)$, then it occurs with multiplicity at least $2$. Hence, if $p$ is inert in $K$, $(X+1)^2 \mid Q_p(X)$. 
\end{proof}

\begin{proposition}\label{prop:2ii-real}
    Suppose that for some prime $\l$, $\rho_\l$ is in case $2(ii)$, and suppose that $r_\l$ is induced from a representation of $G_K$. Then $K$ is real quadratic.
\end{proposition}

\begin{proof}
    By the self twist argument in the proof of \Cref{lem:2ii} and by \Cref{thm:decomp}, for all sufficiently large $\l$, we have
    \[\rho_{\l}\simeq \sigma_{4,\l}\+\chi_{K/\Q}\]
    and
    \[r_\l\simeq\Ind_K^\Q\sigma_{2,\l},\]
    where $(\sigma_{4, \l}\:G_\Q\to \GL_4(\Ql))_\l$ is a compatible system, irreducible for sufficiently large $\l$, and $\sigma_{2,\l}$ is irreducible and two-dimensional, but not necessarily compatible. Moreover, $r_\l$ has Hodge--Tate weights $\{0,1,1,2\}$ and $\det\sigma_{2,\l} = \simil r_\l|_{G_K}$. It follows that $\det\sigma_{2,\l}$ has Hodge--Tate weight $2$ at both embeddings of $K$, and therefore that $\sigma_{2,\l}$ has Hodge--Tate weights $\{0,2\}$ and $\{1,1\}$ at its two embeddings.
    
    Now, suppose for contradiction that $K$ is imaginary. Then $\sigma_{2,\l}$ must be Lie irreducible, since if it was induced from a character of weight $w$, then its Hodge--Tate weights would have to be of the form $\{h_1, h_2\}$ and $\{w-h_1,w-h_2\}$ for some integers $h_1, h_2$. It follows that $\sigma_{3,\l}:=\Sym^2\sigma_{2,\l}\tensor(\det\sigma_{2,\l})\ii$ is irreducible and $\SO_3$-valued, with Hodge--Tate weights $\{-2,0,2\}$ and $\{0,0,0\}$.

    Now, we have
    \[\wedge^2\sigma_{4,\l}|_{G_K}\simeq\sigma_{3,\l}\+\sigma_{3,\l}^\tau,\]
    where $\tau\in G_{\Q}\setminus G_K$. The Hodge--Tate weights of $\sigma_{3,\l}$ ensure that it is not isomorphic to $\sigma_{3,\l}^\tau$, and therefore
    \[\wedge^2\sigma_{4,\l}\simeq\Ind_K^\Q\sigma_{3,\l}\]
    is irreducible.

    Fix a sufficiently large prime $\l_0$. By the Chebotarev density theorem, there exists a good prime $p\ne\l_0$ that splits in $K$ such that $a_p = \Tr\sigma_{3, \l_0}(\Frob_p)$ and $b_p = \Tr\sigma_{3, \l_0}^{\tau}(\Frob_p)$ are distinct. Moreover, by the compatibility of $\wedge^2\sigma_{4,\l}$, the unordered pair $\{a_p,b_p\}$ does not depend on $\l_0$. Since both representations are $\SO_3$-valued, the Frobenius polynomial of $\wedge^2\sigma_{4,\l}$ at $p$ is 
\[
(X-1)^2
\bigl(X^2-(a_p-1)X+1\bigr)
\bigl(X^2-(b_p-1)X+1\bigr).
\]
It follows that $a_p+b_p, a_pb_p\iQ$. Moreover, since $\rho_\l$ is pure, the roots of the above polynomial have absolute value $1$ and occur in reciprocal pairs. Thus $a_p$ and $b_p$ are real. It follows that the field $F = \Q(a_p, b_p)$ is either $\Q$ or a real quadratic field. 

Since $K$ is imaginary, we have $K\cap F = \Q$.
Now fix a sufficiently large prime $\l\ne p$ that is inert in $K$ but splits in $F$. Then $a_p,b_p$ are distinct elements of $\Ql$. It follows that the decomposition of $(\wedge^2\sigma_{4,\l})|_{G_K}$ into irreducibles already occurs over $\Ql$: otherwise their Frobenius traces would be conjugate over $\Ql$ for all primes, but this isn't the case for $a_p$ and $b_p$. Thus $\sigma_{3,\l}$ is isomorphic to a representation that takes values in $\GL_3(\Ql)$.

Let $v$ be the unique prime of $K$ above $\l$. By Frobenius reciprocity, we have
\[\br{\wedge^2\sigma_{4,\l}\tensor_{\Ql}\C_\l(-2)}^{G_{\Ql}}\simeq\br{\sigma_{3,\l}\tensor_{\Ql}\C_\l(-2)}^{G_{K_v}}.\]
Since $\wedge^2\sigma_{4,\l}$ has Hodge--Tate weights $\{-2,0,0,0,0,2\}$, the left-hand side has $\Ql$-dimension $1$. But the right hand side is a $K_v$-vector space, so its $\Ql$-dimension must be even. This contradiction shows that $K$ must be real quadratic.
\end{proof}

\begin{proposition}\label{prop:5ii-real}
        Suppose that for infinitely many primes $\l$, $\rho_\l$ is in case $5(ii)$, and suppose that $r_\l$ is induced from a representation of $G_K$. Then $K$ is real quadratic.
\end{proposition}

\begin{proof}
    If $\rho_\l\simeq \sigma_2\+\sigma_2'\+\chi_{K/\Q}$, then $\sigma_2, \sigma_2'$ are irreducible and two-dimensional with Hodge--Tate weights $\{-1,1\}$. Hence, by \Cref{prop:odd}, if $\l$ is large enough, both $\sigma_2, \sigma_2'$ are odd. Since the eigenvalues of $\rho_\l(c)$ are $\{1,1,1,-1,-1\}$ by assumption, it follows that $\chi_{K/\Q}(c) = 1$. 
\end{proof}

\begin{corollary}\label{cor:2ii5ii}
    Suppose that for infinitely many primes $\l$, $\rho_\l$ is in case $2(ii)$ or $5(ii)$ of \Cref{table-cases}. Then there exists a real quadratic extension $K/\Q$ that is unramified outside $S$ such that $(X+1)^2 \mid Q_p(X)$ for all primes $p\notin S$ that are inert in $K$.
\end{corollary}

\subsubsection{Cases $2(i)$ and $5(i)$ of \Cref{table-cases}}

Fix a prime $\l$ and suppose that $\rho_\l$ is in case $2(i)$ or $5(i)$ of \Cref{table-cases}. Recall that there is a four-dimensional representation $r_\l\:G_\Q\to \Gf(\Qlb)$ such that $\rho_\l\simeq\std(r_\l)$, and that in these cases, $r_\l$ decomposes as a direct sum of two-dimensional representations $\sigma_1\+\sigma_2$, where $\sigma_1$ is irreducible with Hodge--Tate weights $\{-1,1\}$, $\sigma_2$ has Hodge--Tate weights $\{0,0\}$, and $\det\sigma_1 = \det\sigma_2$. As in the proofs of \Cref{lem:2+2+1-trivial,lem:4+1-trivial}, if $\l$ is sufficiently large, $\sigma_2$ has finite image. In particular, for all primes $p$ at which $\rho_\l$ is unramified, the eigenvalues of $\sigma_2(\Frob_p)$ are roots of unity. We deduce the following proposition:

\begin{proposition}\label{prop:2i5i}
    Suppose that for infinitely many primes $\l$, $\rho_\l$ is in case $2(i)$ or $5(i)$ of \Cref{table-cases}. For each prime $p\notin S$, let  $\alpha_p, \beta_p,1, \alpha_p\ii, \beta_p\ii$ be the roots of $Q_p(X)$. Then at least one of $\alpha_p\beta_p$ or $\alpha_p\beta_p\ii$ is a root of unity.
\end{proposition}

\begin{proof}
    Let $\{\alpha_p, \beta_p,1, \alpha_p\ii, \beta_p\ii\}$ be the eigenvalues of $\rho_\l(\Frob_p)$ and let $\{a_p, b_p, d_pa_p\ii, d_pb_p\ii\}$ be the eigenvalues of $r_\l(\Frob_p)$, where $d_p = \simil r_\l(\Frob_p)$. By the above discussion, two of these must be roots of unity. Since $r_\l$ has Hodge--Tate weights $\{-1,0,0,1\}$, $\simil r_\l$ is Artin, and hence a finite order character. It follows that $d_p$ is a root of unity. Hence, at least one of $a_p, b_p$ is a root of unity.
    
    Moreover, from the isomorphism $\wedge^2(r_\l)\tensor(\simil r_\l)\ii\simeq \rho_\l\+\chi_\triv$, we have 
    \[\{\alpha_p, \beta_p, \alpha_p\ii, \beta_p\ii\} = \set{\frac{a_pb_p}{d_p}, \frac{a_p}{b_p}, \frac{d_p}{a_pb_p},\frac{b_p}{a_p}}.\]
    It follows that
    \[\{\alpha_p\beta_p, \alpha_p\beta_p\ii, \alpha_p\ii\beta_p, \alpha_p\ii\beta_p\ii\} = \set{\frac{a_p^2}{d_p}, \frac{b_p^2}{d_p}, \frac{d_p}{a_p^2}, \frac{d_p}{b_p^2}}\]
    contains two roots of unity, so at least one of $\alpha_p\beta_p$ or $\alpha_p\beta_p\ii$ is a root of unity.
\end{proof}

\subsection{The algorithm}\label{Sect: algorithm}

Let $(\rho_\l)_\l$ be a compatible system satisfying the hypotheses at the beginning of the section. In particular, let $S$ be a set of primes such that if $p\notin S\cup\{\l\}$, $\rho_\l$ is unramified at $p$, and $\rho_\l(\Frob_p)$ has characteristic polynomial $Q_p(X)$.

We give three algorithms, which terminate if and only if $\rho_\l$ is not in certain cases of \Cref{table-cases}. We will prove the validity of these algorithms in the next section.

\begin{algorithm-r}\label{algorithm-2i5i}
    The following algorithm terminates if and only if, for all but finitely many primes $\l$, $\rho_\l$ is not in case $2(i)$ or $5(i)$ of \Cref{table-cases}.
    \newline
\begin{algorithm}[H]
\SetAlgoLined
    \For{each prime $p\notin S$}{
    Compute the roots $\{\alpha_p, \beta_p, 1, \alpha_p\ii, \beta_p\ii\}$ of $Q_p(X)$\;
    \If{neither of $\alpha_p\beta_p$ and $\alpha_p\beta_p\ii$ is a root of unity}{
    terminate\;}
    }
\end{algorithm}
\end{algorithm-r}

\begin{algorithm-r}\label{algorithm-2ii5ii}
    The following algorithm terminates if and only if, for all but finitely many primes $\l$, $\rho_\l$ is not in case $2(ii)$ or $5(ii)$ of \Cref{table-cases}.
    \newline
\begin{algorithm}[H]
    Compute the set $\Sigma^+$ of real quadratic extensions unramified outside $S$\;
    \For{each field $K\in\Sigma^+$}{
        \For{each prime $p\notin S$ that is inert in $K$}{
            Evaluate $Q_p(-1)$\;
            \If{$Q_p(-1) \ne 0$}{Mark $K$ as eliminated and continue to the next field;}
        }    
    }
\end{algorithm}
\end{algorithm-r}

\begin{algorithm-r}\label{algorithm-3a}
    The following algorithm terminates only if for all but finitely many primes $\l$, $\rho_\l$ is not in case $3a$ of \Cref{table-cases}. If \Cref{algorithm-2i5i} terminates, then this algorithm terminates if and only if, for all but finitely many primes $\l$, $\rho_\l$ is not in case $3a$ of \Cref{table-cases}. 
    \newline
\begin{algorithm}[H]
    Compute the set $\Sigma^-$ of imaginary quadratic extensions unramified outside $S$\;
    \For{each field $K\in\Sigma^-$}{
        \For{each prime $p\notin S$ that is inert in $K$}{
            Evaluate $Q_p(-1)$\;
            \If{$Q_p(-1) \ne 0$}{Mark $K$ as eliminated and continue to the next field;}
        }    
    }
\end{algorithm}
\end{algorithm-r}

\begin{remark}
    Suppose that we wish to verify that one of these algorithms \emph{does not} terminate. By using an explicit version of the Chebotarev density theorem, one could find an integer $N$ such that if the algorithm terminates, then it terminates on some prime $p\notin S$ with $p\le N$. Hence, to check that the algorithm doesn't terminate, it would be sufficient just to apply it to all primes $p\le N$. However, in practice, this integer $N$ would be prohibitively large. 
    On the other hand, if the algorithms do terminate, then the smallest prime $p$ which terminates the algorithm will typically be quite small. Hence, these algorithms are efficient at verifying irreducibility, but inefficient at verifying reducibility.
\end{remark}

\subsection{Validity of the algorithm}

In this section, we prove the validity of \Cref{algorithm-2i5i,,algorithm-2ii5ii,algorithm-3a}.

\begin{proposition}\label{lem:no-trivial}
    \Cref{algorithm-2i5i} terminates if and only if, for all but finitely many primes $\l$, $\rho_\l$ is neither in case $2(i)$ nor case $5(i)$ of \Cref{table-cases}.
\end{proposition}

\begin{proof}
    For each prime $p\notin S$, let $\{\alpha_p, \beta_p, 1, \alpha_p\ii, \beta_p\ii\}$ denote the roots of $Q_p(X)$ in $\Qb$.

    The only if direction is immediate from \Cref{prop:2i5i}, since \Cref{algorithm-2i5i} terminates if and only if for some prime $p\notin S$, neither $\alpha_p\beta_p$ nor $\alpha_p\beta_p\ii$ is a root of unity.

    For the if direction, suppose that for all primes $p\notin S$, one of $\alpha_p\beta_p$ or $\alpha_p\beta_p\ii$ is a root of unity. By our assumptions, $\rho_\l$ is one of cases $1, 2(i), 2(ii), 3a, 5(i)$ and $5(ii)$ of \Cref{table-cases}, and it is in one of cases $2(i)$ or $5(i)$ if and only if $r_\l$ is reducible. Hence, it remains to show that $r_\l$ is reducible.
    
    For each prime $\l$, let $r_\l\:G_\Q\to\Gf(\Qlb)$ be the representation such that $\rho_\l = \std r_\l$, normalised so that for $p\notin S$, $\simil r_\l(\Frob_p) = d_p$ is a root of unity. Let $\{a_p, b_p, d_pa_p\ii, d_pb_p\ii\}$ be the eigenvalues of $r_\l(\Frob_p)$.
    
    As in the proof of \Cref{prop:2i5i}, we have
    \[\{\alpha_p\beta_p, \alpha_p\beta_p\ii, \alpha_p\ii\beta_p, \alpha_p\ii\beta_p\ii\} = \set{\frac{a_p^2}{d_p}, \frac{b_p^2}{d_p}, \frac{d_p}{a_p^2}, \frac{d_p}{b_p^2}}.\]
    Without loss of generality, we may assume that $\frac{a_p^2}{d_p}$ is a root of unity. Since $d_p$ is a root of unity, it follows that $a_p$ is too.

    Now, by the assumption that $\rho_\l$ is $\Q$-rational, $\alpha_p$ and $\beta_p$ are roots of a degree $5$ polynomial in $\Q[X]$, so $[\Q(\alpha_p, \beta_p):\Q] \le |S_5|=120$. Hence, if one of $\alpha_p\beta_p$ and $\alpha_p\beta_p\ii$ is a root of unity, then it must be a root of unity contained in a degree $120$ extension of $\Q$. Moreover, since $\simil r_\l$ is a finite order character, it takes values in $L^\times$ for some number field $L$. Thus, for all $p$, $d_p$ is a root of unity contained in $L$.
    It follows that there is an integer $N$ such that for all primes $p$, at least one of $a_p, b_p$ is an $N$-th root of unity.

    Let $G_\l$ be the Zariski closure of the image of $r_\l$ and let $G_\l^\circ$ be its identity connected component. By assumption, $G_\l$ is contained in the Zariski closed subset of $\Gf$ consisting of the elements $g\in \Gf$ that have an $N$-th root of unity as an eigenvalue.
    In particular, we cannot have $G_\l^\circ = \Sp_4$. Due to the shape of the Hodge--Tate weights of $r_\l$ the only irreducible subgroup of $\Gf$ that $G_\l^\circ$ can be is $\Sp_4$. It follows that $r_\l$ cannot be Lie irreducible.

    By \cite{patrikis-variations}*{Prop.~3.4.1}, we can write $r_\l\simeq \Ind_{K}^\Q(\sigma\tensor\omega)$, where $\sigma$ is a Lie irreducible representation of $G_K$ and $\omega$ is an Artin representation of $G_K$. Moreover, we cannot have $K=\Q$: the Hodge--Tate weights of $r_\l$ are $\{-1,0,0,1\}$, whereas the Hodge--Tate weights of $\sigma\tensor\omega$ with $\omega$ Artin and degree at least $2$ would be of the form $\{a,a,b,b\}$ for some integers $a,b$.
    It follows that $r_\l$ is an induced representation.
    
    Write $r_\l=\Ind_K^\Q\sigma$ for some representation $\sigma$. Since $r_\l$ is not Artin, neither is $\sigma$, so we can assume that $\sigma$ is Lie irreducible. If a prime $p\notin S$ splits completely in $K$, then the eigenvalues of $r_\l(\Frob_p)$ are exactly the eigenvalues of $\sigma(\Frob_\p)$ for the primes $\p\mid p$. 
    
    If $\sigma$ is one-dimensional, then it is an infinite order character, so there must exist primes $p\notin S$ such that for all $\p\mid p$, $\sigma(\Frob_\p)$ has infinite order, so is not a root of unity. 

    Now suppose that $\sigma$ is a two-dimensional Lie irreducible representation, let $\G^\circ$ be the identity connected component of the Zariski closure of its image. Let $L/K$ be the finite extension such that the Zariski closure of the image of $\sigma|_L$ is $\G^\circ$. Since $\sigma$ is Lie irreducible, we must have $\G^\circ = \SL_2$ or $\GL_2$. The same argument as above shows that the subset of elements $g\in \G^\circ$ that do not have an $N$-th root of unity as an eigenvalue is Zariski open, and hence dense. Hence, by the Chebotarev density theorem, the set of primes $\p$ of $L$ such that $\sigma(\Frob_\p)$ has an $N$-th root of unity as an eigenvalue has Dirichlet density $0$. It follows that for $100\%$ of primes $p$ that split in $L$, the eigenvalues of $\sigma(\Frob_\p)$ are not roots of unity for all $\p\mid p$. For such $p$, the eigenvalues of $r_\l(\Frob_p)$ are not roots of unity either.

    In all cases, it follows that there are infinitely many primes $p$ such that none of the eigenvalues of $r_\l(\Frob_p)$ are roots of unity, contradicting our assumption.
\end{proof}

\begin{proposition}\label{lem:no-finite-order}
    \Cref{algorithm-2ii5ii} terminates if and only if, for all but finitely many primes $\l$, $\rho_\l$ is neither in case $2(ii)$ nor case $5(ii)$ of \Cref{table-cases}.  
\end{proposition}

\begin{proof}
    The only if direction is immediate from \Cref{cor:2ii5ii}, since \Cref{algorithm-2ii5ii} terminates if and only if, for every real quadratic field $K$ that is unramified outside $S$, there is an inert prime $p\notin S$ such that $Q_p(-1)\ne 0$.

    For the if direction, assume that \Cref{algorithm-2ii5ii} does not terminate, so that there is a real quadratic field $K$, unramified outside $S$, such that every inert prime $p\notin S$ satisfies $Q_p(-1)=0$. Let $r_\l\:G_\Q\to\Gf(\Qlb)$ be the representation such that $\rho_\l\simeq\std(r_\l)$. 

    For each prime $p\notin S$, let $\{\alpha_p, \beta_p,1, \alpha_p\ii, \beta_p\ii\}$ be the eigenvalues of $\rho_\l(\Frob_p)$, let $d_p = \simil r_\l(\Frob_p)$, and let $\{a_p, b_p, d_pa_p\ii, d_pb_p\ii\}$ be the eigenvalues of $r_\l(\Frob_p)$. From the isomorphism $\wedge^2(r_\l)\tensor(\simil r_\l)\ii\simeq \rho_\l\+\chi_\triv$, we have 
    \[\{\alpha_p, \beta_p, \alpha_p\ii, \beta_p\ii\} = \set{\frac{a_pb_p}{d_p}, \frac{a_p}{b_p}, \frac{d_p}{a_pb_p},\frac{b_p}{a_p}}.\]
    Now assume that $p$ is inert in $K$. Since $Q_p(-1) = 0$, we may assume without loss of generality that $\alpha_p =\alpha_p\ii= -1$. It follows that either $a_p = -d_pb_p\ii$ or $a_p = -b_p$. In particular, $\Tr r_\l(\Frob_p) = 0$. Hence, if $\chi_{K/\Q}$ is the quadratic character associated to $K$, we have
    \[\Tr r_\l(\Frob_p) =\Tr r_\l\tensor\chi_{K/\Q}(\Frob_p)\]
    for all primes $p\notin S$. It follows that $r_\l\simeq r_\l\tensor\chi_{K/\Q}$.

    Now, if $r_\l$ is irreducible, it follows that it is induced from a representation $\sigma$ of $G_K$, with $K$ a real quadratic field. Thus $\rho_\l$ is in one of cases $2(ii), 5(ii)$, or $3a$ of \Cref{table-cases}. But in case $3a$, by \Cref{rem:asai}, $\rho_\l$ contains the two-dimensional irreducible representation $(\mathop{\mathrm{Ind}}\nolimits _{K}^{\mathbb Q}\det \sigma ) \otimes (\simil r_{\ell})^{-1}$, which in our case, must have Hodge--Tate weights $\{-1,1\}$. These weights can only occur if $K$ is quadratic imaginary. It follows that $\rho_\l$ is in one of cases $2(ii), 5(ii)$ of \Cref{table-cases}.

    Finally, if $r_\l$ is reducible, then since $\rho_\l$ is pure, we must have $r_\l\simeq\sigma_1\+\sigma_2$, where $\sigma_1$ is irreducible and has Hodge--Tate weights $\{-1,1\}$ and $\sigma_2$ has Hodge--Tate weights $\{0,0\}$. Since $r_\l\simeq r_\l\tensor\chi_{K/\Q}$, we have $\sigma_1\simeq\sigma_1\tensor\chi_{K/\Q}$, whence $\sigma_1$ is induced from a character of $G_K$. But this is impossible, since $K$ is real, but $\sigma_1$ is Hodge--Tate regular.
\end{proof}

\begin{proposition}\label{lem:3a-validity}
    \Cref{algorithm-3a} terminates only if for all but finitely many primes $\l$, $\rho_\l$ is not in case $3a$ of \Cref{table-cases}.  Moreover, if \Cref{algorithm-2i5i} terminates, then \Cref{algorithm-3a} terminates if and only if for all but finitely many primes $\l$, $\rho_\l$ is not in case $3a$ of \Cref{table-cases}.
\end{proposition}

\begin{proof}
    The only if direction is immediate from \Cref{cor:3a}, since \Cref{algorithm-3a} terminates if and only if, for every imaginary quadratic field $K$ that is unramified outside $S$, there is an inert prime $p\notin S$ such that $Q_p(-1)\ne 0$.

    For the if direction, assume that \Cref{algorithm-3a} does not terminate, so that there is an imaginary quadratic field $K$, unramified outside $S$, such that every inert prime $p\notin S$ satisfies $Q_p(-1)=0$. Let $r_\l\:G_\Q\to\Gf(\Qlb)$ be the representation such that $\rho_\l\simeq\std(r_\l)$. For each prime $p\notin S$, let $d_p = \simil r_\l(\Frob_p)$ and let $\{a_p, b_p, d_pa_p\ii, d_pb_p\ii\}$ be the eigenvalues of $r_\l(\Frob_p)$. The same argument as in \Cref{lem:no-finite-order} shows that for all primes $p\notin S$ that are inert in $K$, we have $a_p = -d_pb_p\ii$ or $a_p = -b_p$, and moreover that $r_\l\simeq r_\l\tensor\chi_{K/\Q}$.

    Since \Cref{algorithm-2i5i} terminates, by \Cref{lem:no-trivial}, $r_\l$ is irreducible. It follows that it is induced from a representation $\sigma$ of $G_K$, with $K$ an imaginary quadratic field. Thus $\rho_\l$ is in one of cases $2(ii), 5(ii)$, or $3a$ of \Cref{table-cases}. But $\rho_\l$ cannot be in case $2(ii)$ or $5(ii)$: if $\l$ is large enough, in these cases, by \Cref{prop:2ii-real,prop:5ii-real}, $K$ must be real. Hence, $\rho_\l$ is in case $3a$.
\end{proof}

\begin{remark}
    The proof of \Cref{lem:3a-validity} shows that if \Cref{algorithm-3a} does not terminate, then for all but finitely many primes, $\rho_\l$ is either in case $2(i)$, $3a$ or $5(i)$ of \Cref{table-cases}. If, moreover, \Cref{algorithm-2i5i} terminates, then $\rho_\l$ is in case $3a$ for all but finitely many $\l$. On the other hand, if neither \Cref{algorithm-3a} nor \Cref{algorithm-2i5i} terminates, then one can show that for all primes $\l$, $\rho_\l \simeq \sigma_1\+\sigma_2\+\chi_\triv$, where $\sigma_1$, $\sigma_2$ are irreducible two-dimensional representations, that are both induced from some imaginary quadratic extension $K/\Q$. In particular, $\rho_\l$ is in case $5(i)$ of \Cref{table-cases} for all primes $\l$.
\end{remark}
By comparing these propositions with \Cref{table-cases} and \Cref{rem:lie-irred}, we deduce the following corollary:

\begin{corollary}\label{cor:validity}
\mbox{}
\begin{enumerate}
    \item \Cref{algorithm-2i5i} terminates if and only if, for all but finitely many primes $\l$, $\rho_\l$ does not contain the trivial representation.
    \item \Cref{algorithm-2i5i,algorithm-2ii5ii} both terminate if and only if, for all but finitely many primes $\l$, $\rho_\l|_K$ does not contain the trivial representation for any finite extension $K/\Q$.
        \item \Cref{algorithm-2i5i,,algorithm-2ii5ii,algorithm-3a} all terminate if and only if $\rho_\l$ is Lie irreducible for all but finitely many primes $\l$. 
\end{enumerate}
\end{corollary}

\begin{remark}\label{rem:mon-s05}
    We note that if $\rho_\l$ is Lie irreducible, then its algebraic monodromy group is $\SO_5$. Indeed, by \Cref{table:semisimple}, the algebraic group is either $\SO_5$ or $\Sym^4(\SL_2)$, but the latter necessarily has Hodge--Tate weights in an arithmetic progression.
\end{remark}

\section{The irreducibility of families of compatible systems}\label{Sect: spread_out_irr}

In the last section, we introduced an algorithm to verify the irreducibility of a compatible system of Galois representations. In this section, we will demonstrate how, when combined with the main result of \cite{Cadoret-Tamagawa-open-img-I} by Cadoret and Tamagawa, this algorithm can be extended to prove irreducibility across a family of compatible systems of Galois representations.

For the convenience of the reader, we briefly recall the main result of \cite{Cadoret-Tamagawa-open-img-I}.
Let $k$ be a field and let $Y$ be a geometrically connected finite type $k$-scheme. Recall that $\pi^{\rm{\acute{e}t}}_1(Y)$ is the \'etale fundamental group of $Y$ (we omit the base point for simplicity). There is a fundamental short exact sequence
\begin{equation}\label{Eqn: fundamental_short_ext_seq_pi1}
    1\rightarrow \pi^{\rm{\acute{e}t}}_1(Y_{\overline{k}})\rightarrow \pi^{\rm{\acute{e}t}}_{1}(Y) \rightarrow G_{k}\rightarrow 1.
\end{equation}
An $\ell$-adic representation $\rho_{\ell}\:\pi^{\rm{\acute{e}t}}_{1}(Y) \rightarrow \GL_{m}(\Q_{\ell})$ is said to be \emph{Lie perfect} (LP for short) if the Lie algebra of $G:=\rho_{\ell}(\pi^{\rm{\acute{e}t}}_1(Y)) \subset \GL_{m}(\Q_{\ell})$ is perfect, and \emph{geometrically Lie perfect} (GLP for short) if the Lie algebra of $G^{\mathrm{geo}}:=\rho_\ell(\pi^{\rm{\acute{e}t}}_1(Y_{\overline{k}}
)) \subset G$ is perfect. 

Each $k$-rational point $x\in Y(k)$ induces a splitting $x \: G_k \rightarrow \pi^{\rm{\acute{e}t}}_1(Y)$ of the fundamental short
exact sequence \eqref{Eqn: fundamental_short_ext_seq_pi1}. Set $G_x := \rho_\ell \circ x(G_k)$ for the
corresponding closed subgroup of $G$.

\begin{theorem}[\cite{Cadoret-Tamagawa-open-img-I}*{Thm.~1.1}]\label{thm: CT}
    Let $k/\Q$ be a finitely generated field extension and $Y$ be a smooth geometrically connected scheme over $k$ of dimension 1. Suppose that $\rho_{\ell}\: \pi^{\rm{\acute{e}t}}_1(Y)\rightarrow \GL_m(\Z_{\ell})$ is GLP. Then the set $Y_{\rho_{\ell}} := \{x\in Y(k) : G_x \text{ is not open in }G\}$ is finite, and there exists an integer $B_{\rho_{\ell}}\geq 1$ such that $[G: G_x]\leq B_{\rho_{\ell}}$ for every $x\in Y(k)\setminus Y_{\rho_{\ell}}$.
\end{theorem}

One of the most important examples of GLP representations is as follows. Let $B$ be a curve over $k$. Let $f\: \mathcal{X}\rightarrow B$ be a smooth proper morphism.  For $b\in B(k)$, let $\overline{b}$ be a geometric point over $b$ (recall that we fixed an embedding from $k$ to $\overline{k}$).
By smooth base change, $R^{i}f_{\ast}\overline{\Q}_{\ell}(j)$ is a lisse sheaf. It induces a monodromy representation
\[
\rho_{\ell}\: \pi^{\rm{\acute{e}t}}_{1}(B)\rightarrow \GL(H^{i}_{\mathrm{\acute{e}t}}(\mathcal{X}_{\overline{b}}, \overline{\Q}_{\ell}(j))),
\]
where $\mathcal{X}_{\overline{b}}$ is the fibre of $f$ at $\overline{b}$. By \cite{Cadoret-Tamagawa-open-img-I}*{Thm.~5.8}, $\rho_{\ell}$ is GLP. Note
that any subquotient of an LP representation (respectively, of a GLP representation) is again LP (respectively GLP).

Returning to our discussion, suppose that $\mathcal{F}_{\ell}$ is a lisse $\overline{\Q}_{\ell}$-sheaf. We say $\mathcal{F}_{\ell}$ is irreducible if the associated monodromy representation $\rho_{\ell}$ is irreducible. Suppose that $b\in B(k)$. We say $\mathcal{F}_{\ell, \overline{b}}$ is irreducible if the associated Galois representation $\rho_{\ell}\circ b$ is irreducible.
\begin{theorem}\label{Prop: spread_out}
    Let $\mathcal{F}_{\ell}$ be a lisse $\overline{\Q}_{\ell}$-sheaf over a curve. Assume that:
    \begin{enumerate}
        \item $\mathcal{F}_{\ell}$ is rank five; 
        \item the monodromy representation $\rho_{\ell}$ associated to $\mathcal{F}_{\ell}$ is GLP;
        \item for each closed point $b\in B(k)$, $\mathcal{F}_{\ell, \overline{b}}$ is a self-dual Galois representation;
        \item the Hodge--Tate weights of the Galois representation associated to $\mathcal{F}_{\ell, \overline{b}}$ are $\{-1,-1, 0, 1,1\}$.
        \item there exists a point $b_0\in B(k)$, such that the Galois representation associated to $\mathcal{F}_{\ell, \overline{b}_0}$ is irreducible. 
        
    \end{enumerate}
    Then $\mathcal{F}_{\ell, \overline{b}}$ is irreducible for all but finitely many points $b\in B(k)$.
\end{theorem}
\begin{proof}
    Let $H$ be the $\ell$-adic algebraic monodromy group of (the Galois representation associated to) $\mathcal{F}_{\ell, b_0}$. Since $\mathcal{F}_{\ell, b_0}$ is irreducible and rank five, by \Cref{rem:lie-irred}, it is Lie irreducible. Hence, by \Cref{table:semisimple}, $H^\circ$ can be $\SL_5$, $\SL_2$, or $\SO_5$. The self-dual condition rules out the case $\SL_5$, and the condition on the Hodge--Tate weights rules out the $\SL_2$ case. So $H^\circ=\SO_5$. Since $\mathcal{F}_{\ell}$ is GLP, by \Cref{thm: CT}, the connected component of the $\ell$-adic algebraic monodromy group of $\mathcal{F}_{\ell, \overline{x}}$ is also $\SO_5$ for all but finitely many points $x\in B(k)$. So $\mathcal{F}_{\ell, \overline{b}}$ is irreducible for all but finitely many points $b\in B(k)$. 
\end{proof}

Let $\mathcal{X}\rightarrow B$ be a family of elliptic surfaces. Recall from \Cref{Def: transcendental_lisse} that for each prime $\ell$, $\operatorname{NTriv}_{\ell}(\mathcal{X})$ is a lisse sheaf. Let $\rho_{\ell}$ be the monodromy representation associated to $\operatorname{NTriv}_{\ell}(\mathcal{X})$.

\begin{corollary}\label{Cor: spread_out}
Suppose that $\operatorname{NTriv}_{\ell}(\mathcal{X})$ satisfies the assumptions $(i)-(iv)$ of \Cref{Prop: spread_out} for each prime $\ell$. Assume that there exists a rational point $b_0\in B(\Q)$ and a prime $\ell$ such that
\begin{enumerate}
    \item the compatible system $(\rho_{\l, b_0})_\l$ satisfies the hypotheses at the beginning of \Cref{Sect: alg-for-irreducibility}, and
    \item \Cref{algorithm-2i5i,,algorithm-2ii5ii,algorithm-3a} terminate on $\rho_{\ell, b_{0}}$. 
\end{enumerate}
Then there is a finite set of primes $S$ such that, for every
$\ell\notin S$, the representation $\rho_{\ell,b}$ is irreducible for all
but finitely many $b\in B(\mathbb Q)$.
\end{corollary}
\begin{proof}
    Recall that by \Cref{cor:validity}, we know that $\rho_{\ell, b_{0}}$ is Lie irreducible for all but finitely many $\ell$. Take $S$ to be the finite set consisting of all the exceptional primes. Using the splitness of the short exact sequence~\eqref{Eqn: fundamental_short_ext_seq_pi1} we find that $\rho_{\ell, b_{0}}$ is irreducible, i.e.\ condition $(v)$ of \Cref{Prop: spread_out} is fulfilled as well. Then the corollary follows immediately from \Cref{Prop: spread_out}.
\end{proof}

\section{The Tate conjecture for covers of $X_0$}\label{Sect: verify_Tate_conj_concrete_family}

In this section, we prove \Cref{thm:Tate-intro}. Recall that $X_0$ is the No.~63 elliptic surface of \cite{Shioda-Schutt-MW-Lattice}*{Table~8.3} induced by the Weierstrass equation
\begin{equation*}
    X_0\: y^2 + (t+3)xy + y = x^3,
\end{equation*}
and let $X$ be a genus $2$ elliptic surface in the set $\SS$ defined in \eqref{Eqn: Cartesian_diagram}. Assume first that $X$ is not one of the two isolated classes. By the classification in \Cref{Sect: concrete_eg_5dim}, $X$ belongs to one of the four one-dimensional families. Let $\X\to B$ be the family containing $X$. Then, by \Cref{Prop: lisse_sheaf}, the non-trivial part (\Cref{Def: transcendental_lisse}) $\operatorname{NTriv}_{\ell}(\mathcal{X})$ is a rank $5$ lisse sheaf, and its associated monodromy representation is GLP. Hence, $\operatorname{NTriv}_{\ell}(\mathcal{X})$ satisfies conditions $(i)$ and $(ii)$ of \Cref{Prop: spread_out}.

For every $b\in B(\Q)$, \Cref{Prop: dim_5_eg} shows that the non-trivial part (\Cref{Defn: non-trivial_part}) of $\mathcal{X}_b$ induces a five-dimensional self-dual Galois representation with Hodge--Tate weights $\{-1,-1,0,1,1\}$. Hence, $\operatorname{NTriv}_{\ell}(\mathcal{X})$ also satisfies conditions $(iii)$ and $(iv)$ of \Cref{Prop: spread_out}. Finally, by \Cref{prop:etale-compatibility} and \Cref{lem:rho-l-odd}, the compatible system attached to each specialisation $\mathcal{X}_b$ satisfies the hypotheses required to apply \Cref{algorithm-2i5i,algorithm-2ii5ii,algorithm-3a}.

Hence, by \Cref{cor:validity,Cor: spread_out}, with \Cref{cor:validity} applied directly to the two isolated classes, \Cref{thm:Tate-intro} follows from the following theorem:

\begin{theorem}\label{thm:application}
\begin{enumerate}
    \item For each of the four one-dimensional families in
    \Cref{Sect: concrete_eg_5dim}, there is a specialisation for which
    \Cref{algorithm-2i5i,,algorithm-2ii5ii,algorithm-3a} all terminate.
    \item For each of the two isolated $\Q$-isomorphism classes arising
    in Cases $(3)(i)$ and $(4)(i)$,
    \Cref{algorithm-2i5i,algorithm-2ii5ii} both terminate.
\end{enumerate}
\end{theorem}

\begin{remark}
    For each of the four one-dimensional families, the argument also shows that the connected monodromy group is $\SO_5$.
\end{remark}

\subsection{Prime factors of the conductor}\label{Sect: ramified_primes}

Fix an elliptic surface $X$ over $\Q$ and let 
\[(\varphi_\l\:G_\Q\to\GL(H^2_{\rm{\acute{e}t}}(X_{\overline{\Q}}, \Ql(1))^{\mathrm{ss}}))_\l\]
denote the compatible system of Galois representations induced by the second \'etale cohomology of $X$. Let $(\rho_\l(X))_\l$ be the compatible system of Galois representations arising from the non-trivial part of $X$. 

In order to apply \Cref{algorithm-2i5i,,algorithm-2ii5ii,algorithm-3a}, we first need to compute a set of primes $S$ outside of which $(\rho_\l(X))_\l$ is unramified.

We first note that since $\rho_\l(X)$ is a subrepresentation of $\varphi_\l$, if $\varphi_\l$ is unramified at a prime $p$, then so is $\rho_\l(X)$. On the other hand, for the compatible system $(\varphi_\l)_\l$, we can take $S$ to be the set of primes at which $X$ has bad reduction. Thus, it is sufficient to compute this set of primes.

Our method is essentially a consequence of the classical algorithm of Tate \cite[IV, \S9]{Silverman2}, and should be known or even used by experts (for instance, see~\cite[\S5.1]{GT-self-dual}). We take as input the Weierstrass equation \eqref{Eqn: Weierstrass_model} of the generic fibre $E$ of $X$, and produce a finite set $S$ containing all the primes at which $X$ has bad reduction.

Let $\cE$ be the projective closure of the surface defined by a minimal Weierstrass model of $X$. For each prime $p$, let $\overline{X}_p$ and $\overline{\cE}_p$ denote the corresponding reductions. Since $X$ is the minimal resolution of the ADE singularities of $\cE$ and since the blowing-ups are all defined over $\bZ$ \cite[Example~7.12.1]{GTM52}, any singular point of $\overline{X}_p$ must lie above a singular point of $\overline{\cE}_p$. If $\widetilde{\cE}_p$ is a minimal resolution of $\overline{\cE}_p$, then we have the following diagram
\[
\begin{tikzcd}
    X \arrow[d, "\tau"'] & \overline{X}_p \arrow[d, "\bar{\tau}_p"'] & \widetilde{\cE}_p \arrow[ld, "\widetilde{\tau}_p"]\\
    \cE & \overline{\cE}_p
\end{tikzcd}
\]
where $\tau\: X\to \cE$ is the blowing-up map, $\bar{\tau}_p$ is the $\mod p$ reduction of $\tau$ (which may not be a blowing-up map), and $\widetilde{\tau}_p$ is the blowing-up of $\overline{\cE}_p$.
If $\overline{X}_p$ and $\widetilde{\cE}_p$ coincide, then $\overline{X}_p$ is again smooth, and hence $p$ is not a ramified prime. Thus, our arguments are reduced to a comparison between the $\mod p$ reductions of $\tau$ and $\widetilde{\tau}_p$. 

For our application, we need to assume the following extra conditions:
\begin{enumerate}
    \item every singular fibre lies over an algebraic integral point in $\overline{\bQ}$;
    \item the infinite fibre of $X$ is smooth. 
\end{enumerate}
Condition $(i)$ is equivalent to saying that $\Delta (E)$ is a non-zero
constant multiple of a monic polynomial with integer coefficients.
There are only finitely many primes $p$ at which the Weierstrass equation
is not integral or the leading coefficient of $\Delta (E)$ is not a
$p$-adic unit. We include these primes in $S$ and restrict to the
remaining primes throughout the rest of this subsection. Then
$\deg \Delta (E)=\deg \Delta (\overline{E})$, so the singularities of
$\overline{X}_{p}$ can only occur above the reductions of the roots of
$\Delta (E)$.

\begin{proposition}\label{Prop: not_a_singular_pt}
    Keep the above setup and assumptions. Let $\bar{t}_0\in \overline{\bF}_p$ be a root of $\Delta(\overline{E})$. Suppose that:
    \begin{enumerate}
        \item the characteristic $p$ is odd;
        \item the fibre $\overline{\cE}_{p, \bar{t}_0}$ has a nodal singularity;
        \item $\bar{t}_0$ lifts to a unique root $t_0$ of $\Delta(E)$ without counting multiplicity;
        \item ${\rm ord}_{t_0}(\Delta(E))={\rm ord}_{\bar{t}_0}(\Delta(\overline{E}))$.
    \end{enumerate}
    Then $\overline{X}_p$ has no singular point over $\bar{t}_0$. 
\end{proposition}
\begin{proof}
    This is essentially a consequence of Tate's algorithm (for instance, see~\cite[\S5.8.1]{Shioda-Schutt-MW-Lattice}). We write some details for the readers' convenience. Notice that conditions $(ii)$-$(iv)$ guarantee that both $X_{t_0}$ and $\widetilde{\cE}_{p, \bar{t}_0}$ are of the same type ${\rm I}_n$, with $n={\rm ord}_{t_0}(\Delta(E))$. Thus, it suffices to show that the local equations of the irreducible components of $X_{t_0}$ match those of $\widetilde{\cE}_{p, \bar{t}_0}$. 
    
    For this, we are allowed to take a finite extension of the ground field and assume that $t_0=0$ (hence $\bar{t}_0=0$ as well) and the only curve singular point of $\overline{\cE}_{p, \bar{t}_0}$ is at $(x,y)=(0,0)$. Now we can assume that the minimal Weierstrass model of $E$ at $t_0=0$ is of the form 
    \begin{equation*}
        y^2=x^3+a_2'(t)x^2+a_4'(t)t^{m+1}x+a_6'(t)t^n
    \end{equation*}
    with each $a_i'(t)$ a polynomial over $\bZ$, and $m=\lfloor \frac{n}{2} \rfloor$ is the largest integer not exceeding $n/2$. In particular $a_2'(0)\neq 0$ since the fibre is of multiplicative type. Then one can check that at the $j$th blow-up with $1\leq j\leq m-1$ by setting $x=t^jx'$ and $y=t^jy'$, this process will introduce two exceptional divisors with local equation 
        \[
        y'^2=a_2'(0)x'^2.
        \]
        Finally, at the $m$–th blow-up, the local equation of the exceptional divisor(s) is 
        \[
        y'^2=\begin{cases}a_{2}'(0)x^{\prime \,2}+a_{6}'(0)&n=2m\\a_{2}'(0)x^{\prime \,2}&n=2m+1,\end{cases}
        \]
        which will result in two components if $n=2m+1$ and one component if $n=2m$. Since the above arguments are characteristic free when $p>2$, the reductions of the local equations defining the irreducible components of $X_{t_0}$ correspond bijectively to the irreducible components of $\widetilde{\cE}_{p, \bar{t}_0}$. Thus the reduction of the former matches the latter. So there is no singular point of $\overline{X}_p$ lying above $\bar{t}_0$. 
\end{proof}

\begin{corollary}\label{Cor: possible_singular_pts}
    Under the conditions $(i)$ and $(ii)$ above, $p$ is a ramified prime only in one of the following cases:
    \begin{enumerate}
        \item $p=2$;
        \item there is a root $\bar{t}_0$ of $\Delta(\overline{E})$ whose corresponding fibre $\widetilde{\cE}_{p, \bar{t}_0}$ is of additive type;
        \item there is a root $\bar{t}_0$ of $\Delta(\overline{E})$ such that ${\rm ord}_{t_0}(\Delta(E))<{\rm ord}_{\bar{t}_0}(\Delta(\overline{E}))$.
    \end{enumerate}
\end{corollary}
\begin{proof}
    If none of the three cases happen, then one can check that all the conditions of \Cref{Prop: not_a_singular_pt} will be fulfilled.
\end{proof}

\subsection{Applying the algorithms}\label{Sect: calculation}

We now apply  \Cref{algorithm-2i5i,,algorithm-2ii5ii,algorithm-3a} to a representative of each of the four one-dimensional families constructed in \Cref{Sect: concrete_eg_5dim}, and the first two algorithms to representatives of the two isolated classes. We will see that all the relevant algorithms terminate, and \Cref{thm:application} follows.

Since the discussion for every situation is similar, we only write down the details for the first one in our main text, and leave the rest of the calculation to \Cref{Appendix: other_cases}.

\subsubsection{Case $(1)$}\label{case(1)}

In this case the resulting family is (parametrised by $c$)
\[
\X_c\: y^2+(1+3cs^2(s-1))xy+c^3s^6(s-1)^3y=x^3.
\]
    If we pick $c=1$, then the candidate surface is 
    \[
    X_1\: y^2+(1+3s^2(s-1))xy+s^6(s-1)^3y=x^3.
    \]

\begin{lemma}
    The surface $X_1$ has good reduction outside $3, 7$.
\end{lemma}

\begin{proof}
    If $p>3$, then we can write the Weierstrass form of $X_1$ as 
    \[
        y^2=x^3+A(s)x+B(s),
    \]
    with $A(s), B(s)\in\Z[s]$. Computing the resultant of $A$ and $B$, we have 
    \[
    \mathop{\mathrm{Res}}_{s}(A, B) = -2^{48}3^{117}.
    \]
    Hence, $A$ and $B$ have no common zero in characteristic $p>3$. Thus, by \Cref{Cor: possible_singular_pts}$(ii)$ and
\cite[Table~5.1]{Shioda-Schutt-MW-Lattice}, there are no fibres of additive
type at primes $>3$.

    Up to a constant, the discriminant of $X_1$ is 
    \[
    \Delta(X_1)=s^{18}(s-1)^9(3s^2 - 3s + 1)(9s^4 - 9s^3 - 3s^2 + 3s + 1) = 27s^{33} + \cdots.
    \]
    We have
    \[\disc(3s^2 - 3s + 1) = -3\]
    and
    \[\disc(9s^4 - 9s^3 - 3s^2 + 3s + 1)=\pm 3^7\cdot 7.\]
    Hence, by \Cref{Cor: possible_singular_pts}$(iii)$, the only other prime $>3$ of possible bad reduction is $p=7$.

    Finally, one can check by hand that $X_1$ has good reduction at $p=2$.
\end{proof}
    
Now, let $(\rho_\l)_\l$ denote the compatible system arising from the non-trivial part (\Cref{Defn: non-trivial_part}). In the following table, we list the polynomials $Q_p(x)$, the characteristic polynomials of $\rho_\l(\Frob_p)$, for the primes $p\le 13$.

\begin{table}[h!]
\centering
    \begin{tabular}{|c|c|} \hline 
  prime &  characteristic polynomial \\ \hline 
    2 &  $x^5 - \frac{1}{2}x^4 + \frac{3}{4}x^3 - \frac{3}{4}x^2 + \frac{1}{2}x - 1$  \\ \hline 
    3 & Bad prime   \\ \hline 
    5 & $x^5 + \frac{4}{5}x^4 + \frac{22}{25}x^3 - \frac{22}{25}x^2 - \frac{4}{5}x - 1$  \\ \hline 
    7 & Bad prime\\ \hline 
    11 & $x^5 + \frac{14}{11}x^4 + \frac{42}{121}x^3 - \frac{42}{121}x^2 - \frac{14}{11}x - 1$  \\ \hline
    13 & $x^5 - \frac{6}{13}x^4 - \frac{4}{169}x^3 + \frac{4}{169}x^2 + \frac{6}{13}x - 1$  \\ \hline
\end{tabular}
\caption{Characteristic polynomials of $X_1$}
\label{Tab: case 1}
\end{table}

Now we apply \Cref{algorithm-2i5i,,algorithm-2ii5ii,algorithm-3a}.

\begin{enumerate}
    \item \Cref{algorithm-2i5i} terminates at $p=5$.
    \item The set of real quadratic extensions that are unramified outside $\{3,7\}$ is $\Sigma^+=\{\bQ(\sqrt{21})\}$. \Cref{algorithm-2ii5ii} terminates, since $p=11$ is inert in $\Q(\sqrt{21})$, but $Q_{11}(-1)\ne 0$.
    \item The set of imaginary quadratic extensions that are unramified outside $\{3,7\}$ is $\Sigma^-=\{\bQ(\sqrt{-3}),\bQ(\sqrt{-7})\}$. For both extensions, $p=5$ is inert and $Q_5(-1)\ne 0$, so \Cref{algorithm-3a} terminates.
\end{enumerate}

Since all three algorithms terminate, we deduce \Cref{thm:application} for case $(1)$ of \Cref{Tab: ramification_type}. The corresponding calculations for the other three one-dimensional families and the two isolated classes are given in \Cref{Appendix: other_cases}. Hence \Cref{thm:application} follows.

\appendix
\section{The proof of \Cref{thm:application} in the remaining cases}\label{Appendix: other_cases}
In this appendix, we give the details of the calculation in the other cases of \Cref{Sect: calculation}.

\subsubsection{Case $(2)$}\label{case(2)}
In this case, recall that the resulting family is 
\[
\X_c\: y^2+(cs^2(s-1)+3)xy+y=x^3. 
\]
When $c=1$, we have 
\[
X_1\: y^2+(s^2(s-1)+3)xy+y=x^3. 
\]
Analysing the bad primes as in case $(1)$, we get $S=\{3,61,307\}$. In \Cref{Tab: case 2}, we list the characteristic polynomials of $X_1$ for some small primes.

\begin{table}[h!]
    \centering
    \begin{tabular}{|c|c|}
    \hline
       prime  & characteristic polynomial \\  \hline 
       2 & $x^5 - \frac{1}{2}x^4 + \frac{3}{4}x^3 - \frac{3}{4}x^2 + \frac{1}{2}x - 1$  \\ \hline
       3     &   bad prime\\ \hline
       5 & $x^5 + \frac{4}{5}x^4 - \frac{2}{25}x^3 +\frac{2}{25}x^2 - \frac{4}{5}x - 1$  \\ \hline
       7 & $x^5 - \frac{22}{49}x^3 + \frac{22}{49}x^2 - 1$  \\ \hline
       11 & $x^5 - x^4 + \frac{58}{121}x^3 - \frac{58}{121}x^2 + x - 1$  \\ \hline
       13 & $x^5-\frac{6}{13}x^4 - \frac{4}{169}x^3 + \frac{4}{169}x^2 + \frac{6}{13}x - 1$  \\ \hline
    \end{tabular}
    \caption{Characteristic polynomials of $X_1$ for small primes}
    \label{Tab: case 2}
\end{table}

Now we apply \Cref{algorithm-2i5i,,algorithm-2ii5ii,algorithm-3a}:  

\begin{enumerate}
\item \Cref{algorithm-2i5i} terminates at $p=5$.
\item The set of real quadratic extensions which are unramified outside $\{3,61,307\}$ is 
\[\Sigma^+=\{\bQ(\sqrt{61}), \bQ(\sqrt{3\cdot 307}), \bQ(\sqrt{3\cdot 61\cdot 307})\}.\] 
\begin{itemize}
    \item For $\bQ(\sqrt{61})$ and $\bQ(\sqrt{3\cdot 61\cdot 307})$, $p=2$ is inert, and $Q_2(-1)\neq 0$;
    \item For $\bQ(\sqrt{3\cdot 307})$, $p=11$ is inert, and $Q_{11}(-1)\neq 0$;
\end{itemize}
Hence, \Cref{algorithm-2ii5ii} terminates.
\item The set of imaginary quadratic extensions which are unramified outside $\{3,61,307\}$ is 
\[\Sigma^-=\{\bQ(\sqrt{-3}), \bQ(\sqrt{-307}), \bQ(\sqrt{-3\cdot 61}), \bQ(\sqrt{-61\cdot 307})\}.\] 
\begin{itemize}
    \item For $\bQ(\sqrt{-3})$ and $\bQ(\sqrt{-307})$, $p=2$ is inert, and $Q_2(-1)\neq 0$;
    \item For $\bQ(\sqrt{-3\cdot 61})$ and $\bQ(\sqrt{-61\cdot 307})$, $p=5$ is inert, and $Q_5(-1)\neq 0$;
\end{itemize}
Hence, \Cref{algorithm-3a} terminates.
\end{enumerate}
Since all three algorithms terminate, we deduce \Cref{thm:application} for case $(2)$ of \Cref{Tab: ramification_type}.

\subsubsection{Case $(3)$}\label{Case 3}
There are two subcases under this situation. The first subcase consists of the single $\Q$-isomorphism class, represented by 
\[
X_1\:y^2+(-4s^3-3s+1)xy+s^3y=x^3.
\]
Consequently, it suffices to verify the Tate conjecture for $X_1$. By \Cref{cor:validity}, in this case, the termination of \Cref{algorithm-2i5i,,algorithm-2ii5ii} alone is sufficient to establish the Tate conjecture for $X_1$. In particular, \Cref{algorithm-3a} is not needed.

To apply these two algorithms, notice that $S=\{2,3\}$. In \Cref{Tab: case 3(2)}, we list the first few characteristic polynomials. 
\begin{table}[h!]
    \centering
    \begin{tabular}{|c|c|}
    \hline
       Prime  & Characteristic polynomial  \\ \hline 
       2 &   bad prime \\ \hline
       3       & bad prime\\ \hline
       5 & $x^5 + \frac{4}{5}x^4 - \frac{1}{5}x^3 + \frac{1}{5}x^2 -  \frac{4}{5}x - 1$  \\ \hline
       7 & $x^5 + \frac{4}{7}x^4 - \frac{5}{49}x^3 + \frac{5}{49}x^2 - \frac{4}{7}x - 1$  \\ \hline
       11 & $x^5 - \frac{2}{11}x^4 - \frac{13}{11}x^3 +  \frac{13}{11}x^2 +\frac{2}{11}x - 1$  \\ \hline
       13 & $x^5 + \frac{19}{13}x^4 +\frac{142}{169}x^3 -\frac{142}{169}x^2 - \frac{19}{13}x - 1$  \\ \hline
    \end{tabular}
    \caption{Characteristic polynomials of $X_1$ for small primes}
    \label{Tab: case 3(2)}
\end{table}
Now we apply the algorithms: 
\begin{enumerate}
\item \Cref{algorithm-2i5i} terminates at $p=5$.
\item The set of real quadratic extensions which are unramified outside $\{2,3\}$ is 
\[\Sigma^+=\{\bQ(\sqrt{2}), \bQ(\sqrt{3}), \bQ(\sqrt{6})\}.\] 
\begin{itemize}[leftmargin=*]
    \item For $\bQ(\sqrt{2})$ and $\Q(\sqrt{6})$, $p=13$ is inert, and $Q_{13}(-1)\neq 0$;
    \item For $\bQ(\sqrt{3})$, $p=7$ is inert, and $Q_7(-1)\neq 0$;
\end{itemize}
Hence, \Cref{algorithm-2ii5ii} terminates.

\end{enumerate}
Since both algorithms terminate, we deduce \Cref{thm:application} for the first subcase of $(3)$ of \Cref{Tab: ramification_type}.

For the second subcase, we choose our representative elliptic surface $X$ to have coefficients $(a_0,a_1,a_3)=(\frac{9}{2}, -3, -\frac{16}{27})$, i.e.\  
\[
X\: y^2+\br{-\frac{16}{27}s^3+\frac{3}{2}}xy+(s-1)^3y=x^3.
\]
Analysing the bad primes as in case $(1)$, we get $S=\{2,3,7\}$. In \Cref{Tab: case 3(1)}, we list the first few characteristic polynomials. 

\begin{table}[h!]
    \centering
    \begin{tabular}{|c|c|}
    \hline
       Prime  & Characteristic polynomial \\ \hline 
       2 &   bad prime \\ \hline
       3     &   bad prime\\ \hline
       5 & $x^5 + \frac{1}{5}x^4 - \frac{2}{25}x^3 + \frac{2}{25}x^2 - \frac{1}{5}x -1$  \\ \hline
       7 & bad prime\\ \hline
       11 & $x^5 -\frac{2}{11}x^4 + \frac{19}{121}x^3 - \frac{19}{121}x^2 +  \frac{2}{11}x - 1$  \\ \hline
       13 & $x^5 + \frac{1}{13}x^4 - \frac{56}{169}x^3 + \frac{56}{169}x^2 - \frac{1}{13}x - 1$  \\ \hline
       17 & $x^5 - \frac{11}{17}x^4 - \frac{404}{289}x^3 + \frac{404}{289}x^2 + \frac{11}{17}x - 1$  \\ \hline 
    \end{tabular}
    \caption{Characteristic polynomials of $X$ for small primes}
    \label{Tab: case 3(1)}
\end{table}

Now we apply \Cref{algorithm-2i5i,,algorithm-2ii5ii,algorithm-3a}:  
\begin{enumerate}
\item \Cref{algorithm-2i5i} terminates at $p=5$.
\item The set of real quadratic extensions which are unramified outside $\{2,3,7\}$ is 
\[\Sigma^+=\{\bQ(\sqrt{2}), \bQ(\sqrt{3}), \bQ(\sqrt{7}), \bQ(\sqrt{2\cdot 3}), \bQ(\sqrt{ 2\cdot 7}),\bQ(\sqrt{ 3\cdot 7}), \bQ(\sqrt{2\cdot 3\cdot 7}) \}.\] 
\begin{itemize}[leftmargin=*]
    \item For $\bQ(\sqrt{2}), \bQ(\sqrt{3}), \bQ(\sqrt{7}), \bQ(\sqrt{2\cdot 3\cdot 7}) $, $p=5$ is inert, and $Q_5(-1)\neq 0$;
    \item For $\bQ(\sqrt{2\cdot 3}),\bQ(\sqrt{ 3\cdot 7})$, $p=13$ is inert, and $Q_{13}(-1)\neq 0$;
    \item For $\bQ(\sqrt{ 2\cdot 7})$, $p=17$ is inert, and $Q_{17}(-1)\neq 0$.
\end{itemize}
Hence, \Cref{algorithm-2ii5ii} terminates.
\item The set of imaginary quadratic extensions which are unramified outside $\{2,3,7\}$ is 
\begin{equation*}
\begin{split}
\Sigma^-=
    \set{
    \begin{aligned}
         & \bQ(\sqrt{-1}),\ \bQ(\sqrt{-2}),\ \bQ(\sqrt{-3}),\ \bQ(\sqrt{-7}), \bQ(\sqrt{-2\cdot 3}),\ \\
         & \bQ(\sqrt{- 2\cdot 7}),\ \bQ(\sqrt{ -3\cdot 7}),\ \bQ(\sqrt{-2\cdot 3\cdot 7})
    \end{aligned}
   }.
\end{split}
\end{equation*}
\begin{itemize}[leftmargin=*]
    \item For $\bQ(\sqrt{-2}), \bQ(\sqrt{-3}), \bQ(\sqrt{-7}), \bQ(\sqrt{-2\cdot 3\cdot 7}) $, $p=5$ is inert, and $Q_5(-1)\neq 0$;
    \item For $\bQ(\sqrt{-1})$, $p=11$ is inert and $Q_{11}(-1)\neq 0$;
    \item For $\bQ(\sqrt{-2\cdot 3}),\bQ(\sqrt{- 3\cdot 7})$, $p=13$ is inert, and $Q_{13}(-1)\neq 0$;
    \item For $\bQ(\sqrt{- 2\cdot 7})$, $p=17$ is inert, and $Q_{17}(-1)\neq 0$.
\end{itemize}
Hence, \Cref{algorithm-3a} terminates.
\end{enumerate}
Since all three algorithms terminate, we deduce \Cref{thm:application} for this case.

\subsubsection{Case $(4)$}\label{case(4)} 
We also have two subcases in this situation. The first subcase corresponds to the single $\Q$-isomorphism class
\[
X_1\: y^2+(-12s^3 + 6561s + 59049)xy+(-4s^3 - 4374s + 19683)^3y=x^3
\]
with $S=\{2,3\}$. 

In \Cref{Tab: case 4(2)}, we list the first few characteristic polynomials. 
\begin{table}[h!]
    \centering
    \begin{tabular}{|c|c|}
    \hline
       Prime  & Characteristic polynomial   \\ \hline 
       2 &   bad prime \\ \hline
       3       & bad prime\\ \hline
       5 & $x^5 + \frac{4}{5}x^4 - \frac{1}{5}x^3 + \frac{1}{5}x^2 - \frac{4}{5}x - 1$  \\ \hline
       7 & $x^5 + \frac{4}{7}x^4 - \frac{5}{49}x^3 + \frac{5}{49}x^2 - \frac{4}{7}x - 1$   \\ \hline
       11 & $x^5 - \frac{2}{11}x^4 - \frac{13}{11}x^3 + \frac{13}{11}x^2 + \frac{2}{11}x - 1$  \\ \hline
       13 & $x^5 + \frac{19}{13}x^4 + \frac{142}{169}x^3 - \frac{142}{169}x^2 - \frac{19}{13}x - 1$  \\ \hline
    \end{tabular}
    \caption{Characteristic polynomials of $X_1$ for small primes}
    \label{Tab: case 4(2)}
\end{table}
Now we apply \Cref{algorithm-2i5i,,algorithm-2ii5ii}: 
 \begin{enumerate}
\item \Cref{algorithm-2i5i} terminates at $p=5$.
\item The set of real quadratic extensions which are unramified outside $\{2,3\}$ is 
\[\Sigma^+=\{\bQ(\sqrt{2}), \bQ(\sqrt{3}), \bQ(\sqrt{6})\}.\] 
\begin{itemize}[leftmargin=*]
    \item For $\bQ(\sqrt{2})$, $p=13$ is inert, and $Q_{13}(-1)\neq 0$;
    \item For $\bQ(\sqrt{3})$, $p=7$ is inert, and $Q_7(-1)\neq 0$;
    \item For $\bQ(\sqrt{6})$, $p=13$ is inert, and $Q_{13}(-1)\neq 0$.
\end{itemize}
Hence, \Cref{algorithm-2ii5ii} terminates.
\end{enumerate}
Since both algorithms terminate, by \Cref{cor:validity}, we deduce \Cref{thm:application} for the first subcase of $(4)$ of \Cref{Tab: ramification_type}.

In the second subcase, we choose the candidate $X$ to correspond to $(a_0,a_1,a_3)=(\frac{1}{6}, -\frac{1}{9}, -\frac{16}{729})$. Then 
$
v=\varphi(s)=\frac{-\frac{16}{729}s^3-\frac{1}{9}s+\frac{1}{6}}{s-1}
$
and the elliptic surface is 
\[
X\: y^2+(-96s^3 + 972s - 729)xy+(-32s^3 - 162s + 243)^3y=x^3
\]
with $S=\{2,3,7\}$.  

In \Cref{Tab: case 4(1)}, we list the first few characteristic polynomials.
\begin{table}[h!]
    \centering
    \begin{tabular}{|c|c|}
    \hline
       Prime  & Characteristic polynomial  \\ \hline 
       2   & bad prime \\ \hline
       3      & bad prime\\ \hline
       5 & $x^5 + \frac{1}{5}x^4 - \frac{2}{25}x^3 + \frac{2}{25}x^2 - \frac{1}{5}x - 1$  \\ \hline
       7 &  bad prime \\ \hline
       11 & $x^5 - \frac{2}{11}x^4 + \frac{19}{121}x^3 - \frac{19}{121}x^2 + \frac{2}{11}x - 1$  \\ \hline
       13 & $x^5 + \frac{1}{13}x^4 - \frac{56}{169}x^3 + \frac{56}{169}x^2 - \frac{1}{13}x - 1$  \\ \hline
       17 & $x^5 - \frac{11}{17}x^4 - \frac{404}{289}x^3 + \frac{404}{289}x^2 + \frac{11}{17}x - 1$  \\ \hline
    \end{tabular}
    \caption{Characteristic polynomials of $X$ for small primes}
    \label{Tab: case 4(1)}
\end{table}
Now we apply \Cref{algorithm-2i5i,,algorithm-2ii5ii,algorithm-3a}:  
\begin{enumerate}
\item \Cref{algorithm-2i5i} terminates at $p=5$.
\item The set of real quadratic extensions which are unramified outside $\{2,3,7\}$ is 
\[\Sigma^+=\{\bQ(\sqrt{2}), \bQ(\sqrt{3}), \bQ(\sqrt{7}), \bQ(\sqrt{2\cdot 3}), \bQ(\sqrt{ 2\cdot 7}),\bQ(\sqrt{ 3\cdot 7}), \bQ(\sqrt{2\cdot 3\cdot 7}) \}.\] 
\begin{itemize}[leftmargin=*]
    \item For $\bQ(\sqrt{2}), \bQ(\sqrt{3}), \bQ(\sqrt{7}), \bQ(\sqrt{2\cdot 3\cdot 7}) $, $p=5$ is inert, and $Q_5(-1)\neq 0$;
    \item For $\bQ(\sqrt{2\cdot 3}),\bQ(\sqrt{ 3\cdot 7})$, $p=13$ is inert, and $Q_{13}(-1)\neq 0$;
    \item For $\bQ(\sqrt{ 2\cdot 7})$, $p=17$ is inert, and $Q_{17}(-1)\neq 0$.
\end{itemize}
Hence, \Cref{algorithm-2ii5ii} terminates.
\item The set of imaginary quadratic extensions which are unramified outside $\{2,3,7\}$ is 
\begin{equation*}
\begin{split}
\Sigma^-=
    \set{
    \begin{aligned}
         & \bQ(\sqrt{-1}),\ \bQ(\sqrt{-2}),\ \bQ(\sqrt{-3}),\ \bQ(\sqrt{-7}), \bQ(\sqrt{-2\cdot 3}),\ \\
         & \bQ(\sqrt{- 2\cdot 7}),\ \bQ(\sqrt{ -3\cdot 7}),\ \bQ(\sqrt{-2\cdot 3\cdot 7})
    \end{aligned}
   }.
\end{split}
\end{equation*}
\begin{itemize}[leftmargin=*]
    \item For $\bQ(\sqrt{-2}), \bQ(\sqrt{-3}), \bQ(\sqrt{-7}), \bQ(\sqrt{-2\cdot 3\cdot 7}) $, $p=5$ is inert, and $Q_5(-1)\neq 0$;
    \item For $\bQ(\sqrt{-1})$, $p=11$ is inert and $Q_{11}(-1)\neq 0$;
    \item For $\bQ(\sqrt{-2\cdot 3}),\bQ(\sqrt{- 3\cdot 7})$, $p=13$ is inert, and $Q_{13}(-1)\neq 0$;
    \item For $\bQ(\sqrt{- 2\cdot 7})$, $p=17$ is inert, and $Q_{17}(-1)\neq 0$.
\end{itemize}
Hence, \Cref{algorithm-3a} terminates.
\end{enumerate}

Hence, \Cref{thm:application} holds for case $(4)$ of \Cref{Tab: ramification_type}.

\section*{Acknowledgements}

The authors would like to thank Stefan Patrikis for numerous helpful conversations and for his assistance with \Cref{thm:lifting,psw-generalisation}. We thank Chun Yin Hui for several comments and corrections, in particular for pointing out an error in the proof of \Cref{prop:non-self-dual}, and that \Cref{lem:pure-ht-weight} holds without the assumption of weak compatibility. We thank the anonymous referee for their meticulous comments and corrections, which have greatly improved the accuracy of this paper. We are grateful to Noam Elkies for a helpful correspondence. The third author was partially supported by the Israel Science Foundation (grant No. 1963/20), by the US-Israel Binational Science Foundation (grant No. 2018250), and by an AMS--Simons travel grant.

\bibliography{bibliography}

\end{document}